\documentclass[12pt]{article}
\usepackage{amssymb}
\usepackage{amsthm}
\usepackage{amsmath}
\usepackage{amscd}
\usepackage{url}
\usepackage{array}
\usepackage{appendix}
\usepackage[latin2]{inputenc}
\usepackage{t1enc}
\usepackage{enumerate}
\usepackage[mathscr]{eucal}
\usepackage[british]{babel}
\usepackage[dvips]{graphicx}
\usepackage[dvips]{epsfig}
\usepackage{epsf}
\usepackage{indentfirst}

\newcommand{\Ker}{\mathrm{Ker}}

\newcommand{\Hom}{\mathrm{Hom}}

\newcommand{\End}{\mathrm{End}}
\newcommand{\GL}{\mathrm{GL}}

\newcommand{\kd}{\mathrm{Kdim}}
\newcommand{\rk}{\mathrm{rk}}

\newcommand{\Coker}{\mathrm{Coker}}
\newcommand{\gr}{\mathrm{gr}^{\cdot}}
\newcommand{\bg}{(\hspace{-0.06cm}(}
\newcommand{\jg}{)\hspace{-0.06cm})}

\newcommand{\bs}{[\hspace{-0.04cm}[}
\newcommand{\js}{]\hspace{-0.04cm}]}
\newtheorem{thm}{Theorem}[section]
\newtheorem{pro}[thm]{Proposition}
\newtheorem{lem}[thm]{Lemma}
\newtheorem{cor}[thm]{Corollary}

\newtheorem{que}{Question}

\theoremstyle{definition}
\newtheorem*{rem}{Remark}
\newtheorem*{rems}{Remarks}

\begin{document}
\title{Generalized Robba rings}
\author{Gergely Z\'abr\'adi\\ with an appendix by Peter Schneider}
\maketitle
\begin{abstract}
We prove that any projective coadmissible module over the locally
analytic distribution algebra of a compact $p$-adic Lie group is
finitely generated. In particular, the category of coadmissible
modules does not have enough projectives. In the Appendix a
``generalized Robba ring'' for uniform pro-$p$ groups is
constructed which naturally contains the locally analytic
distribution algebra as a subring. The construction uses the
theory of generalized microlocalization of quasi-abelian normed
algebras that is also developed there. We equip this generalized
Robba ring with a self-dual locally convex topology extending the
topology on the distribution algebra. This is used to show some 
results on coadmissible modules.
\end{abstract}

\begin{tableofcontents}

\end{tableofcontents}

\section{Introduction}

The theory of locally analytic representations of $p$-adic
analytic groups has recently been developed by Schneider and
Teitelbaum \cite{ST3,ST4,ST5,ST6,foundations,ST2}. The category
$\mathcal{C}_{D(G,K)}$ of coadmissible modules over the $K$-valued
locally analytic distribution algebra $D(G,K)$ plays a crucial
role in the theory. Here $\mathbb{Q}_p\subseteq
K\subseteq\mathbb{C}_p$ is a complete and discretely valued field
and $G$ a compact $p$-adic Lie group. A locally analytic
representation of $G$ over the field $K$ is called
\emph{admissible} if its strong dual is a coadmissible module over
$D(G,K)$. (Locally analytic representations of non-compact
$p$-adic Lie groups are called admissible if their restriction to
a---or equvialently, to any---compact open subgroup is
admissible.) The distribution algebra $D(G,K)$ is a
Fr\'echet-Stein algebra as it is the projective limit of the
noetherian Banach algebras $D_{r_l}(G,K)$ ($l\geq 1$) with flat
connecting maps. Here the algebras $D_{r_l}(G,K)$ are certain
completions of $D(G,K)$ with respect to norms $r_l$ (for details
see section \ref{z11}).

The finitely presented modules are always coadmissible, but the
coadmissible modules over $D(G,K)$ need not even be finitely
generated, and there exist finitely generated $D(G,K)$-modules
that are not coadmissible. However, in this paper we show that all
the \emph{projective} objects in $\mathcal{C}_{D(G,K)}$ are
finitely generated (and then, of course, also finitely presented)
generalizing earlier results in the commutative case by Gruson
\cite{G}. In particular, the category $\mathcal{C}_{D(G,K)}$ does
not have enough projective objects. As a by-product we also
compute the Grothendieck group $K_0$ of the finitely generated
projective $D_r(G,K)$-modules whenever $G$ is uniform.

However, the main aim of this paper is to investigate the properties of the
so-called ``generalized Robba ring'' that is constructed in the Appendix of
this paper. The ring $\mathcal{R}(G,K)$ (the generalized Robba ring) is
defined as follows. Using the theory of generalized microlocalization (see
the Appendix) one obtains ``closed noncommutative annuli'' for uniform
pro-$p$ groups, ie.\ rings $D_{[\rho,\rho']}(G,K)$ (with
$p^{-1}<\rho<\rho'<1$) whose elements are interpreted as noncommutative
Laurent series converging on a closed $p$-adic polyannulus $\rho\leq
|T_1|=\dots=|T_d|\leq\rho'$ (with $d$ being the dimension of the uniform
pro-$p$ group $G$). As it is shown in section 4, these are noetherian Banach
algebras (whenever $\rho$ and $\rho'$ are rational powers of $p$). One then
takes the projective limit with respect to the natural flat (cf.\ Thm.\ 4.3)
connecting maps $D_{[\rho,\rho_2]}(G,K)\hookrightarrow D_{[\rho,\rho_1]}(G,K)$
(for $p^{-1}<\rho<\rho_1<\rho_2<1$) to obtain the Fr\'echet-Stein algebras
$D_{[\rho,1)}(G,K)$. The generalized Robba ring is then the inductive 
limit of these Fr\'echet-Stein algebras when $\rho$ tends to $1$. The ring
$\mathcal{R}(G,K)$ admits a self-dual ring topology such that
$\Hom_{K}^{ct}(\mathcal{R}(G,K),K)\cong\mathcal{R}(G,K)$ as topological $\mathcal{R}(G,K)$-modules.  
The self-dual topology on $\mathcal{R}(G,K)$ is unfortunately not
equal to the inductive limit topology of the Fr\'echet topologies
on $ D_{[\rho,1)}(G,K)$ as the latter is not self-dual (see Corollary
  \ref{r1}) with respect to the pairing constructed in section 5. However,
the topology we construct (the ``nice'' topology) has the
following properties:

\begin{enumerate}[$(i)$]
\item The ring $\mathcal{R}(G,K)$ is self-dual in the nice
topology with respect to the pairing
\begin{eqnarray*}
\mathcal{R}(G,K)\times\mathcal{R}(G,K)&\rightarrow& K\\
(x,y)&\mapsto&``\hbox{the constant term of }xy."
\end{eqnarray*}
\item On the classical Robba ring $\mathcal{R}(\mathbb{Z}_p,K)$
the nice topology coincides with the inductive limit topology.
\item The subspace topology on the distribution algebra $D(G,K)$
equals the usual Fr\'echet topology.
\end{enumerate}

In the last section of this paper we use the Fr\'echet-Stein
structure of the algebras $ D_{[\rho,1)}(G,K)$ to construct an
exact functor $R$ from the category of coadmissible modules over
$D(G,K)$ to the category of modules over the generalized Robba
ring $\mathcal{R}(G,K)$. We call the image of this functor the
category of coadmissible modules over $\mathcal{R}(G,K)$. Our main
result (Theorem \ref{z39}) in this section states that whenever $M$ is a finitely
generated coadmissible module over $D(G,K)$ then
$\Hom_K^{ct}(R(M),K)$ is also coadmissible where the topology on $R(M)$ is a
canonical topology coming from the nice topology of $\mathcal{R}(G,K)$. By Proposition
\ref{z10} this includes all the projective objects in the
category.

\subsection*{Acknowledgements}

I gratefully acknowledge the hospitality of the Westf\"alische
Wilhelmsuniversit\"at M\"unster where the most of this paper was
written. During this research I was supported by the German
Israeli Foundation for Research and Development and by the
Deutsche Forschungsgemeinschaft (DFG). I am mostly indebted to
Peter Schneider for introducing me to the representation theory of
$p$-adic groups, for drawing my attention to Robba rings, for his
valuable comments, and for reading earlier versions of the
manuscript so carefully. I would also like to thank Tobias Schmidt
for reading parts of the manuscript and for various discussions
about the topic. I am grateful to the referee for his careful reading of the
manuscript, for his remarks on the presentation of results, and for
simplifying the proof of Thm.\ 5.5.

\section{Preliminaries and basic notations}\label{z11}

Let $p$ be a prime number and let $\kappa$ be $1$ if $p$ is odd
and $2$ if $p=2$. Let $\mathbb{Q}_p\subseteq K$ be a complete and
discretely valued field, and $k$ denotes the residue field of $K$.
Further, let $G$ be a compact locally $\mathbb{Q}_p$-analytic
group which we will assume for the most of the paper to be
uniform.

For any compact locally $\mathbb{Q}_p$-analytic group $G$ we
denote by $D(G,K)$ the algebra of $K$-valued locally analytic
distributions on $G$. Recall that $D(G,K)$ is equal to the strong
dual of the locally convex vector space $C^{an}(G,K)$ of
$K$-valued locally $\mathbb{Q}_p$-analytic functions on $G$ with the
convolution product.

Now if $G$ is uniform then it has a bijective global chart
\begin{eqnarray*}
\mathbb{Z}_p^{d}&\rightarrow&G\\
(x_1,\dots,x_{d})&\mapsto&h_1^{x_1}\cdots h_{d}^{x_{d}}
\end{eqnarray*}
where $h_{1},\dots,h_{d}$ is a fixed (ordered) minimal set of
topological generators of $G$. Putting
$b_i:=h_{i}-1\in\mathbb{Z}[G]$, ${\bf
b}^{\alpha}:=b_{1}^{\alpha_{1}}\cdots b_{d}^{\alpha_{d}}$ for
$\alpha\in\mathbb{N}_0^{d}$ we can identify $D(G,K)$ with the ring
of all formal series
\begin{equation*}
\lambda=\sum_{\alpha\in\mathbb{N}_0^{d}}d_{\alpha}\textbf{b}^\alpha
\end{equation*}
with $d_\alpha$ in $K$ such that the set
$\{|d_\alpha|\rho^{\kappa|\alpha|}\}_\alpha$ is bounded for all
$0<\rho<1$. Here the first $|\cdot|$ is the normalized absolute
value on $K$ and the second one denotes the degree of $\alpha$,
that is $\sum_{i}\alpha_{i}$. Note that this is different from the convention
in the Appendix, where $|\alpha|$ is defined as $\sum_{i}|\alpha_i|$. For any $\rho$ in $p^{\mathbb{Q}}$
with $p^{-1}<\rho<1$, we have a multiplicative norm
$\|\cdot\|_{\rho}$ on $D(G,K)$ \cite{foundations} given by
\begin{equation*}
\|\lambda\|_{\rho}:=\sup_\alpha|d_\alpha|\rho^{\kappa|\alpha|}.
\end{equation*}
The family of norms $\|\cdot\|_{\rho}$ defines the Fr\'echet
topology on $D(G,K)$. The completion with respect to the norm
$\|\cdot\|_{\rho}$ is denoted by $D_{\rho}(G,K)$. (Whenever $G$ is
not uniform then we choose an open normal uniform subgroup $H$ and
representatives $g_1,\dots,g_{G:H}$ for the cosets in $G/H$ we can
define the norms $\|\cdot\|_{\rho}$ on $D(G,K)=\bigoplus_i
D(H,K)g_i$ by
$\|\sum_i\lambda_ig_i\|_{\rho}:=\max_i\|\lambda_i\|_{\rho}$ and
denote by $D_{\rho}(G,K)$ the corresponding completions.)

Now if we fix real numbers $p^{-1}<\rho_1\leq \cdots\leq
\rho_l\leq\cdots<1$ in $p^{\mathbb{Q}}$ such that
$\rho_l\rightarrow 1$ then the distribution algebra $D(G,K)$ is
the projective limit of the noetherian Banach-algebras
$D_{\rho_l}(G,K)$. This gives the Fr\'echet-Stein algebra
structure on $D(G,K)$. A \emph{coherent sheaf} for the projective
system $(D_{\rho_l}(G,K),l)$ of algebras is a projective system of
finitely generated modules $M_l$ over $D_{\rho_l}(G,K)$ together
with isomorphisms
\begin{equation*}
D_{\rho_l}(G,K)\otimes_{D_{\rho_{l+1}}(G,K)}M_{l+1}\overset{\cong}{\rightarrow}
M_l
\end{equation*}
for any positive integer $l$. Following \cite{foundations} we call
a $D(G,K)$-module $M$ \emph{coadmissible} if it is isomorphic to
the $D(G,K)$-module of ``global sections''
\begin{equation*}
\Gamma(M_l):=\varprojlim_l M_l
\end{equation*}
of a coherent sheaf $(M_l)_l$ for $(D(G,K),\rho_l)$. A
coadmissible $D(G,K)$-module need not be finitely generated,
neither is a finitely generated module always coadmissible. For example if one
takes an ideal $I$ in $D(\mathbb{Z}_p,K)$ that is not closed (or equivalently,
not finitely generated) then the quotient
$D(\mathbb{Z}_p,K)/I$ is finitely generated (even cyclic), but not
coadmissible. On the other hand put
$M:=\prod_{i=1}^{\infty}D(\mathbb{Z}_p,K)/\left(\frac{\log(1+x)}{(x+1)^{p^i}-1}\right)$
as a module over $D(\mathbb{Z}_p,K)$. Since the zeroes of $\log(1+x)$ are
exactly of the form $\zeta-1$ with $\zeta\in\mu_{p^{\infty}}$, the restriction
of $M$ to any closed disc of radius $\rho<1$ is finitely generated. Hence $M$ is coadmissible as
it is clearly the projective limit of its restrictions to closed discs of
radius $\rho<1$. Although, $M$ is not finitely generated, since its
localisation $M_{\zeta_{p^m}-1}$ at the
point corresponding to the Galois orbit of $\zeta_{p^m}-1$ has rank $m-1$
for any $m\geq1$ with $\zeta_{p^m}$ being a primitive $p^m$th root of unity.
However, any finitely presented module is coadmissible. For
further details see \cite{foundations}.

\section{Projective coadmissible modules}

\subsection{Filtrations and the Grothendieck group}

The ring $D_{\rho}(G,K)$ is a filtered left and right noetherian
$K$-Banach algebra realizing the Fr\'echet-Stein structure on
$D(G,K)$. The filtration is given by the additive subgroups
\begin{eqnarray}
F_{\rho}^sD_{\rho}(G,K)&:=&\{\lambda\in D_{\rho}(G,K)\colon\|\lambda\|_{\rho}\leq p^{-s}\},\notag\\
F_{\rho}^{s+}D_{\rho}(G,K)&:=&\{\lambda\in
D_{\rho}(G,K)\colon\|\lambda\|_{\rho}< p^{-s}\}\notag
\end{eqnarray}
for $s$ in $\mathbb{R}$ with associated graded ring
\begin{equation*}
\gr
D_{\rho}(G,K):=\bigoplus_{s\in\mathbb{R}}F_{\rho}^s(D_{\rho}(G,K))/F_{\rho}^{s+}D_{\rho}(G,K).
\end{equation*}
It is proven in \cite{foundations} that $\gr D_{\rho}(G,K)$ is
isomorphic to the polynomial ring $(\gr K)[X_{1},\dots,X_{d}]$ and
the principal symbols $\sigma(b_{i})$ of $b_{i}$ are the variables
$X_{i}$ provided that $\rho$ lies in $p^{\mathbb{Q}}$. The graded ring of the field $K$ is naturally isomorphic
to the Laurent polynomial ring $k[X_0,X_0^{-1}]$ over the residue
field $k$ of $K$. The uniformizer maps to the variable $X_0$ under
$\gr$.

\begin{lem}\label{z9}
Any graded projective module over the ring $\gr D_{\rho}(G,K)$ is
stably graded free.
\end{lem}
\begin{proof}
This is a classical statement proven in \cite{G} (§1, Theorem 1).
\end{proof}

Now we claim that $D_{\rho}(G,K)$ satisfies the invariant basis
property. Indeed, $D_{\rho}(G,K)$ is left and right Noetherian
ring without zero divisors so it has a full skew field of
fractions (by Goldie's Theorem). Hence the rank of a finitely
generated free module is well-defined.

For any real number $s$ such that $F_{\rho}^sD_{\rho}(G,K)$ is not
equal to $F_{\rho}^{s+}D_{\rho}(G,K)$ we denote by
$D_{\rho}(G,K)(s)$ the module $D_{\rho}(G,K)$ over itself together
with the shifted filtration
\begin{eqnarray}
F_{\rho}^tD_{\rho}(G,K)(s)&:=&F_{\rho}^{t+s}D_{\rho}(G,K)\notag\\
F_{\rho}^{t+}D_{\rho}(G,K)(s)&:=&F_{\rho}^{(t+s)+}D_{\rho}(G,K).\notag
\end{eqnarray}
We call a finitely generated module \emph{filtered free} if it is
the (filtered) direct sum of modules of the form
$D_{\rho}(G,K)(s)$. Let $\mathcal{C}_0$ be the category of those
finitely generated filtered projective $D_{\rho}(G,K)$-modules $P$
that are direct summands of some filtered free modules of finite
rank as filtered modules. Further we denote by $\mathcal{C}$ the
category of admissibly filtered finitely generated
$D_{\rho}(G,K)$-modules. We call a filtration on a module $M$
\emph{admissible} if it is the image of the filtration of a
filtered free module under a surjective homomorphism onto $M$.
Note that any finitely generated $D_{\rho}(G,K)$-module admits an
admissible filtration. Moreover, the category $\mathcal{C}$ is
abelian by the following lemma.

\begin{lem}\label{z3}
The induced filtration on any submodule of a finitely generated filtered free module is admissible.
\end{lem}
\begin{proof}
Let $M$ be a submodule of the filtered free module $F$. Then $\gr
M$ is a submodule of the finitely generated graded free module
$\gr F$. So $\gr M$ is generated by a finite set of homogeneous
elements in $\gr F$. We can lift up these homogeneous elements to
$M$ and denote the lifts by $m_1,\dots,m_l$. We claim first that
the elements $m_i$ generate the module $M$. Indeed,
$D_{\rho}(G,K)$ is complete with respect to the filtration and for
any element $m$ in $M$ we can choose a sum $\sum_ia_im_i$ such
that the projections to the graded module $\sigma(m)$ and
$\sigma(\sum_ia_im_i)$ are equal. On the other hand since
$D_{\rho}(G,K)$ is a noetherian Banach algebra, any submodule of a
finitely generated module is closed and the claim follows. Now it
suffices to prove that the filtration on $M$ is the induced by the
surjection from $F_1:=\bigoplus_iD_r(G,K)(s_i)$ sending the
generator of $D_{\rho}(G,K)(s_i)$ to $m_i$. Here $s_i$ is the
degree of $m_i$. By construction we have for any real number $s$
that the image of $F_{\rho}^s F_1$ is contained in $F_r^sM$. On
the other hand if $m$ is in $F_r^sM\setminus F_r^{s+}M$ then
$\sigma(m)$ is a homogeneous element of $\gr M$ of degree $s$. So
it can be written in the form
\begin{equation*}
\sigma(m)=\sum_ib_i\sigma(m_i)
\end{equation*}
such that $b_i\sigma(m_i)$ is also homogeneous of degree $s$. By a
similar argument as above using the completeness of $M$ with
respect to the filtration the result follows.
\end{proof}

\begin{lem}\label{z4}
Let $F$ be a finitely generated filtered free
$D_{\rho}(G,K)$-module and $\pi$ be a homogeneous projection of
$\gr F$ of degree $0$. Then $\pi$ can be lifted to a projection of
$F$.
\end{lem}
\begin{proof}
It is easy to see that one can lift $\pi$ up to an endomorphism
$g$ of the module $F$ as it suffices to give $g$ on the free
generators $e_i$ $(i=1,\dots,\rk (F))$ such that
$\sigma(g(e_i))=\pi(\sigma(e_i))$. Now let $B$ be the closed
subring of $\End_{D_{\rho}(G,K)}(F)$ generated by $g$. Then $B$ is
a commutative positively filtered ring that is complete with
respect to the filtration. Moreover, the element $g-g^2$ has
positive degree. Hence taking the ideal $I$ of elements of
positive degree and applying Chapter III, §4, Lemma 2 in \cite{Bo}
one obtains an element $g_1$ in $B$ with $g_1=g_1^2$ and $g-g_1$
is in $I$.
\end{proof}

\begin{lem}\label{z5}
The natural inclusion from $\mathcal{C}_0$ to $\mathcal{C}$
induces an isomorphism $K_0(\mathcal{C}_0)\rightarrow K_0(\mathcal{C})$ on the Grothendieck groups.
\end{lem}
\begin{proof}
It is easy to see (theorem 4 and 5 in \cite{BHS}) that it suffices
to show that each module in $\mathcal{C}$ has a finite resolution
by elements in $\mathcal{C}_0$. Now using Lemma \ref{z3} any
finitely generated admissibly filtered $D_{\rho}(G,K)$-module $M$
has a filtered-free resolution
\begin{equation}
\dots\rightarrow F_j\overset{d_j}{\rightarrow}\dots\rightarrow
F_0\rightarrow M\rightarrow0\label{z2}
\end{equation}
by finitely generated modules as $D_{\rho}(G,K)$ is noetherian.
Moreover, the functor $\gr$ is exact and gives a free resolution
\begin{equation}
\dots\rightarrow \gr F_j\overset{\gr
d_j}{\rightarrow}\dots\rightarrow \gr F_0\rightarrow \gr
M\rightarrow0.\label{z1}
\end{equation}
The global dimension of $\gr D_{\rho}(G,K)$ is finite and bounded
by its Krull dimension. This means that there is a positive
integer $j$ such that $\Ker(\gr d_j)=\gr\Ker(d_j)$ is projective.
Hence we have a section $\beta$ from $\Ker(\gr d_j)$ to $\gr
F_{j+1}$. By Lemma \ref{z4} the projection $\beta\circ\gr d_{j+1}$
of $\gr F_{j+1}$ can be lifted to a projection $g$ of $F_{j+1}$.
Therefore $g(F_{j+1})$ is in the category $\mathcal{C}_0$ and it
maps bijectively onto $\Ker(d_j)$ under the map $d_{j+1}$ and we
are done.
\end{proof}

\begin{pro}\label{z7}
Let $1/p<\rho<1$ be in $p^{\mathbb{Q}}$. Then
$K_0(D_{\rho}(G,K))\cong\mathbb{Z}$, in other words all the
projective $D_{\rho}(G,K)$-modules are stably free.
\end{pro}
\begin{proof}
By Lemma \ref{z5} it is enough to prove that any object in
$\mathcal{C}_0$ is stably filtered free as the map
\begin{equation*}
K_0(\mathcal{C})\rightarrow K_0(D_{\rho}(G,K))
\end{equation*}
induced by the forgetful functor is clearly surjective and
$D_{\rho}(G,K)$ has the invariant basis property. Now if $P$ is in
$\mathcal{C}_0$ then the associated graded module $\gr P$ is also
projective, so by Lemma \ref{z9} it is stably graded free. Since
$\gr$ is exact, $P$ is also stably filtered free and we are done.
\end{proof}

\subsection{The Krull dimension and stable range}

Whenever $R$ is a (not necessarily commutative) ring then
following \cite{RG} its \emph{right Krull dimension} $\kd R$ is
defined by the deviation of the poset of right ideals in $R$. For
the deviation of a poset see Chapter 6 in \cite{MR}. By Corollary
4.8 in \cite{MR} this definition of Krull dimension coincides with
the usual one whenever $R$ is commutative. Since distribution
algebras of groups are isomorphic to their opposite ring via the
map induced by the anti-involution on the group sending the
elements to their inverse, the right and left Krull dimensions of
$D_{\rho}(G,K)$ are the same and we just call this the Krull
dimension $\kd D_{\rho}(G,K)$ of $D_{\rho}(G,K)$. A priori this is
just an ordinal, but we shall show that $\kd D_{\rho}(G,K)$ is
finite.

\begin{lem}\label{z6}
The Krull dimension of the ring $D_{\rho}(G,K)$ is less than or
equal to $d+1=\kd(\gr D_{\rho}(G,K))$.
\end{lem}
\begin{proof}
This follows from Proposition 7.1.2 of Chapter I of \cite{HO} and
Corollary 6.4.8 of \cite{MR} as $D_{\rho}(G,K)$ is complete with
respect to the filtration.
\end{proof}

Now we aim at studying the projective objects in the category of
coadmissible $D(G,K)$-modules. For this we introduce the notion of
\emph{stable range} following \cite{B}. Let $R$ be a ring and we
denote by $\GL(R)$ the infinite dimensional general linear group
$\varinjlim_n\GL_n(R)$ and by $\mathrm{E}(R)$ its subgroup
generated by elementary matrices. Then we have
\begin{eqnarray*}
\mathrm{E}(R)&\cong&[\GL(R),\GL(R)]\hbox{ and}\\
K_1(R)&\cong&\GL(R)/\mathrm{E}(R).
\end{eqnarray*}
Let $a=(a_1,\dots,a_l)$ be an element in the right module $R^l$.
We call $a$ \emph{unimodular} if there is a homomorphism $f\colon
R\rightarrow R$ such that $f(a)=1$ (this is equivalent in this special case to
that the left ideal $Ra_1+\dots+Ra_l$ equals $R$). We say that the positive
integer $n$ defines a \emph{stable range} for $\GL(R)$ if, for all
$l>n$, given $(a_1,\dots,a_l)$ unimodular in $R^l$, there exist
$b_1,\dots,b_{l-1}$ in $R$ such that
$(a_1+b_1a_l,\dots,a_{l-1}+b_{l-1}a_l)$ is unimodular in
$R^{l-1}$.

\begin{cor}\label{z8}
The number $d+2$ defines a stable range for $\GL(D_{\rho}(G,K))$.
Any projective $D_{\rho}(G,K)$-module of rank at least $d+2$ is
free. In particular if $P$ is a finitely generated projective
module then it can be generated by $\max(\rk(P),d+2)$ elements.
\end{cor}
\begin{proof}
The first statement follows from Lemma \ref{z6} and the Theorem in
\cite{St} stating that if the Krull dimension of a right
noetherian ring $R$ is at most $n-1$ then $n$ defines a stable
range for $\GL(R)$. Now by Theorem 4.2 in \cite{B} the group
$\mathrm{E}_{d+3}(D_{\rho}(G,K))$ generated by elementary
$(d+3)\times(d+3)$ matrices is surjective on the unimodular rows
in $D_{\rho}(G,K)^{d+3}$. So if $P$ is a projective module such
that
\begin{equation*}
P\oplus F\cong D_{\rho}(G,K)^t
\end{equation*}
for some integer $t\leq d+3$ and $F\cong D_{\rho}(G,K)$ then there
is an automorphism $g$ of $D_{\rho}(G,K)^t$ sending the generator
of $F$ to the first basis vector in $D_{\rho}(G,K)^t$ as the
generator of $F$ is unimodular. So we obtain
\begin{equation*}
P\cong(P\oplus
F)/F\overset{g}{\cong}D_{\rho}(G,K)^t/D_{\rho}(G,K)\cong
D_{\rho}(G,K)^{t-1}.
\end{equation*}
The result follows by Proposition \ref{z7}.
\end{proof}

\begin{lem}\label{z44}
Let $G$ be a uniform pro-$p$ group of dimension $d$. Then for any
positive integer $m$ the subgroup
$\mathrm{E}_{m+d+2}(D_{\rho}(G,K))$ of
$\GL_{m+d+2}(D_{\rho}(G,K))$ generated by elementary matrices acts
transitively on the surjections from $D_{\rho}(G,K)^{m+d+2}$ to
$D_{\rho}(G,K)^m$.
\end{lem}
\begin{proof}
Fix a basis of both $D_{\rho}(G,K)^{m+d+2}$ and $D_{\rho}(G,K)^m$.
Let $A=(a_{ij})_{1\leq i\leq m,1\leq j\leq m+d+2}$ be the matrix
of a surjection. We need to show that there is a matrix $B$ in
$\mathrm{E}_{m+d+2}(D_{\rho}(G,K))$ such that we have
\begin{equation}
AB=
\begin{pmatrix}
1&0&0&\cdots&0&0&\cdots&0\\
0&1&0&\cdots&0&0&\cdots&0\\
\vdots&\ddots&\ddots&\ddots&\vdots&\vdots&&\vdots\\
0&\cdots&0&1&0&0&\cdots&0\\
0&\cdots&\cdots&0&1&0&\cdots&0\\
\end{pmatrix}\label{z43}
\end{equation}
is the standard surjection. By Theorem 4.2a) in \cite{B}
$\mathrm{E}_{m+d+2}(D_{\rho}(G,K))$ acts transitively on
unimodular rows since $d+2$ defines a stable range by Corollary \ref{z8}. On the other hand the first (or in fact any) row
of $A$ is unimodular as $A$ is surjective. This means that we may
assume without loss of generality that the first row of $A$ is
\begin{equation}
\begin{pmatrix}
1&0&\cdots&0\label{z42}
\end{pmatrix}.
\end{equation}
Using this above remark we prove the statement by induction on
$m$. If $m=1$ then we are done by the above argument. Now let
$m>1$ and suppose the statement is true for $m-1$. Since the first
row of $A$ is as in $(\ref{z42})$ we have that the matrix
$A'=(a_{ij})_{2\leq i\leq m,2\leq j\leq m+d+2}$ is also a
surjection. By the inductional hypothesis there is a matrix $B$ in
$\mathrm{E}_{m+d+1}(D_{\rho}(G,K))$ such that
\begin{equation*}
A\begin{pmatrix} 1&0\\
0&B
\end{pmatrix}
=\begin{pmatrix}
1&0&0&\cdots&0&0&\cdots&0\\
*&1&0&\cdots&0&0&\cdots&0\\
\vdots&0&\ddots&\ddots&\vdots&\vdots&&\vdots\\
*&\vdots&\ddots&1&0&0&\cdots&0\\
*&0&\cdots&0&1&0&\cdots&0\\
\end{pmatrix}.
\end{equation*}
Now $A$ is almost in the required form and the nonzero entries in
the first column can be easily removed by multiplication by
elementary matrices.
\end{proof}

\begin{lem}\label{z45}
The natural maps
\begin{equation*}
\mathrm{E}_{l}(D_{\rho_1}(G,K))\rightarrow
\mathrm{E}_{l}(D_{\rho_2}(G,K))
\end{equation*}
have dense image for any $p^{-1}<\rho_2<\rho_1<1$ and positive
integer $l$.
\end{lem}
\begin{proof}
Since any element in $\mathrm{E}_{l}(D_{\rho_2}(G,K))$ is the
product of finitely many elementary matrices the result follows
from the density of the image of the map
$D_{\rho_1}(G,K)\rightarrow D_{\rho_2}(G,K)$.
\end{proof}

The main result of this section is the following generalization of
Theorem 1 in V.\ of \cite{G}.

\begin{thm}\label{z10}
Let $G$ be a compact locally $\mathbb{Q}_p$-analytic group of
dimension $d$. All the projective objects in the category of
coadmissible modules over $D(G,K)$ are finitely generated. In
particular any projective object is a projective module over
$D(G,K)$ and the category does not have enough projectives.
\end{thm}
\begin{rem}
Whenever $G$ is uniform the rank of a projective coadmissible
$D(G,K)$-module $P$ is defined by the rank of
$D_{\rho}(G,K)\otimes P$. This is clearly independent of $\rho$.
\end{rem}
\begin{proof}
It suffices to prove the statement when $G$ is uniform as any
compact $p$-adic Lie group has a finite index uniform subgroup.

Now assume that $G$ is uniform and let $P$ be a projective object
in the category of coadmissible $D(G,K)$-modules. We claim first
that for any $\rho$ in $p^{\mathbb{Q}}$ with $1/p<\rho<1$ the
$D_{\rho}(G,K)$-module $P_{\rho}:=D_{\rho}(G,K)\otimes_{D(G,K)}P$
is also projective (and finitely generated since $P$ is
coadmissible). Let $m_1,\dots,m_l$ be elements in $P$ such that
$1\otimes m_1,\dots,1\otimes m_l$ generate
$P_{\rho}=D_{\rho}(G,K)\otimes P$. Take the map
\begin{eqnarray}
\varphi\colon D(G,K)^l&\rightarrow& P\notag\\
e_j&\mapsto& m_j\notag
\end{eqnarray}
where $e_j$ is the $j$th standard basis vector of $D(G,K)^l$. Since
$P$ is a projective object we obtain an exact sequence
\begin{equation*}
0\rightarrow\Hom(P,\Ker(\varphi))\rightarrow\Hom(P,D(G,K)^l)\rightarrow\Hom(P,P)\rightarrow\Hom(P,\Coker(\varphi))\rightarrow0
\end{equation*}
of coadmissible modules. By Lemma 8.4 in \cite{foundations} the
following sequence is also exact for any $p^{-1}<\rho<1$ in
$p^{\mathbb{Q}}$.
\begin{eqnarray}
0\rightarrow\Hom_{D_{\rho}(G,K)}(P_{\rho},\Ker(\varphi)_{\rho})\rightarrow\Hom_{D_{\rho}(G,K)}(P_{\rho},D_{\rho}(G,K)^l)\rightarrow\notag\\
\rightarrow\Hom_{D_{\rho}(G,K)}(P_{\rho},P_{\rho})
\rightarrow\Hom_{D_{\rho}(G,K)}(P_{\rho},\Coker(\varphi)_{\rho})\rightarrow0\notag
\end{eqnarray}
On the other hand $\Coker(\varphi)_{\rho}=0$ as $1\otimes
m_1,\dots,1\otimes m_l$ generate $P_\rho$. So the map
\begin{equation*}
\Hom(P_{\rho},D_{\rho}(G,K)^l)\rightarrow\Hom(P_{\rho},P_{\rho})
\end{equation*}
is surjective which means that $P_{\rho}$ is a direct summand of
$D_{\rho}(G,K)^l$ and hence projective.

By Corollary \ref{z8} the $D_{\rho}(G,K)$-module $P_{\rho}\oplus
D_{\rho}(G,K)^{\max(\rk(P),d+2)-\rk(P)}$ is free of rank
$\max(\rk(P),d+2)$. So we have
\begin{eqnarray}
P\oplus D(G,K)^{\max(\rk(P),d+2)-\rk(P)}&\cong&\varprojlim_{\rho}
P_{\rho}\oplus\varprojlim_{\rho}
D_{\rho}(G,K)^{\max(\rk(P),d+2)-\rk(P)}\notag\\
&\cong&\varprojlim_{\rho} (P_{\rho}\oplus
D_{\rho}(G,K)^{\max(\rk(P),d+2)-\rk(P)})\notag\\
&\cong& \varprojlim_{\rho} D_{\rho}(G,K)^{\max(\rk(P),d+2)}.\notag
\end{eqnarray}
Note that the connecting maps in $\varprojlim_{\rho}
D_{\rho}(G,K)^{\max(\rk(P),d+2)}$ may not map the standard basis
vectors to each other. This is the reason why we need Lemma
\ref{z44}. We are going to show that $D(G,K)^{m+d+2}$ surjects
onto $P\oplus D(G,K)^{m-\rk(P)}$ where $m=\max(\rk(P),d+2)$. This
will show that $P$ is finitely generated. Let $E_{\rho}$ be the
set of surjections from $D_{\rho}(G,K)^{m+d+2}$ to
$D(G,K)_{\rho}^{m}$. The sets $E_{\rho}$ form a projective system
with continuous connecting maps and it suffices to show that the
projective limit of this system is nonempty. By Lemma \ref{z44} we
have a transitive and continuous action of
$\mathrm{E}_{m+d+2}(D_{\rho}(H,K))$ on $E_{\rho}$ for each
$p^{-1}\leq \rho<1$ in $p^{\mathbb{Q}}$. Applying Lemma \ref{z45}
with $l=m+d+2$ we obtain that the image of $E_{\rho_1}$ in
$E_{\rho_2}$ is also dense. Therefore the projective system
$(E_\rho)_\rho$ satisfies the Mittag-Leffler condition so its
projective limit is nonempty.

Once we know that $P$ is finitely generated, it is easy to see that $P$ is a projective module over $D(G,K)$ because it is the direct
summand of a free module (of finite rank) as the finitely
generated free modules over $D(G,K)$ are coadmissible. The last
statement follows from the fact that there are coadmissible
modules over $D(G,K)$ which are not finitely generated.
\end{proof}

\begin{rems}\begin{enumerate}
\item The conversion of Proposition \ref{z10} is also true. If $P$ is a
finitely generated projective module over $D(G,K)$ then it is of
course finitely presented and so coadmissible by Corollary 3.4$v$
in \cite{foundations}.
\item It would be interesting to have a generalization of
Proposition \ref{z10} to any noncommutative Fr\'echet-Stein
algebra. However, general noetherian Banach algebras might have
infinite Krull dimension. Moreover, even in the case of
\emph{commutative} locally $L$-analytic groups where $L\supsetneq
\mathbb{Q}_p$ the Grothendieck group $K_0(D_\rho(G,K))$ is more
complicated. For example when $o_L$ is the additive group of the
ring of integers in $L$ we have
$K_0(D_\rho(o_L,K))\neq\mathbb{Z}$ (see \cite{S1}).
\end{enumerate}
\end{rems}

\section{Fr\'echet-Stein structure}\label{sec:Frechet-Stein}

Let $G$ be a uniform pro-$p$ group. In this section we are going
to prove that the topological $K$-algebras $D_{[\rho,1)}(G,K)$
constructed in the Appendix are Fr\'echet-Stein.

Let $p^{-1}<\rho_1<\rho_3<\rho_2<1$ be real numbers in $p^{\mathbb{Q}}$. Let $D_{[\rho_1,\rho_2)^{bd}}(G,K)$ and
$D_{(\rho_1,\rho_2]^{bd}}(G,K)$ be the subspace of $D_{[\rho_1,\rho_3]}(G,K)$ and
$D_{[\rho_3,\rho_2]}(G,K)$, respectively, consisting of those Laurent series
that are bounded in the norm
$\|\cdot\|_{\rho_1,\rho_2}:=\max(\|\cdot\|_{\rho_1},\|\cdot\|_{\rho_2})$. This
definition is clearly independent of the choice of $\rho_3$. These
are a priori Banach spaces, however, since the norm
$\|\cdot\|_{\rho_1,\rho_2}$ is submultiplicative on monomials they form
subalgebras of $D_{[\rho_1,\rho_3]}(G,K)$ and
$D_{[\rho_3,\rho_2]}(G,K)$, respectively. 
Moreover, this norm induces a filtration on each algebra given by
\begin{eqnarray*}
F_{\rho_1,\rho_2}^s D_{*}(G,K)&:=&\{x\in D_{*}(G,K)\colon\|x\|_{\rho_1,\rho_2}\leq p^{-s}\},\\
F_{\rho_1,\rho_2}^{s+}D_*(G,K)&:=&\{x\in
D_*(G,K)\colon\|x\|_{\rho_1,\rho_2}< p^{-s}\}
\end{eqnarray*}
where $*$ denotes either $[\rho_1,\rho_2]$, $(\rho_1,\rho_2]^{bd}$, or
$[\rho_1,\rho_2)^{bd}$.

\begin{pro}
If $\rho_1$ and $\rho_2$ are both in $p^{\mathbb{Q}}$ then the
above algebras $D_*(G,K)$ are all noetherian.
\end{pro}
\begin{proof}
We may extend scalars without loss of generality to a finite
extension of $K$. So we assume that both $\rho_1$ and $\rho_2$ are
integral powers of the absolute value of the uniformizer of $K$.
At first we prove the statement for $D_{[\rho_1,\rho_2]}(G,K)$.
The other cases are dealt with in a similar way. The underlying
space of the graded ring
$\gr_{\rho_1,\rho_2}D_{[\rho_1,\rho_2]}(G,K)$ is the set of
Laurent polynomials in variables $X_1,\dots,X_d$ over the graded
ring $\gr K=k[X_0,X_0^{-1}]$ of $K$. Here $X_0$ is the principal
symbol of the uniformizer of the field $K$ and
$X_i=\sigma(b_i)/X_0^s$ for each $1\leq i\leq d$ where $\rho_2$ is
the $s$th power of the absolute value of the unformizer of $K$.
The $\gr K$-module structure is clear the question is only the
multiplication of monomials ${\bf X}^\alpha=X_1^{\alpha_1}\cdots
X_d^{\alpha_d}$. However, this is non-standard, the product of two
monomials ${\bf X}^\alpha$ and ${\bf X}^\beta$ is defined for any
$\alpha$ and $\beta$ in $\mathbb{Z}^d$ by
\begin{equation}
{\bf X}^\alpha{\bf X}^\beta=\begin{cases} 0&\hbox{if
}|\alpha|<0\hbox{
and }|\beta|>0\hbox{ or if }|\alpha|>0\hbox{ and }|\beta|<0;\\
{\bf X}^{\alpha+\beta}&\hbox{otherwise.}
\end{cases}\label{z29}
\end{equation}
Indeed, we have $\|{\bf b}^{\alpha}{\bf
b}^{\beta}\|_{\rho_1,\rho_2}=\|{\bf
b}^{\alpha}\|_{\rho_1,\rho_2}\|{\bf b}^{\beta}\|_{\rho_1,\rho_2}$
if and only if
\begin{equation}
\begin{cases}
\hbox{either }&\|{\bf b}^{\alpha}\|_{\rho_1,\rho_2}=\|{\bf
b}^{\alpha}\|_{\rho_1}\hbox{ and }\|{\bf
b}^{\beta}\|_{\rho_1,\rho_2}=\|{\bf b}^{\beta}\|_{\rho_1}\\
\hbox{or }&\|{\bf b}^{\alpha}\|_{\rho_1,\rho_2}=\|{\bf
b}^{\alpha}\|_{\rho_2}\hbox{ and }\|{\bf
b}^{\beta}\|_{\rho_1,\rho_2}=\|{\bf b}^{\beta}\|_{\rho_2}.
\end{cases}\label{z30}
\end{equation}
Since $\rho_1<\rho_2$ the two norms of the term ${\bf b}^\alpha$
coincide if and only if the degree $|\alpha|$ of $\alpha$ is zero.
So the condition $(\ref{z30})$ is equivalent to that at least one
of $|\alpha|$ and $|\beta|$ is zero or they have the same sign and
claim $(\ref{z29})$ follows.

Now the filtration on $D_{[\rho_1,\rho_2]}(G,K)$ is complete and
indexed by the integer multiples of a rational number, so by Prop.
I.7.1.2 in \cite{HO} it suffices to show that
$\gr_{\rho_1,\rho_2}D_{[\rho_1,\rho_2]}(G,K)$ is noetherian. For
this take an ideal $I$ in
$\gr_{\rho_1,\rho_2}D_{[\rho_1,\rho_2]}(G,K)$. Let $\pi$ denote
the projection of $\gr_{\rho_1,\rho_2}D_{[\rho_1,\rho_2]}(G,K)$ to
the terms $\sum c_{\alpha}{\bf X}^{\alpha}$ with $c_{\alpha}$ in
$k[X_0,X_0^{-1}]$ and $\alpha$ in $\mathbb{Z}^d$ with
$|\alpha|\geq 0$. This projection $\pi$ is in fact a ring
homomorphism onto the Laurent polynomial ring
\begin{equation}
B:=k[X_0,X_0^{-1},X_1,X_2X_1^{-1},\dots,X_dX_1^{-1},X_1X_2^{-1},\dots,X_1X_d^{-1}]\label{z33}
\end{equation}
which is noetherian. So $\pi(I)$ is finitely generated. On the
other hand $\Ker(\pi)\cap I$ is also an ideal in the Laurent
polynomial ring
\begin{equation*}
k[X_0,X_0^{-1},X_1^{-1},X_2X_1^{-1},\dots,X_dX_1^{-1},X_1X_2^{-1},\dots,X_1X_d^{-1}]
\end{equation*}
which is viewed as a subring of
$\gr_{\rho_1,\rho_2}D_{[\rho_1,\rho_2]}(G,K)$. This is also
noetherian and therefore $\Ker(\pi)\cap I$ is also finitely
generated. Putting together these generators with (any) lifts of
the generators of $\pi(I)$ we get a finite generating system of
$I$ and the result follows for $D_{[\rho_1,\rho_2]}(G,K)$.

The proof of the statement for the rings
$D_{(\rho_1,\rho_2]^{bd}}(G,K)$ and $D_{[\rho_1,\rho_2)^{bd}}(G,K)$ is
similar. We only show it for $D_{[\rho_1,\rho_2)^{bd}}(G,K)$ and leave
the other case to the reader. The underlying space of the graded
ring $\gr_{\rho_1,\rho_2}D_{[\rho_1,\rho_2)^{bd}}(G,K)$ is the Laurent
polynomial ring $C[X_0,X_0^{-1}]$ over the ring
\begin{equation*}
C:=k[X_2X_1^{-1},\dots,X_dX_1^{-1},X_1X_2^{-1},\dots,X_1X_d^{-1}]\bg
X_1\jg.
\end{equation*}
The multiplication is still given by $(\ref{z29})$. Hence it
remains to show that the ring
\begin{equation}
C_0:=k[X_2X_1^{-1},\dots,X_dX_1^{-1},X_1X_2^{-1},\dots,X_1X_d^{-1}]\bs
X_1\js\label{z34}
\end{equation}
is noetherian as this is the `positive part' of
$\gr_{\rho_1,\rho_2}D_{[\rho_1,\rho_2)^{bd}}(G,K)$. However, $C_0$ is
also completely filtered by the lowest degree and the
corresponding graded ring is the Laurent polynomial ring
\begin{equation*}
k[X_1,X_2X_1^{-1},\dots,X_dX_1^{-1},X_1X_2^{-1},\dots,X_1X_d^{-1}]
\end{equation*}
which is noetherian. The result follows again by applying
Proposition I.7.1.2 of \cite{HO} twice.
\end{proof}

\begin{lem}\label{z35}
The ring $C_0[X_0,X_0^{-1}]$ is flat over the ring $B$. (For the
definition of $B$ and $C_0$ see $(\ref{z33})$ and $(\ref{z34})$,
respectively.)
\end{lem}
\begin{proof}
The ring of formal power series (in the variable $X_1$) over any
ring $R$ is well-known to be flat over the ring of polynomials.
Moreover, if the ring $R_2$ is flat over the ring $R_1$ then so is
the Laurent polynomial ring $R_2[X_0,X_0^{-1}]$ over
$R_1[X_0,X_0^{-1}]$. Here we use these statements with
\begin{eqnarray*}
R&:=&k[X_2X_1^{-1},\dots,X_dX_1^{-1},X_1X_2^{-1},\dots,X_1X_d^{-1}]\\
R_1&:=&k[X_1,X_2X_1^{-1},\dots,X_dX_1^{-1},X_1X_2^{-1},\dots,X_1X_d^{-1}]\\
R_2&:=&C_0
\end{eqnarray*}
and obtain the statement of the Lemma.
\end{proof}

Now we can state our main Theorem in this section which is a
slight generalization of Theorem 4.9 of \cite{foundations}.

\begin{thm}\label{z37}
The topological $K$-algebra $ D_{[\rho,1)}(G,K)$ is a
Fr\'echet-Stein $K$-algebra for any $p^{-1}<\rho<1$ in
$p^{\mathbb{Q}}$.
\end{thm}
\begin{proof}
We certainly have
\begin{equation*}
 D_{[\rho,1)}(G,K)=\varprojlim_{\rho_2\rightarrow1}D_{[\rho,\rho_2]}(G,K)
\end{equation*}
is a limit of noetherian Banach algebras. So it suffices to show
that whenever $\rho<\rho_2<\rho_3<1$ are all in $p^{\mathbb{Q}}$
then $D_{[\rho,\rho_2]}(G,K)$ is flat as a (left or right)
$D_{[\rho,\rho_3]}(G,K)$-module. As in the proof of Theorem 4.9 of
\cite{foundations} we do this in two steps. We show that both of
the maps
\begin{eqnarray}
D_{[\rho,\rho_3]}(G,K)&\hookrightarrow&
D_{[\rho,\rho_3)^{bd}}(G,K)\hbox{ and}\label{z31}\\
D_{[\rho,\rho_3)^{bd}}(G,K)&\hookrightarrow&
D_{[\rho,\rho_2]}(G,K)\label{z32}
\end{eqnarray}
are flat and then so is their composite. Both of the statements
can be checked over a finite extension of $K$ as the tensor
product with a finite extension is faithfully flat.

We begin with showing that the map $(\ref{z31})$ is flat. By
Proposition 1.2 in \cite{foundations} (see also \cite{HO}) it
suffices to verify that the map of graded rings
\begin{equation*}
\gr_{\rho,\rho_3}D_{[\rho,\rho_3]}(G,K)\hookrightarrow
\gr_{\rho,\rho_3}D_{[\rho,\rho_3)^{bd}}(G,K)
\end{equation*}
is flat. However, the functors
\begin{equation*}
\gr_{\rho,\rho_3}D_{[\rho,\rho_3)}(G,K)\otimes_{\gr_{\rho,\rho_3}D_{[\rho,\rho_3]}(G,K)}\cdot\hspace{1cm}\hbox{
and }\hspace{1cm}C_0[X_0,X_0^{-1}]\otimes_B\cdot
\end{equation*}
coincide on $\gr_{\rho,\rho_3}D_{[\rho,\rho_3]}(G,K)$-modules and
the assertion follows from Lemma \ref{z35}.

Now we prove the flatness of $(\ref{z32})$. For this we are going
to use the filtration $F^s_{\rho,\rho_2}$ induced by the norm
$\|\cdot\|_{\rho,\rho_2}$ on both $D_{[\rho,\rho_3)^{bd}}(G,K)$ and
$D_{[\rho,\rho_2]}(G,K)$. This filtration is not complete on
$D_{[\rho,\rho_3)^{bd}}(G,K)$, however, it is complete on its subring
$F_{\rho,\rho_3}^0D_{[\rho,\rho_3)^{bd}}(G,K)$. It suffices to prove
the flatness of $D_{[\rho,\rho_2]}(G,K)$ over
$F_{\rho,\rho_3}^0D_{[\rho,\rho_3)^{bd}}(G,K)$ as we have
\begin{equation*}
D_{[\rho,\rho_3)^{bd}}(G,K)=\mathbb{Q}_p\otimes_{\mathbb{Z}_p}F_{\rho,\rho_3}^0D_{[\rho,\rho_3)^{bd}}(G,K).
\end{equation*}
As the filtration is complete on both sides we are reduced to show
by Proposition 1.2 of \cite{foundations} that the map of graded
rings
\begin{equation}
\gr_{\rho,\rho_2}F_{\rho,\rho_3}^0D_{[\rho,\rho_3)^{bd}}(G,K)\hookrightarrow\gr_{\rho,\rho_2}D_{[\rho,\rho_2]}(G,K)\label{z36}
\end{equation}
is flat. This is clear as the right hand side of \eqref{z36} is the
localization of the left hand side at $X_0$.
\end{proof}

\section{Topologies and self-duality}

Recall that the classical Robba ring
$\mathcal{R}=\mathcal{R}(\mathbb{Z}_p,K)$ can be identified with
the ring of formal Laurent-series over $K$ that are convergent in
some open annulus with outer radius $1$. $\mathcal{R}$ is a
self-dual topological $K$-algebra with respect to the perfect
pairing \cite{R}
\begin{eqnarray}
\mathcal{R}\times\mathcal{R}&\rightarrow&K\notag\\
(x,y)&\mapsto&``\hbox{constant term of the power series
}x(T)y(T)''.\notag
\end{eqnarray}

In this section we are going to show that the Robba ring for more
general uniform pro-$p$ groups defined in the Appendix is also
self-dual in an appropriate topology and pairing.

\subsection{The weak topology}

An immediate consequence of Proposition \ref{expansion} in the
Appendix is that $\mathcal{R}(G,K)$ is also a topological
$K$-algebra with the inductive limit topology. However, we need to
introduce yet another topology on each $ D_{[\rho,1)}(G,K)$ and on
$\mathcal{R}(G,K)$ so that they all become self-dual. Let $R$
denote any of the rings $ D_{[\rho,1)}(G,K)$ or
$\mathcal{R}(G,K)$. Then we define the ``weak topology'' on $R$ as
the coarsest locally convex topology such that for any (fixed)
$x_0$ in $R$ the map
\begin{eqnarray}
R&\rightarrow& K\label{z19}\\
x&\mapsto& (x_0,x):=``\hbox{the constant term of }x_0x''\notag
\end{eqnarray}
is continuous. We shall see later on that these topologies are
different from the usual (Fr\'echet/inductive limit) topology on
each $R$, the weak topology is strictly coarser. Let us first
remark that the multiplication on $R$ is separately continuous in
the weak topology. Now we need the following lemmata.

\begin{lem}\label{z21}
For any index $\alpha$ in $\mathbb{Z}^d$ the following statements
hold.
\begin{enumerate}[$(i)$]
\item For any index $\beta\neq-\alpha$ we have $|({\bf
b}^{\alpha},{\bf b}^{\beta})|\leq \min(1,p^{-|\alpha|-|\beta|})$.
Moreover, if we assume further that $|\alpha|+|\beta|=0$ then we
have $|({\bf b}^{\alpha},{\bf b}^{\beta})|\leq p^{-1}$.

\item For any given real number $\varepsilon>0$ and integer $m$
there are only finitely many indices $\beta$ in $\mathbb{Z}^d$
such that $|\beta|=m$ and $|({\bf b}^{\alpha},{\bf
b}^{\beta})|>\varepsilon$.
\end{enumerate}
\end{lem}
\begin{proof}
The first assertion follows immediately from Lemma
\ref{coefficients} in the Appendix.

For $(ii)$ we are going to prove the following quantitative
estimate. Put
\begin{equation*}
N:=\max(\alpha_1+\beta_1,\dots,\alpha_d+\beta_d).
\end{equation*}

Then we have
\begin{equation}
|({\bf b}^{\alpha},{\bf b}^{\beta})|\leq
p^{-\frac{|\alpha|+|\beta|+N}{2}}.\label{z20}
\end{equation}
Indeed, the result follows from this noting that for any fixed
degree and fixed maximum of the coordinates there are only
finitely many $\beta$.

Now we fix $\alpha$ and $\beta$ in order to prove $(\ref{z20})$.
We may assume without loss of generality for the moment that
$K=\mathbb{Q}_p$ as the product ${\bf b}^{\alpha}{\bf b}^{\beta}$
is already defined over $\mathbb{Q}_p$ (and even over
$\mathbb{Z}_p$). Now we analyze further the Laurent-series
expansion of ${\bf b}^{\alpha}{\bf b}^{\beta}$ by giving an
(infinite) algorithm computing it. The algorithm is the following:

Before the $n$th step of the algorithm we have a formal sum
\begin{equation}\label{z47}
\sum_{i=1}^{\infty}c_{n,i}w_{n,i}
\end{equation}
where $c_{n,i}$ is in $\mathbb{Z}_p$ and $w_{n,i}$ is a word in
the free group $F(b_1,\dots,b_d)$ on the generators
$b_1,\dots,b_d$ such that the sum satisfies the convergence
condition
\begin{equation}\label{z46}
\lim_{i\rightarrow\infty}\|c_{n,i}w_{n,i}\|_{p^{-1/2}}=0.
\end{equation}
Since the norm $\|\cdot\|_{p^{-1/2}}$ is multiplicative on
$D_{[p^{-1/2},1)}(G,K)$, we have
\begin{equation*}
\|c_{n,i}w_{n,i}\|_{p^{-1/2}}=|c_{n,i}|p^{-\frac{|w_{n,i}|}{2}}
\end{equation*}
where $|w_{n,i}|$ is the image of $w_{n,i}$ under the homomorphism
from $F(b_1,\dots,b_d)$ to $\mathbb{Z}$ sending all the generators
to $1$. We also assume for the sake of simplicity that the terms
$c_{n,i}w_{n,i}$ are in decreasing order with respect to the norm
$\|\cdot\|_{p^{-1/2}}$. Note that granted the convergence
criterion $(\ref{z46})$ the above sum defines an element in the
ring $D_{[p^{-1/2},1)}(G,K)$, however it is not in the canonical
form. Indeed, the sum $(\ref{z47})$ is also convergent in any
other norm $\|\cdot\|_{\rho}$ with $p^{-1/2}\leq\rho<1$ because
the coefficients are in $\mathbb{Z}_p$. We call a word $w$ in the
free group $F(b_1,\dots,b_d)$ \emph{bad} if it is not of the form
\begin{equation*}
w=b_1^{\alpha_1}b_2^{\alpha_2}\cdots b_d^{\alpha_d}.
\end{equation*}
Before the first step we have $w_{1,1}={\bf b}^{\alpha}{\bf
b}^{\beta}$, $c_{1,1}=1$, and all the other terms are $0$ in the
sum. Now we describe the $n$th step of the algorithm producing the
sum $\sum_{i=1}^{\infty}c_{n+1,i}w_{n+1,i}$ from
$\sum_{i=1}^{\infty}c_{n,i}w_{n,i}$ such that in
$D_{[p^{-1/2},1)}(G,\mathbb{Q}_p)$ they converge to the same
element. Let $s_n$ denote the smallest positive integer such that
$w_{n,s_n}$ is a bad word. If there is no such an integer then the
algorithm terminates and our sum defines an element of
$D_{[p^{-1/2},1)}(G,\mathbb{Q}_p)$ in the canonical form (and in
fact of $ D_{[\rho,1)}(G,\mathbb{Q}_p)$ for any $p^{-1}<\rho<1$,
but we do not need this). We are going to replace $w_{n,s_n}$ by
an infinite sum of the form
\begin{equation}
w_0+\sum_{i=1}^{\infty}c_iw_i\label{z17}
\end{equation}
such that we have
\begin{eqnarray}
w_{n,s_n}&=&w_0+\sum_{i=1}^{\infty}c_iw_i\hbox{ in }D_{[p^{-1/2},1)}(G,\mathbb{Q}_p);\notag\\
\|w_{n,s_n}\|_{p^{-1/2}}&=&\|w_0\|_{p^{-1/2}}>\|c_iw_i\|_{p^{-1/2}}
\hbox{ for any }i>0;\hspace{0.5cm} \hbox{ and}\label{z50}\\
\lim_{i\rightarrow\infty}\|c_iw_i\|_{p^{-1/2}}&=&0.\notag
\end{eqnarray}
For this we write $w_{n,s_n}$ in the form
\begin{equation}
b_{x_1}^{e_1}b_{x_2}^{e_2}...b_{x_z}^{e_z}\label{z51}
\end{equation}
where $1\leq x_j\leq d$ and $e_j$ are $+1$ or $-1$ for each $1\leq
j\leq z$. Now let $t$ be the smallest positive integer such that
$x_{t}>x_{t+1}$. There exists such a $t$ since the word
$w_{n,s_n}$ is bad. We would like to swap $b_{x_t}^{e_t}$ and $b_{x_{t+1}}^{e_{t+1}}$ in the word $w$ so that the word becomes `less bad'. To do this we need the following identities that hold in any
associative ring in which $b_{x_t}$ and $b_{x_{t+1}}$ are
invertible.
\begin{eqnarray}
b_{x_t}b_{x_{t+1}}&=&b_{x_{t+1}}b_{x_t}+(b_{x_t}b_{x_{t+1}}-b_{x_{t+1}}b_{x_t}),\notag\\
b_{x_t}^{-1}b_{x_{t+1}}&=&b_{x_{t+1}}b_{x_t}^{-1}+b_{x_t}^{-1}(b_{x_{t+1}}b_{x_t}-b_{x_t}b_{x_{t+1}})b_{x_t}^{-1},\notag\\
b_{x_t}b_{x_{t+1}}^{-1}&=&b_{x_{t+1}}^{-1}b_{x_t}+b_{x_{t+1}}^{-1}(b_{x_{t+1}}b_{x_t}-b_{x_t}b_{x_{t+1}})b_{x_{t+1}}^{-1}\hbox{, and}\label{z14}\\
b_{x_t}^{-1}b_{x_{t+1}}^{-1}&=&b_{x_{t+1}}^{-1}b_{x_t}^{-1}+b_{x_t}^{-1}b_{x_{t+1}}^{-1}(b_{x_t}b_{x_{t+1}}-b_{x_{t+1}}b_{x_t})b_{x_{t+1}}^{-1}b_{x_t}^{-1}\notag.
\end{eqnarray}
On the other hand $b_{x_t}b_{x_{t+1}}-b_{x_{t+1}}b_{x_t}$ has a
canonical expansion $\sum_\gamma a_\gamma{\bf b}^\gamma$ in the
distribution algebra $D(G,K)$ with $a_{\gamma}$ in $\mathbb{Z}_p$.
So by collecting all the identities in $(\ref{z14})$ and plugging
in the canonical expansion of
$b_{x_t}b_{x_{t+1}}-b_{x_{t+1}}b_{x_t}$ we obtain
\begin{equation}
b_{x_t}^{e_t}b_{x_{t+1}}^{e_{t+1}}=b_{x_{t+1}}^{e_{t+1}}b_{x_t}^{e_t}+\sum_{\gamma\in\mathbb{N}_0^{d}}
a_{\gamma}e_te_{t+1}b_{x_t}^{\frac{e_t-1}{2}}b_{x_{t+1}}^{\frac{e_{t+1}-1}{2}}{\bf
b}^{\gamma}b_{x_{t+1}}^{\frac{e_{t+1}-1}{2}}b_{x_t}^{\frac{e_t-1}{2}}\label{z15}
\end{equation}
in $D_{[p^{-1/2},1)}(G,K)$. Now we multiply both sides of the
equations $(\ref{z15})$ formally by $b_{x_1}^{e_1}\cdots
b_{x_{t-1}}^{e_{t-1}}$ from the left and by
$b_{x_{t+2}}^{e_{t+2}}\cdots b_{x_z}^{e_z}$ from the right and get
an equation of the form $(\ref{z17})$ with
\begin{eqnarray}
w_0&:=&b_{x_1}^{e_1}\cdots
b_{x_{t-1}}^{e_{t-1}}b_{x_{t+1}}^{e_{t+1}}b_{x_t}^{e_t}b_{x_{t+2}}^{e_{t+2}}\cdots
b_{x_z}^{e_z}\hspace{1cm}\hbox{ and}\label{z52}\\
c_i&:=&a_{\gamma_i}e_te_{t+1}\notag\\
w_i&:=&b_{x_1}^{e_1}\cdots b_{x_{t-1}}^{e_{t-1}}
b_{x_t}^{\frac{e_t-1}{2}}b_{x_{t+1}}^{\frac{e_{t+1}-1}{2}}{\bf
b}^{\gamma_i}b_{x_{t+1}}^{\frac{e_{t+1}-1}{2}}b_{x_t}^{\frac{e_t-1}{2}}
b_{x_{t+2}}^{e_{t+2}}\cdots b_{x_z}^{e_z}\label{z53}
\end{eqnarray}
for an arbitrary (but fixed) ordering $(\gamma_i)_{i\geq 1}$ of
the set $\mathbb{N}_0^d$. The properties $(\ref{z50})$ follow from
Proposition \ref{qa} and the convergence of the expansion of
$b_{x_t}b_{x_{t+1}}-b_{x_{t+1}}b_{x_t}$.

We define $\sum_{i=1}^{\infty}c_{n+1,i}w_{n+1,i}$ by rearranging the
sum
\begin{equation*}
\sum_{i=1}^{s_n-1}c_{n,i}w_{n,i}+c_{n,s_n}w_0+\sum_{i=1}^{\infty}c_{n,s_n}c_iw_i+\sum_{i=s_n+1}^{\infty}c_{n,i}w_{n,i}
\end{equation*}
in norm decreasing order. Moreover, we may assume that
$c_{n+1,i}w_{n+1,i}=c_{n,i}w_{n,i}$ for any $1\leq i< s_n$ and
$c_{n+1,s_n}w_{n+1,s_n}=c_{n,s_n}w_0$ as these terms are already
in norm decreasing order. The description of the $n$th step of the algorithm ends here.

We claim that the formal sum
\begin{equation}
\sum_{i=1}^{\infty}c_{\infty,i}w_{\infty,i}:=\sum_{i=1}^{\infty}\lim_{n\rightarrow\infty}c_{n,i}w_{n,i}\label{z48}
\end{equation}
makes sense and it equals the Laurent series expansion of ${\bf
b}^{\alpha}{\bf b}^{\beta}$ in $D_{[p^{-1/2},1)}(G,K)$ (and hence
in any of the rings $ D_{[\rho,1)}(G,K)$). By construction one of
the following could happen after the $n$th step:
\begin{enumerate}
\item $w_0$ is bad, $s_n=s_{n+1}$, and $w_{n+1,s_{n+1}}=w_0$.
\item $s_{n+1}>s_n$.
\end{enumerate}
The first case can only happen finitely many times in a line. Indeed, if we
define
\begin{equation*}
D(w_{s_n}):=``\hbox{the number of pairs }1\leq i<j\leq z\hbox{
such that }x_i>x_j\hbox{''}
\end{equation*}
then $D(w_{n,s_n})>D(w_{n+1,s_{n+1}})$ in the first case. However,
$D(w_{n+1,s_{n+1}})=0$ is equivalent to that $w_{n+1,s_{n+1}}$ is
not a bad word. It follows that the second case appears infinitely
many times, in other words $s_n$ tends to infinity. This shows
that $\lim_{n\rightarrow\infty}c_{n,i}w_{n,i}$ exists for any $i$
since $c_{n,i}w_{n,i}=c_{n+1,i}w_{n+1,i}$ whenever $i<s_n$.
Moreover, in the sum $\sum_ic_{n,i}w_{n,i}$ for any fixed $n$
there are only finitely many terms with the same norm. On the
other hand, the new terms $c_{n,s_n}c_iw_i$ satisfy
\begin{equation*}
\|c_{n,s_n}c_iw_i\|_{p^{-1/2}}<\|c_{n,s_n}w_{n,s_n}\|_{p^{-1/2}}
\end{equation*}
by $(\ref{z50})$. This implies that for any integer $n_0$ there is
an $n>n_0$ such that
$\|c_{n,s_n}w_{n,s_n}\|_{p^{-1/2}}>\|c_{n+1,s_{n+1}}w_{n+1,s_{n+1}}\|_{p^{-1/2}}$
and hence the sum $(\ref{z48})$ is convergent in
$D_{[p^{-1/2},1)}(G,K)$ because the range of the norm
$\|\cdot\|_{p^{-1/2}}$ is discrete and the terms in each sum
$\sum_{i=1}^{\infty}c_{n,i}w_{n,i}$ are in norm-decreasing order.
Moreover, none of the terms in this sum are bad (by construction)
therefore this is the Laurent series expansion of ${\bf
b}^{\alpha}{\bf b}^{\beta}$ since by construction we have the
equality
\begin{equation*}
\sum_{i=1}^{\infty}c_{n,i}w_{n,i}=\sum_{i=1}^{\infty}c_{n+1,i}w_{n+1,i}
\end{equation*}
in $D_{[p^{-1/2},1)}(G,K)$ for any positive integer $n$.

Now we show $(\ref{z20})$ by using the above algorithm. Put
\begin{eqnarray*}
f(cw)&:=&\log_{p^{-1/2}}\|cw\|_{p^{-1/2}}+\max_{1\leq
j\leq d} m_j(w)\in\mathbb{Z},\\
f_n&:=&\inf_{i\geq 1}f(c_{n,i}w_{n,i})\in \mathbb{Z}\cup\{-\infty\},\hspace{0.5cm}\hbox{ and}\\
f_{\infty}&:=&\inf_{i\geq 1}f(c_{\infty,i}w_{\infty,i})\in \mathbb{Z}\cup\{-\infty\}
\end{eqnarray*}
where $m_j$ is the homomorphism from the free group
$F(b_1,\dots,b_d)$ to $\mathbb{Z}$ sending $b_j$ to $1$ and all
the other generators to $0$ (in other words $m_j(w)$ is the degree
of $w$ in the variable $b_j$). We claim that for any $n$ we have
$f_n\leq f_{n+1}$. In particular all these $f_n$ are finite. For
this it suffices to show that we have
\begin{equation}
f(c_{n,s_n}w_{n,s_n})=f(c_{n,s_n}w_0)\leq
f(c_{n,s_n}c_iw_i).\label{z49}
\end{equation}
as the new terms in $\sum_{i=1}^{\infty}c_{n+1,i}w_{n+1,i}$ are
exactly $c_{n,s_n}w_0$ and $c_{n,s_n}c_iw_i$ ($i\geq 1$). The
equality part of $(\ref{z49})$ is obvious as
$m_j(w_{n,s_n})=m_j(w_0)$ for any $j$ and
\begin{equation*}
\|c_{n,s_n}w_{n,s_n}\|_{p^{-1/2}}=\|c_{n,s_n}w_0\|_{p^{-1/2}}\hspace{1cm}(\hbox{see
equations }(\ref{z51})\hbox{ and }(\ref{z52})).
\end{equation*}
The inequality follows from the facts that (see equations
$(\ref{z50})$ and $(\ref{z53})$)
\begin{eqnarray*}
\|c_{n,s_n}c_iw_i\|_{p^{-1/2}}<\|c_{n,s_n}w_{n,s_n}\|_{p^{-1/2}},\\
m_j(w_i)=m_j\left(b_{x_1}^{e_1}\cdots b_{x_{t-1}}^{e_{t-1}}
b_{x_t}^{\frac{e_t-1}{2}}b_{x_{t+1}}^{\frac{e_{t+1}-1}{2}}{\bf
b}^{\gamma_i}b_{x_{t+1}}^{\frac{e_{t+1}-1}{2}}b_{x_t}^{\frac{e_t-1}{2}}
b_{x_{t+2}}^{e_{t+2}}\cdots b_{x_z}^{e_z}\right)\geq
m_j(w_{n,s_n})-1
\end{eqnarray*}
(since $\gamma_i$ is in $\mathbb{N}_0^d$), and the range of $f$ is
in $\mathbb{Z}$. We also obtain $f_n\leq f_{\infty}$ for any
$n\geq 1$ by taking the limit. In particular we have
\begin{eqnarray*}
|\alpha|+|\beta|+N=f_1\leq f_{\infty}&\leq& f\left(({\bf b}^{\alpha},{\bf b}^{\beta})\right)=\log_{p^{-1/2}}|({\bf b}^{\alpha},{\bf b}^{\beta})|\hspace{0.5cm}\hbox{ whence}\\
p^{-\frac{|\alpha|+|\beta|+N}{2}}&\geq&p^{-\frac{f\left(({\bf b}^{\alpha},{\bf b}^{\beta})\right)}{2}}=|({\bf b}^{\alpha},{\bf b}^{\beta})|
\end{eqnarray*}
and we are done.
\end{proof}

\begin{rem}
The above algorithm computing the Laurent series expansion of
${\bf b}^\alpha{\bf b}^\beta$ can actually be used to
\emph{define} the rings $ D_{[\rho,1)}(G,K)$ (for $p^{-1}<\rho<1$
in $p^{\mathbb{Q}}$) by the underlying vector space of formal
Laurent series together with the multiplication given by this
algorithm. One, of course, has to check the associative and
distributive laws. However, if we think of $ D_{[\rho,1)}(G,K)$ as
a subring of the completion $Q_\rho(G,K)$ of the field of
fractions $Q(D_{\rho}(G,K))$ of the algebra $D_\rho(G,K)$ then we
obtain immediately that $ D_{[\rho,1)}(G,K)$ satisfies the ring
axioms. Indeed, $D_\rho(G,K)$ is a noetherian domain, so it has a
field of fractions. Moreover, the norm $\|\cdot\|_\rho$ is
multiplicative on $D_\rho(G,K)$, so it can be extended
multiplicatively to the field $Q(D_{\rho}(G,K))$ of fractions and
we can take the completion $Q_\rho(G,K)$ with respect to
$\|\cdot\|_\rho$. The Laurent series expansions of elements in $
D_{[\rho,1)}(G,K)$ will all be convergent in $Q_\rho(G,K)$, so $
D_{[\rho,1)}(G,K)$ is naturally a subring of $Q_\rho(G,K)$.
\end{rem}

\begin{lem}\label{z22}
For any index $\alpha$ in $\mathbb{Z}^d$ and any $p^{-1}<\rho<1$
there is an element $f^{(\alpha)}$ in $ D_{[\rho,1)}(G,K)$ such
that
\begin{equation*}
(f^{(\alpha)},{\bf b}^{\beta})=\delta_{\alpha\beta}
\end{equation*}
for all $\beta$ in $\mathbb{Z}^d$. Here $\delta_{\alpha\beta}$
denotes the Kronecker delta, ie.\ $\delta_{\alpha\beta}=1$ if
$\alpha=\beta$ and $0$ otherwise.
\end{lem}
\begin{proof}
We may assume without loss of generality that $K=\mathbb{Q}_p$. We
are going to give the power series expansion of $f^{(\alpha)}$ by
approximating it in the ring $D_{(p^{-1},1)}(G,K)$. We start with
the term $f_0^{(\alpha)}:={\bf b}^{-\alpha}$. As the coefficients
of ${\bf b}^{-\alpha}{\bf b}^{\beta}$ are in $\mathbb{Z}_p$ we
have
\begin{equation*}
|(f_0^{(\alpha)},{\bf b}^{\beta})-\delta_{\alpha\beta}|\leq 1.
\end{equation*}
Let us denote by $C_0$ the set of $\gamma$ such that
\begin{equation*}
|(f_0^{(\alpha)},{\bf b}^{\gamma})-\delta_{\alpha\gamma}|=1.
\end{equation*}
Note that this set may well be nonempty and even infinite when $G$
is noncommutative. However, by Lemma \ref{z21} $(i)$ all these
$\gamma$ have degree $|\gamma|<|\alpha|$, and by $(ii)$ for any
fixed degree there are only finitely many of them. Hence we can
remove
\begin{equation}
\frac{(f_0^{(\alpha)},{\bf b}^{\gamma})}{({\bf b}^{-\gamma},{\bf
b}^{\gamma})}{\bf b}^{-\gamma}\label{z77}
\end{equation}
from $f_0^{(\alpha)}$ at first for all the $\gamma$ in $C_0$ with
maximal degree. Since $({\bf b}^{-\gamma},{\bf b}^{\gamma})$ is in
$\mathbb{Z}_p^{\times}$ the coefficient of $(\ref{z77})$ is in
$\mathbb{Z}_p$. Note that other ``bad'' $\gamma$ may as well
arise, but all of them have strictly smaller degree and for any
fixed degree there are only finitely many of them. Henceforth we
get a convergent Laurent series in $ D_{[\rho,1)}(G,K)$ (for any
$p^{-1}<\rho<1$)
\begin{eqnarray}
f_1^{(\alpha)}&:=&f_0^{(\alpha)}-\sum_{i=1}^{\infty}
a_{\gamma_{1,i}}{\bf
b}^{-\gamma_{1,i}}\hspace{0.5cm}\hbox{ such that}\notag\\
|(f_1^{(\alpha)},{\bf b}^\beta)-\delta_{\alpha\beta}|&\leq&
p^{-1}\hspace{0.5cm}\hbox{ for any }\beta\in\mathbb{Z}^d\hspace{0.5cm}\hbox{ and }\notag\\
\|f_1^{(\alpha)}-f_0^{(\alpha)}\|_{\rho}&\leq&\rho\|f_0^{(\alpha)}\|_{\rho}\label{z76}
\end{eqnarray}
since $|\gamma_{1,i}|<|\alpha|$ tends to $-\infty$ with
$i\rightarrow\infty$. Here we choose recursively $\gamma_{1,i}$
with maximal degree so that
\begin{equation*}
|(f_0^{(\alpha)}-\sum_{j=1}^{i-1} a_{\gamma_{1,j}}{\bf
b}^{-\gamma_{1,j}},{\bf
b}^{\gamma_{1,i}})-\delta_{\alpha\gamma_{1,i}}|=1
\end{equation*}
and put
\begin{equation*}
a_{\gamma_{1,i}}:=\frac{(f_0^{(\alpha)}-\sum_{j=1}^{i-1}
a_{\gamma_{1,j}}{\bf b}^{-\gamma_{1,j}},{\bf
b}^{\gamma_{1,i}})}{({\bf b}^{-\gamma_{1,i}},{\bf
b}^{\gamma_{1,i}})}\in\mathbb{Z}_p.
\end{equation*}
Now we use the above method in order to modify
$f_{n-1}^{(\alpha)}$ in the $n$th step ($n\geq 2$) by a Laurent
series with coefficients in $p^{n-1}\mathbb{Z}_p$ and minimal
degree $-|\alpha|-n+1$, such that
\begin{eqnarray}
|(f_n^{(\alpha)},{\bf b}^\beta)-\delta_{\alpha\beta}|&\leq&
p^{-n}\hbox{ and}\notag\\
\|f_n^{(\alpha)}-f_{n-1}^{(\alpha)}\|_{\rho}&\leq&\frac{1}{p^{n-1}}\rho^{-|\alpha|-n+1}.\label{z75}
\end{eqnarray}
Indeed, we choose $\gamma_{n,i}$ with maximal degree so that
\begin{equation*}
|(f_{n-1}^{(\alpha)}-\sum_{j=1}^{i-1} a_{\gamma_{n,j}}{\bf
b}^{-\gamma_{n,j}},{\bf
b}^{\gamma_{n,i}})-\delta_{\alpha\gamma_{n,i}}|=p^{-n+1}
\end{equation*}
and put
\begin{eqnarray}
f_n^{(\alpha)}&:=&f_{n-1}^{(\alpha)}-\sum_{i=1}^{\infty}a_{\gamma_{n,i}}{\bf
b}^{-\gamma_{n,i}}\hbox{ where}\notag\\
a_{\gamma_{n,i}}&:=&\frac{(f_{n-1}^{(\alpha)}-\sum_{j=1}^{i-1}
a_{\gamma_{n,j}}{\bf b}^{-\gamma_{n,j}},{\bf
b}^{\gamma_{n,i}})}{({\bf b}^{-\gamma_{n,i}},{\bf
b}^{\gamma_{n,i}})}\in p^{n-1}\mathbb{Z}_p.\notag
\end{eqnarray}
So by $(\ref{z75})$ the sequence $f_n^{(\alpha)}$ is Cauchy in
each norm $\|\cdot\|_{\rho}$ and therefore defines an element
\begin{equation*}
f^{(\alpha)}:=\lim_{n\rightarrow\infty}f_n^{(\alpha)}
\end{equation*}
in $ D_{[\rho,1)}(G,K)$ with the required property for any
$\rho>p^{-1}$. Moreover, one has
\begin{equation*}
\|f^{(\alpha)}-{\bf b}^{-\alpha}\|_{\rho}<\|{\bf
b}^{-\alpha}\|_{\rho}
\end{equation*}
for any $p^{-1}<\rho<1$ by $(\ref{z76})$ and $(\ref{z75})$. Note
that the difference between $(\ref{z76})$ and $(\ref{z75})$ for
$n=1$ comes from the fact that we have a better estimate in Lemma
\ref{z21} $(i)$ whenever $|\alpha|+|\beta|=0$.
\end{proof}

\begin{lem}\label{z23}
The ``Laurent polynomials'' are dense in each $ D_{[\rho,1)}(G,K)$
(in both the Fr\'echet and the weak topologies). In particular
they form a dense subset of $\mathcal{R}(G,K)$ in both the
inductive limit topology and the weak topology.
\end{lem}
\begin{proof}
The density in the Fr\'echet, resp.\ the inductive limit topology
follows easily from the definition. On the other hand the maps
$(\ref{z19})$ are continuous in the Fr\'echet, resp.\ the
inductive limit topology. So any subset of $R$ that is open in the
weak topology is also open in the Fr\'echet, resp. the inductive
limit topology hence the result.
\end{proof}

\begin{cor}\label{z25}
Let $R$ denote one of the rings $ D_{[\rho,1)}(G,K)$ or
$\mathcal{R}(G,K)$. A series $\sum_{\alpha}c_{\alpha}f^{(\alpha)}$
is convergent in $R$ (in either of the topologies) if and only if
$\sum_{\alpha} c_{\alpha}{\bf b}^{-\alpha}$ defines an element of
$R$. Moreover, any element of $R$ can be written in this form. In
particular, the weak topology on $R$ is symmetric, in other words
it can also be defined by right multiplication by $x_0$.
\end{cor}
\begin{proof}
If $\sum_{\alpha} c_{\alpha}{\bf b}^{-\alpha}$ is an element of
some $ D_{[\rho,1)}(G,K)$ then so is $\sum_{\alpha}
c_{\alpha}f^{(\alpha)}$ as we have
\begin{equation*}
\|f^{(\alpha)}\|_r=\|{\bf b}^{-\alpha}\|_r
\end{equation*}
for any $p^{-1}< r<1$ in $p^{\mathbb{Q}}$. So the series
$\sum_{\alpha} c_{\alpha}f^{(\alpha)}$ is convergent in either of
the topologies on $R$. On the other hand we have
\begin{equation*}
\|f^{(\alpha)}-{\bf b}^{-\alpha}\|_r<\|{\bf b}^{-\alpha}\|_r
\end{equation*}
therefore we can write any element in $ D_{[\rho,1)}(G,K)$ in the
required form. This shows that the natural maps
\begin{eqnarray*}
 D_{[\rho,1)}(G,K)&\rightarrow&  D_{[\rho,1)}(\mathbb{Z}_p^d,K)\\
\mathcal{R}(G,K)&\rightarrow& \mathcal{R}(\mathbb{Z}_p^d,K)
\end{eqnarray*}
defined by the identification of the two underlying vector spaces
on which the rings are defined is a homeomorphism in the weak
topology on both sides. The result follows by the similar
statement for the other weak topology defined by right
multiplication by $x_0$.
\end{proof}

Now we can state and prove our main theorem in this section.

\begin{thm}\label{z27}
The rings $ D_{[\rho,1)}(G,K)$ and $\mathcal{R}(G,K)$ are
self-dual topological $K$-algebras (in their own weak topology)
with respect to the perfect pairing
\begin{eqnarray}
R\times R&\rightarrow& K\notag\\
(x({\bf b}),y({\bf b}))&\mapsto&"\hbox{the constant term of
}x({\bf b})y({\bf b})."\notag
\end{eqnarray}
Here $R$ denotes either of the above rings.
\end{thm}
\begin{proof}
The above pairing is clearly $K$-bilinear and separately
continuous. So it remains to prove that it induces a topological isomorphism
\begin{equation}
R\rightarrow\Hom_K^{ct}(R,K)\label{z16}
\end{equation}
with the weak topology on the dual space. Moreover, by the choice of the
topologies on both sides, it remains to show that
\eqref{z16} is a bijection.

First we need to prove that the map $(\ref{z16})$ is injective, in
other words for any $x\neq 0$ in $R$ there exists a $y$ in $R$
such that $xy$ has nonzero constant term. For this fix a
$p^{-1}<r<1$ in $p^{\mathbb{Q}}$ such that $x$ is in
$D_{[r,1)}(G,K)$. Write
\begin{equation*}
x({\bf b})=\sum_{\alpha\in\mathbb{Z}^d}x_{\alpha}{\bf b}^\alpha
\end{equation*}
and let $\beta$ be such that $\|x\|_r=|x_{\beta}|r^{|\beta|}$. We
claim that we may choose $y={\bf b}^{-\beta}$. Indeed, we have
\begin{equation*}
\|x({\bf b})y({\bf b})-\sum_{\alpha}x_{\alpha}{\bf
b}^{\alpha-\beta}\|_r<\|x({\bf b})y({\bf b})\|_r=|x_{\beta}|.
\end{equation*}
So in particular the constant term of $x({\bf b})y({\bf b})$
differs from $x_{\beta}$ by something with absolute value less
than $|x_{\beta}|$, hence it is nonzero.

For the surjectivity let $\lambda\colon R\rightarrow K$ be $K$-linear and
continuous with respect to the weak topology on $R$ and the natural topology
on $K$. Then $U:=\{y\in R\mid |\lambda(y)|_K<1\}$ is open in $R$. Thus, there
are $x_1,\dots,x_s\in R$, and there is $\varepsilon>0$ such that
\begin{equation*}
\bigcap_{i=1}^s\{y\in R\mid |''\hbox{constant term of
}x_iy''|_K<\varepsilon\}\subset U. 
\end{equation*}
Put $\lambda_{x_i}(y):=$''constant term of $x_iy$''. Consider $y\in
V:=\bigcap_{i=1}^s\Ker(\lambda_{x_i})$. Then $Ky\subset V\subset U$, and
$|\lambda(cy)|<1$ for any $c\in K$ implies $\lambda(y)=0$. Therefore $\lambda$
factors through $R/V$. Next, the map $R/V\rightarrow K^s$,
$y+V\mapsto(\lambda_{x_1}(y),\dots,\lambda_{x_s}(y))$, is injective showing
that $\lambda$ is a linear combination
$a_1\lambda_{x_1}+\dots+a_s\lambda_{x_s}$ of the $\lambda_{x_i}$, and thus
$\lambda=\lambda_x$ with $x=a_1x_1+\dots+a_sx_s$.
\end{proof}

\subsection{The nice topology}

The problem with the weak topology is that it is artificially
constructed forcing the ring to be self-dual. Moreover, its
restriction to $D(G,K)$---even in the Robba-ring case---does not
give back the Fr\'echet topology on $D(G,K)$, since any open
lattice in the weak topology contains a subspace of finite
codimension. However, we are going to construct a third topology
on the generalized Robba ring $\mathcal{R}(G,K)$, the ``nice''
topology which has the following properties.

\begin{enumerate}[$(i)$]
\item It is stronger than the weak topology, but weaker than the
inductive limit topology.
\item The ring $\mathcal{R}(G,K)$ is self-dual in the nice
topology with respect to the pairing
\begin{eqnarray}
\mathcal{R}(G,K)\times\mathcal{R}(G,K)&\rightarrow& K\label{z56}\\
(x,y)&\mapsto&``\hbox{the constant term of }xy."\notag
\end{eqnarray}
\item On the classical Robba ring $\mathcal{R}(\mathbb{Z}_p,K)$
the nice topology coincides with the inductive limit topology.
\item The subspace topology on the distribution algebra $D(G,K)$
equals the usual Fr\'echet topology.
\end{enumerate}

The construction is the following. For any fixed element
\begin{equation*}
x_0=\sum_{\alpha\in\mathbb{Z}^d} c_{0,\alpha}{\bf b}^{\alpha}
\end{equation*}
in $\mathcal{R}(G,K)$ we take the $o_K$-submodule
\begin{equation}
L_{x_0}:=\left\{\sum_{\alpha\in\mathbb{Z}^d} d_{\alpha}{\bf
b}^{\alpha} \hbox{ such that } |c_{0,-\alpha}||d_{\alpha}|\leq
1\hbox{ for any } \alpha\in\mathbb{Z}^d\right\}\label{z54}
\end{equation}
in $\mathcal{R}(G,K)$.
\begin{lem}
The set $L_{x_0}$ defined in $(\ref{z54})$ is a lattice in
$\mathcal{R}(G,K)$.
\end{lem}
\begin{proof}
We need to show that for any nonzero element $x=\sum
c_{\alpha}{\bf b}^{\alpha}$ in $\mathcal{R}(G,K)$ there exists a
constant $c$ in the ring of integers $o_K$ such that $cx$ lies in
$L_{x_0}$. We can take the product $x\circ x_0$ in the commutative
ring $\mathcal{R}(\mathbb{Z}_p^d,K)$ which has the same underlying
set as $\mathcal{R}(G,K)$. The constant term of $x\circ x_0$ is by
definition
\begin{equation*}
\sum_{\alpha\in\mathbb{Z}^d}c_{\alpha}c_{0,-\alpha},
\end{equation*}
so this sum is convergent and hence bounded. Therefore there
exists a constant $c$ in $o_K$ such that the element $cx$ lies in
$L_{x_0}$. Note that this is a stronger assumption than just that
the constant term of $cx\circ x_0$ be integral.
\end{proof}
Now we define a lattice open in the nice topology if it contains a
lattice of the form $L_{x_0}$. This gives a locally convex
topology on $\mathcal{R}(G,K)$ since for any elements
$x_1,\dots,x_m$ in $\mathcal{R}(G,K)$ there exists an element
$x_0$ in $\mathcal{R}(G,K)$ such that we have
\begin{equation*}
L_{x_1}\cap\dots\cap L_{x_m}=L_{x_0}.
\end{equation*}
Indeed, we can take $x_0$ such that its coefficients satisfy
\begin{equation*}
|c_{0,\alpha}|=\max_{1\leq i\leq m}|c_{i,\alpha}|
\end{equation*}
where $c_{i,\alpha}$ is the coefficient of ${\bf b}^\alpha$ in
$x_i$.

\begin{lem}\label{z59}
The nice topology on $\mathcal{R}(G,K)$ is stronger than the weak
topology, but weaker than the inductive limit topology.
\end{lem}
\begin{proof}
The lattice $L_{x_0}$ is certainly contained in the lattice
\begin{equation*}
\{x\in \mathcal{R}(G,K)\colon |(f_0,x)|\leq 1\}
\end{equation*}
where $f_0=\sum_{\alpha}c_{0,-\alpha}f^{(\alpha)}$ with
$x_0=\sum_{\alpha}c_{0,\alpha}{\bf b}^{\alpha}$ which form a base
of neighbourhood of the origin in the weak topology.

For the second statement choose real numbers $\rho<\rho_1<1$ such
that $x_0$ is in $ D_{[\rho,1)}(G,K)$. Then by definition
$(\ref{z54})$ $L_{x_0}\cap D_{[\rho,1)}(G,K)$ contains the lattice
\begin{equation*}
\left\{x\in \mathcal{R}(G,K)\colon \|x\|_{\rho,\rho_1}\leq
\frac{1}{\|x_0\|_{\rho,\rho_1}}\right\}
\end{equation*}
showing that $L_{x_0}$ is open in the inductive limit topology of
$\mathcal{R}(G,K)$.
\end{proof}

\begin{lem}\label{r2}
The map
\begin{equation}
\mathcal{R}(G,K)\rightarrow\Hom_K^{ct,nice}(\mathcal{R}(G,K),K)\label{z57}
\end{equation}
induced by the pairing $(\ref{z56})$ is surjective.
\end{lem}
\begin{proof}
First of all one needs to verify that the pairing $(\ref{z56})$ is
separately continuous. Indeed, this pairing is separately
continuous even in the weak topology and the nice topology is
stronger.

The proof of the surjectivity is similar to the proof of Theorem
\ref{z27}. Take a continuous linear functional $\varphi$ from
$\mathcal{R}(G,K)$ to $K$. Since $\varphi$ is continuous in the
nice topology there exists an element $y$ in $\mathcal{R}(G,K)$
such that $\varphi^{-1}(o_K)\supseteq L_y$. Write $y$ in the form
\begin{equation*}
y=\sum_{\alpha\in\mathbb{Z}^d}y_{\alpha}{\bf b}^{\alpha}.
\end{equation*}
By definition of $L_y$ we have $|\varphi({\bf b}^{\alpha})|\leq
y_{-\alpha}$ for all $\alpha\in\mathbb{Z}^d$. This means that the
series
\begin{equation*}
\sum_{\alpha\in\mathbb{Z}^d}\varphi({\bf b}^{\alpha})f^{(\alpha)}
\end{equation*}
defines an element $f_{\varphi}$ in $\mathcal{R}(G,K)$ such that
$f_{\varphi}$ maps to the functional $\varphi$ under the map
$(\ref{z57})$.
\end{proof}

\begin{lem}\label{z70}
The subspace topology on $D(G,K)$ in the nice topology coincides
with its canonical Fr\'echet topology.
\end{lem}
\begin{proof}
Since the restriction of the inductive limit topology of
$\mathcal{R}(G,K)$ on $D(G,K)$ equals its canonical topology it
suffices to show that whenever $L\subset D(G,K)$ is an open
lattice in the Fr\'echet topology then there exists an element
$x(L)$ in $\mathcal{R}(G,K)$ such that $L_{x(L)}\cap
D(G,K)\subseteq L$. Now since $L$ is open in the Fr\'echet
topology there exist a real numbers $p^{-1}<\rho<1$ and an integer
$l$ such that $y$ is in $L$ whenever $\|y\|_{\rho}\leq p^{-l}$ for
the element $y$ in $D(G,K)$. Now we define
\begin{eqnarray}
x&:=&\sum_{\alpha\in\mathbb{N}_0^d}x_{-\alpha}{\bf
b}^{-\alpha}\in\mathcal{R}(G,\mathbb{Q}_p)\hbox{ with}\label{z58}\\
p^l\rho^{|\alpha|}\leq&|x_{-\alpha}|&<p^{l+1}\rho^{|\alpha|}.\notag
\end{eqnarray}
Note that $(\ref{z58})$ determines $x_{-\alpha}$ in $\mathbb{Q}_p$
up to a unit in $\mathbb{Z}_p^{\times}$---otherwise we choose
$x_{-\alpha}$ arbitrarily. Now we see that $x$ is an element in
$D_{(\rho,1)}(G,K)$ and we have $L_{x(L)}\cap D(G,K)\subseteq L$
by construction.
\end{proof}

\begin{pro}
The multiplication of the ring $\mathcal{R}(G,K)$ is separately
continuous in the nice topology.
\end{pro}
\begin{proof}
By the left-right symmetry of the nice topology it suffices to
show that for any fixed elements $x_0$ and $y_0$ in
$\mathcal{R}(G,K)$ the pre-image $y_0^{-1}L_{x_0}$ of the lattice
$L_{x_0}$ under the left multiplication by $y_0$ is open.
$y_0^{-1}L_{x_0}$ is clearly a lattice in $\mathcal{R}(G,K)$.
Moreover, the Laurent polynomials are dense in $\mathcal{R}(G,K)$
in the inductive limit topology and $y_0^{-1}L_{x_0}$ is open (and
hence closed) in the inductive limit topology by the continuity of
the multiplication by $y_0$. Therefore it suffices to show that
there exists an element
\begin{equation*}
z_0=\sum_{\alpha\in\mathbb{Z}^d}a_{\alpha}{\bf b}^{\alpha}
\end{equation*}
in $\mathcal{R}(G,K)$ such that for any $\alpha$
\begin{eqnarray}
|a_{\alpha}|s_{\alpha}&\geq& 1\hspace{1cm}\hbox{where}\label{z55}\\
s_{\alpha}&:=&\sup\{|s|\colon s{\bf b}^{-\alpha}\in
y_0^{-1}L_{x_0}\}\in\mathbb{R}^{>0}\cup\{+\infty\}.\label{z74}
\end{eqnarray}
Indeed, if $z_0$ satisfies $(\ref{z55})$ then $L_{z_0}\subseteq
y_0^{-1}L_{x_0}$. This condition is equivalent to that there
exists a real number $p^{-1}<\rho<1$ such that the set
\begin{equation*}
\left\{\frac{\rho_1^{|\alpha|}}{s_{\alpha}}\right\}_{\alpha\in\mathbb{Z}^d}
\end{equation*}
converges to $0$ for any $1>\rho_1\geq\rho$ because this is the
convergence condition for the coefficients $a_{\alpha}$ so that
$z_0$ lies in $ D_{[\rho,1)}(G,K)$. We determine the real number
$\rho$ the following way. By Lemma \ref{z59} $L_{x_0}$ is open in
the inductive limit topology of $\mathcal{R}(G,K)$, hence for any
fixed $p^{-1}<\rho_2<1$ there exists a $\rho_2<\rho_3<1$ such that
$L_{x_0}\cap D_{[\rho_2,1)}(G,K)$ contains the lattice
\begin{equation}
\left\{x\in D_{[\rho_2,1)}(G,K)\colon \|x\|_{\rho_2,\rho_3}\leq
c(x_0,\rho_2,\rho_3)\right\}\label{z60}
\end{equation}
where $c(x_0,\rho_2,\rho_3)$ is a positive real number depending
on $x_0$, $\rho_2$, and $\rho_3$. We pick $\rho_2$ such that both
$y_0$ and $x_0$ lie in $D_{[\rho_2,1)}(G,K)$ and claim that we can
take any $\rho$ which is bigger than $\rho_3$. To show this fix
$\rho_1>\rho_3$ and choose a real number $\varepsilon>0$. We need
to verify that for all but finitely many $\alpha$ in
$\mathbb{Z}^d$ we have
\begin{equation*}
\frac{\rho_1^{|\alpha|}}{s_{\alpha}}<\varepsilon.
\end{equation*}
Equivalently, by definition $(\ref{z74})$ of $s_\alpha$ we need to
find a coefficient $c_{-\alpha}$ in $\mathbb{Q}_p$ for all but
finitely many $\alpha$ in $\mathbb{Z}^d$ with the properties
\begin{enumerate}[$(i)$]
\item $c_{-\alpha}y_0{\bf b}^{-\alpha}$ lies in $L_{x_0}$, and
\item $|c_{-\alpha}|>\frac{\rho_1^{|\alpha|}}{\varepsilon}$.
\end{enumerate}

At first we show that there exists an integer $n_0$ such that
whenever $|\alpha|<n_0$ then there is a $c_{-\alpha}$ with
properties $(i)$ and $(ii)$. Indeed, we choose $n_0<0$ small
enough so that
\begin{equation}
\left(\frac{\rho_1}{\rho_3}\right)^{n_0}\leq \frac{\varepsilon
c(x_0,\rho_2,\rho_3)}{p\|y_0\|_{\rho_2,\rho_3}}.\label{z61}
\end{equation}
Since $\rho_1>\rho_3$ there exists such an $n_0$. Moreover,
whenever $|\alpha|<n_0$ then we have
\begin{equation*}
p\frac{\rho_1^{|\alpha|}}{\varepsilon}\|y_0{\bf
b}^{-\alpha}\|_{\rho_2,\rho_3}\leq
p\frac{\rho_1^{|\alpha|}}{\varepsilon}\|y_0\|_{\rho_2,\rho_3}\|{\bf
b}^{-\alpha}\|_{\rho_2,\rho_3}=p\frac{\rho_1^{|\alpha|}}{\varepsilon}\|y_0\|_{\rho_2,\rho_3}\rho_3^{-|\alpha|}\leq
c(x_0,\rho_2,\rho_3),
\end{equation*}
whence we may take any $c_{-\alpha}$ such that
\begin{equation}
\frac{\rho_1^{|\alpha|}}{\varepsilon}<|c_{-\alpha}|\leq
p\frac{\rho_1^{|\alpha|}}{\varepsilon}\label{z62}
\end{equation}
and $c_{-\alpha}y_0{\bf b}^{-\alpha}$ will lie in $L_{x_0}$ by
$(\ref{z60})$. Indeed, there exists a $c_{-\alpha}$ in
$\mathbb{Q}_p$ with the required absolute value in $(\ref{z62})$.

On the other hand we now prove that there also exists an integer
$n_1>0$ such that there is a $c_{-\alpha}$ with properties $(i)$
and $(ii)$ whenever $|\alpha|>n_1$. For this we take a real number
$1>\rho_4>\rho_1$. The lattice $L_{x_0}\cap D_{[\rho_4,1)}(G,K)$
is open in the Fr\'echet topology of $ D_{[\rho_4,1)}(G,K)$ by
Lemma \ref{z59}. This means that there exist real numbers
$1>\rho_5>\rho_4$ and $c(x_0,\rho_4,\rho_5)$ such that the lattice
\begin{equation*}
\left\{x\in D_{[\rho_4,1)}(G,K)\colon \|x\|_{\rho_4,\rho_5}\leq
c(x_0,\rho_4,\rho_5)\right\}
\end{equation*}
is contained in $L_{x_0}\cap D_{[\rho_4,1)}(G,K)$. We choose
$n_1$---in a similar way as above in $(\ref{z61})$---such that
\begin{equation*}
\left(\frac{\rho_1}{\rho_4}\right)^{n_0}\leq \frac{\varepsilon
c(x_0,\rho_4,\rho_5)}{p\|y_0\|_{\rho_4,\rho_5}}.
\end{equation*}
Then by the same argument as above there exists a required
$c_{-\alpha}$ for any $\alpha$ with degree bigger than $n_1$.

Now take an integer $n_2$ with $n_0\leq n_2\leq n_1$. It suffices
to show that there exists a $c_{-\alpha}$ satisfying $(i)$ and
$(ii)$ for all but finitely many $\alpha$ with fixed degree
$|\alpha|=n_2$. At first we write
\begin{equation*}
x_0=\sum_{\gamma\in \mathbb{Z}^d}c_{0,\gamma}{\bf b}^{\gamma}
\end{equation*}
and put
\begin{equation*}
A(x_0,y_0,n_2):=\{\alpha\in\mathbb{Z}^d\colon |\alpha|=n_2\hbox{
and there does not exist a }c_{-\alpha}\hbox{ satisfying
}(i)\hbox{ and }(ii)\}.
\end{equation*}
By construction of $L_{x_0}$ we obtain that
\begin{eqnarray}
L_{x_0}&=&\bigcap_{\gamma\in\mathbb{Z}^d}L_{c_{0,\gamma}{\bf
b}^{\gamma}}\hspace{1cm}\hbox{and}\notag\\
A(x_0,y_0,n_2)&=&\bigcup_{\gamma\in\mathbb{Z}^d}A(c_{0,\gamma}{\bf
b}^{\gamma},y_0,n_2).\label{z64}
\end{eqnarray}
Now we show that for all but finitely many $\gamma$ the set
$A(c_{0,\gamma}{\bf b}^{\gamma},y_0,n_2)$ is empty. Indeed, by
Lemma \ref{z59} $A(c_{0,\gamma}{\bf b}^{\gamma},y_0,n_2)$ is empty
whenever we have
\begin{equation}
p\frac{\rho_1^{n_2}}{\varepsilon}\|y_0\|_{\rho_2,\rho_3}\max(\rho_2^{-n_2},\rho_3^{-n_2})\leq
\frac{1}{\|c_{0,\gamma}{\bf
b}^{\gamma}\|_{\rho_2,\rho_3}}\label{z63}
\end{equation}
because $L_{c_{0,\gamma}{\bf b}^{\gamma}}$ contains the lattice
\begin{equation*}
\left\{x\in D_{[\rho_2,1)}(G,K)\colon\|x\|_{\rho_2,\rho_3}\leq
\frac{1}{\|c_{0,\gamma}{\bf b}^\gamma\|_{\rho_2,\rho_3}}\right\}.
\end{equation*}
On the other hand $x_0$ is in $D_{[\rho_2,1)}(G,K)$ by the choice
of $\rho_2$, so $(\ref{z63})$ is satisfied for all but finitely
many $\gamma$ in $\mathbb{Z}^d$. Therefore we may assume without
loss of generality that $x_0$ is of the form
$x_0=c_{0,\gamma_0}{\bf b}^{\gamma_0}$ for some $\gamma_0$ in
$\mathbb{Z}^d$ and $c_{0,\gamma_0}$ in $K$ since $(\ref{z64})$ is
essentially a finite union. Moreover, we have
\begin{equation*}
(f^{(-\gamma_0)}y_0)^{-1}L_{c_{0,\gamma_0}}=y_0^{-1}L_{c_{0,\gamma_0}{\bf
b}^{\gamma_0}}
\end{equation*}
so we may, as well, assume that $\gamma_0=(0,\dots,0)$. Now we
write
\begin{equation*}
y_0=\sum_{\beta\in\mathbb{Z}^d}y_{0,\beta}{\bf b}^{\beta}.
\end{equation*}
By a similar argument as above for $y_0$ instead of $x_0$ we
get
\begin{equation*}
A(c_{0,\gamma_0},y_0,n_2)=\bigcup_{\beta\in\mathbb{Z}^d}A(c_{0,\gamma_0},y_{0,\beta}{\bf
b}^{\beta},n_2)
\end{equation*}
and for all but finitely many $\beta$ in $\mathbb{Z}^d$ the set
$A(c_{0,\gamma_0},y_{0,\beta}{\bf b}^{\beta},n_2)$ is empty. Hence
we may assume that $y_0=y_{0,\beta_0}{\bf b}^{\beta_0}$ for some
$\beta_0$ in $\mathbb{Z}^d$ and the result follows from Lemma
\ref{z21}$(ii)$.
\end{proof}

So putting everything together we get the following

\begin{cor}\label{r3}
The generalized Robba ring $\mathcal{R}(G,K)$ is a self-dual
topological $K$-algebra in the nice topology.
\end{cor}

\begin{pro}
The nice topology on the classical Robba ring
$\mathcal{R}(\mathbb{Z}_p,K)$ coincides with the inductive limit
topology.
\end{pro}
\begin{proof}
One direction follows from Lemma \ref{z59}. For the other
direction take a lattice $L$ in $\mathcal{R}(\mathbb{Z}_p,K)$
which is open in the inductive limit topology. We need to show
that there exists an element $x_0$ in
$\mathcal{R}(\mathbb{Z}_p,K)$ such that $L_{x_0}\subseteq L$.
Since $L$ is open in the inductive limit topology, for any
$p^{-1}<\rho_1<1$ there exists a $\rho_2>\rho_1$ such that $L\cap
D_{[\rho_1,1)}(\mathbb{Z}_p,K)$ contains the lattice
\begin{equation}
\{x\in D_{[\rho_1,1)}(\mathbb{Z}_p,K)\colon
\|x\|_{\rho_1,\rho_2}\leq c_{L,\rho_1}\}\label{z68}
\end{equation}
for some positive real number $c_{L,\rho_1}$. Put
\begin{equation*}
s_n:=\sup(|s|\colon sT^n\hbox{ is in }
L)\in\mathbb{R}^{>0}\cup\{+\infty\}.
\end{equation*}
We need to verify that there exists a real number
$p^{-1}<\rho_0<1$ such that
\begin{eqnarray}
\lim_{n\rightarrow\infty}\frac{\rho^{n}}{s_{-n}}&=&0\hspace{1cm}\hbox{and}\label{z66}\\
\lim_{n\rightarrow-\infty}\frac{\rho^{n}}{s_{-n}}&=&0\label{z67}
\end{eqnarray}
for all $1>\rho>\rho_0$. We show that $\rho_0:=\rho_2$ will do.
Take a real number $\rho$ such that $1>\rho>\rho_0$. Choose
$1>\rho_3>\rho$ and $1>\rho_4>\rho_3$ such that $L\cap
D_{[\rho_3,1)}(\mathbb{Z}_p,K)$ contains the lattice
\begin{equation}
\{x\in D_{[\rho_3,1)}(\mathbb{Z}_p,K)\colon
\|x\|_{\rho_3,\rho_4}\leq c_{L,\rho_3}\}.\label{z69}
\end{equation}
So we obtain that
\begin{eqnarray}
s_{-n}&\geq&\frac{c_{L,\rho_1}}{p\max(\rho_1^{-n},\rho_2^{-n})}\hspace{1cm}\hbox{and}\label{z65}\\
s_{-n}&\geq&\frac{c_{L,\rho_3}}{p\max(\rho_3^{-n},\rho_4^{-n})}.\label{z13}
\end{eqnarray}
Indeed, there exist coefficients $c_{1,-n}$ and $c_{2,-n}$ in
$\mathbb{Q}_p$ such that
\begin{eqnarray*}
\frac{c_{L,\rho_1}}{\max(\rho_1^{-n},\rho_2^{-n})}>&c_{1,-n}&\geq\frac{c_{L,\rho_1}}{p\max(\rho_1^{-n},\rho_2^{-n})}\hspace{1cm}\hbox{and}\\
\frac{c_{L,\rho_3}}{\max(\rho_3^{-n},\rho_4^{-n})}>&c_{2,-n}&\geq\frac{c_{L,\rho_3}}{p\max(\rho_3^{-n},\rho_4^{-n})}
\end{eqnarray*}
whence both $c_{i,-n}T^{-n}$ are in $L$ for $i=1,2$ by
$(\ref{z68})$ and $(\ref{z69})$, resp. Now $(\ref{z66})$ follows
from $(\ref{z13})$, and $(\ref{z67})$ from $(\ref{z65})$.
\end{proof}

\begin{cor}\label{r1}
Whenever $\dim G>1$ the dual space of $\mathcal{R}(G,K)$ in the inductive limit
topology is strictly bigger than the dual in the nice topology. In
particular these topologies do not coincide.
\end{cor}
\begin{proof}
We show that the linear functional
\begin{eqnarray*}
\varphi\colon\mathcal{R}(G,K)&\rightarrow& K\\
\sum_{\alpha\in\mathbb{Z}^d}d_{\alpha}{\bf
b}^{\alpha}&\mapsto&\sum_{\alpha\in\mathbb{Z}^d,|\alpha|=0}d_{\alpha}
\end{eqnarray*}
is continuous in the inductive limit topology, but not in the nice
topology. The latter statement is clear as this functional would
correspond to an element in $\mathcal{R}(G,K)$ by Lemma
\ref{r2}. However, the nice topology is independent of the
uniform pro-$p$ group structure on $G$ and whenever $G$ is
commutative this element would have to be of the form
\begin{equation*}
\sum_{|\alpha|=0}{\bf b}^{\alpha}
\end{equation*}
which is not an element of $\mathcal{R}(G,K)$.

For the other statement it suffices to show that the above
functional is continuous on each $ D_{[\rho,1)}(G,K)$. This is
clear as it satisfies
\begin{equation*}
|\varphi(x)|\leq\|x\|_r
\end{equation*}
for any $\rho\leq r<1$.
\end{proof}

\section{Further results on coadmissible modules}

In this section we define an exact functor $R$ from the category of coadmissible
(left) modules over the distribution algebra $D(G,K)$ to the category of
(left) modules over the generalized Robba ring $\mathcal{R}(G,K)$. We call the
image of this functor the category of coadmissible modules over
$\mathcal{R}(G,K)$. We show that whenever $N$ is a \emph{finitely generated}
coadmissible module over $D(G,K)$ then
$\Hom_{\mathcal{R}(G,K)}(R(N),\mathcal{R}(G,K))$ is also coadmissible
(ie.\ lies in the image of the functor $R$). By Corollary \ref{r3}
$\mathcal{R}(G,K)$ is a self-dual topological $K$-algebra, so we would like to
interpret the module $\Hom_{\mathcal{R}(G,K)}(R(N),\mathcal{R}(G,K))$ as the
na\"ive dual space $\Hom_K^{ct}(R(N),K)$ with respect to some topology on
$R(N)$. For this we equip
$\mathcal{R}(G,K)$ with the nice topology. Further, if $M$ is an (abstract) $\mathcal{R}(G,K)$-module then we
define the \emph{canonical topology} on it by taking the strongest
locally convex topology on $M$ such that for any fixed element
$m_0$ in $M$ the map
\begin{eqnarray*}
\mathcal{R}(G,K)&\rightarrow& M\\
x&\mapsto& xm_0
\end{eqnarray*}
is continuous in the nice topology of $\mathcal{R}(G,K)$.
\begin{lem}\label{r5}
The action of $\mathcal{R}(G,K)$ on $M$ is separately continuous
in the canonical topology of $M$. Moreover, any algebraic
homomorphism between two modules $M$ and $N$ is continuous in the
canonical topology. In particular the canonical topology on the
module $\mathcal{R}(G,K)$ over itself coincides with the nice
topology.
\end{lem}
\begin{proof}
For the first statement we need to verify that for any fixed $x_0$
in $\mathcal{R}(G,K)$ the map
\begin{eqnarray*}
M&\rightarrow&M\\
m&\mapsto&x_0m
\end{eqnarray*}
is continuous. By the construction of the topology this follows
from the continuity of the map
\begin{eqnarray*}
\mathcal{R}(G,K)&\rightarrow& M\\
x&\mapsto& xx_0m_0
\end{eqnarray*}
for any fixed $x_0$ in $\mathcal{R}(G,K)$ and $m_0$ in $M$. The
second statement follows similarly.
\end{proof}

So from now on any $\mathcal{R}(G,K)$-module will be equipped with its
canonical topology and we understand continuous maps as ``continuous in the
canonical topology''.
\begin{lem}\label{r4}
Let $M$ be any module over $\mathcal{R}(G,K)$ with its canonical
topology. Then the map
\begin{equation*}
\mu\colon\Hom_{\mathcal{R}(G,K)}(M,\mathcal{R}(G,K))\rightarrow\Hom^{ct}_K(M,K)
\end{equation*}
given by $(\mu(\varphi))(m):=$``the constant term of $\varphi(m)$'' is a $K$-linear bijection.
\end{lem}
\begin{proof}
By Lemma \ref{r5} the functional $\mu(\varphi)$ is indeed continuous showing
that the map $\mu$ is well-defined. The injectivity of $\mu$ follows from the
non-degeneracy of the pairing on $\mathcal{R}(G,K)$. On the other hand, if
$\overline{\varphi}$ is a continuous linear functional on $M$ then the map 
\begin{eqnarray*}
\varphi_{m_0}\colon \mathcal{R}(G,K)&\rightarrow& K\\
\varphi_{m_0}(r)&:=&\overline{\varphi}(rm_0)
\end{eqnarray*}
is continuous in the nice topology on $\mathcal{R}(G,K)$ and therefore has to
come from the right multiplication by an element $r_0=:\varphi(m_0)$ (by
Corollary \ref{r3}) and the Lemma follows.
\end{proof}

Our next observation is the following
\begin{pro}
Let $P$ be a projective coadmissible left $D(G,K)$-module. Then we
have
\begin{equation*}
\Hom_{D(G,K)}(P,D(G,K))\otimes_{D(G,K)}\mathcal{R}(G,K)\cong\Hom_K^{ct}(\mathcal{R}(G,K)\otimes_{D(G,K)}P,K).
\end{equation*}
\end{pro}
\begin{proof}
We choose another projective $D(G,K)$-module $Q$ such
that $P\oplus Q\cong D(G,K)^m$ for some integer $m$. Note that $P$
is finitely generated by Proposition \ref{z10}. Now we have the
following isomorphisms
\begin{eqnarray}
\Hom(P,D(G,K))\otimes\mathcal{R}(G,K)\oplus\Hom(Q,D(G,K))\otimes\mathcal{R}(G,K)\overset{\sim}{\longrightarrow}\notag\\
\overset{\sim}{\longrightarrow}(\Hom(P,D(G,K))\oplus\Hom(Q,D(G,K)))\otimes\mathcal{R}(G,K)
\overset{\sim}{\longrightarrow}\notag\\
\overset{\sim}{\longrightarrow}\Hom(D(G,K)^m,D(G,K))\otimes\mathcal{R}(G,K)\overset{\sim}{\longrightarrow}
\Hom(\mathcal{R}(G,K)^m,\mathcal{R}(G,K))\overset{\sim}{\longrightarrow}\notag\\
\overset{\sim}{\longrightarrow}\Hom(\mathcal{R}(G,K)\otimes
P\oplus
\mathcal{R}(G,K)\otimes Q,\mathcal{R}(G,K))\overset{\sim}{\longrightarrow}\notag\\
\overset{\sim}{\longrightarrow}\Hom(\mathcal{R}(G,K)\otimes
P,\mathcal{R}(G,K))\oplus \Hom(\mathcal{R}(G,K)\otimes
Q,\mathcal{R}(G,K)).\notag
\end{eqnarray}
Moreover, the isomorphism from the first term to the last term is
the direct sum of the maps
\begin{eqnarray*}
\Hom(P,D(G,K))\otimes\mathcal{R}(G,K)&\rightarrow&\Hom(P\otimes\mathcal{R}(G,K),\mathcal{R}(G,K))\hspace{1cm}\hbox{ and}\\
\Hom(Q,D(G,K))\otimes\mathcal{R}(G,K)&\rightarrow&\Hom(Q\otimes\mathcal{R}(G,K),\mathcal{R}(G,K))
\end{eqnarray*}
so both these maps are isomorphisms and the result follows by Lemmata \ref{r5}
and \ref{r4}.
\end{proof}

In the rest of this section we are going to generalize the above
observation to any finitely generated coadmissible module over the
distribution algebra $D(G,K)$.

For that we need the following

\begin{lem}\label{z38}
Let $G$ be a uniform pro-$p$ group. Then for any
$p^{-1}<\rho_1<\rho_2<1$ in $p^{\mathbb{Q}}$ the inclusion
\begin{equation*}
D_{\rho_2}(G,K)\hookrightarrow D_{[\rho_1,\rho_2]}(G,K)
\end{equation*}
of rings is flat.
\end{lem}
\begin{proof}
The proof is similar to that of Theorem \ref{z37}. We choose a
$\rho_0$ in the intersection of $p^{\mathbb{Q}}$ with the open
interval $(p^{-1},\rho_1)$ and show at first that
$D_{(\rho_0,\rho_2]^{bd}}(G,K)$ is flat over $D_{\rho_2}(G,K)$. We take
the filtration $F^s_{\rho_2}$ induced by the norm
$\|\cdot\|_{\rho_2}$ on $D_{\rho_2}(G,K)$ and
$F^0_{\rho_0,\rho_2}D_{(\rho_0,\rho_2]^{bd}}(G,K)$. Both are complete
with respect to this filtration. Moreover,
$F^0_{\rho_2}D_{\rho_2}(G,K)$ is contained in
$F^0_{\rho_0,\rho_2}D_{(\rho_0,\rho_2]^{bd}}(G,K)$ and we have
\begin{equation*}
\mathbb{Q}_p\otimes_{\mathbb{Z}_p}F^0_{\rho_2}D_{\rho_2}(G,K)=D_{\rho_2}(G,K)\hbox{
and
}\mathbb{Q}_p\otimes_{\mathbb{Z}_p}F^0_{\rho_0,\rho_2}D_{(\rho_0,\rho_2]^{bd}}(G,K)=D_{(\rho_0,\rho_2]^{bd}}(G,K).
\end{equation*}
So by Proposition 1.2 in \cite{foundations} it suffices to show
that the associated graded map
\begin{equation*}
\gr_{\rho_2} F^0_{\rho_2}D_{\rho_2}(G,K)\hookrightarrow
\gr_{\rho_2} F^0_{\rho_0,\rho_2}D_{(\rho_0,\rho_2]^{bd}}(G,K)
\end{equation*}
is flat. Since the norm $\|\cdot\|_{\rho_2}$ is multiplicative
both these rings are integral domains. Moreover, we may assume
without loss of generality that $\rho_2$ is an integral power of
the absolute value of the uniformizer in $K$ as we may replace $K$
by a finite extension. So the graded ring $\gr_{\rho_2}
F^0_{\rho_2}D_{\rho_2}(G,K)$ is isomorphic to the polynomial ring
$k[X_0,X_1,\dots,X_d]$ and $\gr_{\rho_2}
F^0_{\rho_0,\rho_2}D_{(\rho_0,\rho_2]^{bd}}(G,K)$ is its localization
at $X_1,\dots,X_d$ and therefore flat.

The proof of the flatness of the map
\begin{equation*}
D_{(\rho_0,\rho_2]^{bd}}(G,K)\hookrightarrow D_{[\rho_1,\rho_2]}(G,K)
\end{equation*}
is entirely analogous to the proof of Theorem \ref{z37} and so we
omit the details.
\end{proof}
Now we fix a sequence $p^{-1}<\rho_1<\dots<\rho_n<\dots<1$ with
$\rho_n\rightarrow1$ in $p^{\mathbb{Q}}$ and introduce (for any
positive integer $l$) the functor $R_{\rho_l}$ from the category
$Coh(D(G,K),\rho_n)$ of coherent sheafs for the projective system
of noetherian Banach algebras $(D_{\rho_n}(G,K))_{n\geq1}$ to the
category $Coh(D_{[\rho_l,1)}(G,K),\rho_n)$ of coherent sheafs for
the system $(D_{[\rho_l,\rho_n]}(G,K))_{n>l}$ by putting
\begin{equation*}
R_{\rho_l}((M_n)_{n\geq1}):=(D_{[\rho_l,\rho_n]}(G,K)\otimes_{D_{\rho_n}(G,K)}M_n)_{n>l}.
\end{equation*}
This is an exact functor by Lemma \ref{z38}. Moreover, by the
equivalence of the categories of coherent sheafs and coadmissible
modules over Fr\'echet-Stein algebras it can be viewed as a
functor from the category $\mathcal{C}_{D(G,K)}$ of coadmissible
modules over $D(G,K)$ to the category
$\mathcal{C}_{D_{[\rho_l,1)}(G,K)}$ of coadmissible modules over
$D_{[\rho_l,1)}(G,K)$ since the latter is also a Fr\'echet-Stein
algebra by Theorem \ref{z37}. We denote this functor by the same
letter. Moreover, let $R$ be the the injective limit of these
functors mapping coadmissible modules $M$ over the distribution
algebra $D(G,K)$ to the module
\begin{equation*}
R(M):=\varinjlim_{l\rightarrow\infty}R_{\rho_l}(M)
\end{equation*}
over the Robba ring $\mathcal{R}(G,K)$. The functor $R$ is also
exact as $\varinjlim$ is exact. We call a module over the ring
$\mathcal{R}(G,K)$ \emph{coadmissible} if it is in the image of
the functor $R$.

\begin{lem}\label{z40}
Let $M$ be a finitely generated coadmissible module over $D(G,K)$.
Then the natural map
\begin{equation*}
\varphi_l\colon D_{[\rho_l,1)}(G,K)\otimes_{D(G,K)}M\rightarrow
R_l(M)
\end{equation*}
is surjective with $D_{[\rho_l,1)}(G,K)$-torsion kernel for any
$l$. In particular if the $D_{[\rho_l,1)}(G,K)$-module $N$ has no
torsion then we have the isomorphism
\begin{equation*}
\Hom_{D_{[\rho_l,1)}(G,K)}(D_{[\rho_l,1)}(G,K)\otimes_{D(G,K)}M,N)\cong\Hom_{D_{[\rho_l,1)}(G,K)}(R_l(M),N).
\end{equation*}
\end{lem}
\begin{proof}
The image of $\varphi_l$ is a coadmissible
$D_{[\rho_l,1)}(G,K)$-module as it is a finitely generated
submodule of a coadmissible module. Moreover, for any $n>l$ we
have
\begin{equation*}
D_{[\rho_l,\rho_n]}(G,K)\otimes
\mathrm{Im}(\varphi)=D_{[\rho_l,\rho_n]}(G,K)\otimes R_l(N).
\end{equation*}
Therefore the surjectivity. Let us assume now indirectly that
there is an element $x$ in the kernel of $\varphi_l$ which is not
torsion. This means we have an injective homomorphism from
$D_{[\rho_l,1)}(G,K)$ to $D_{[\rho_l,1)}(G,K)\otimes_{D(G,K)}M$
mapping $1$ to $x$. This contradicts to the flatness of
$D_{[\rho_l,\rho_n]}(G,K)$ over $D_{[\rho_l,1)}(G,K)$ which is a
consequence of the Fr\'echet-Stein property. The last statement
follows immediately from the long exact sequence of the functor
$\Hom_{D_{[\rho_l,1)}(G,K)}(\cdot,N)$.
\end{proof}

Our main result in this section is the following
\begin{thm}\label{z39}
Let $M$ be a \emph{finitely generated} coadmissible module over
the distribution algebra $D(G,K)$ of a uniform pro-$p$ group $G$.
Then $\Hom_K^{ct}(R(M),K)$ is a coadmissible module over
$\mathcal{R}(G,K)$.
\end{thm}
\begin{proof}
We are going to prove that in fact we have
\begin{equation*}
\Hom_K^{ct}(R(M),K)=R(\Hom_{D(G,K)}(M,D(G,K))).
\end{equation*}
At first, we clearly have
$\Hom_K^{ct}(R(M),K)=\Hom_{\mathcal{R}(G,K)}(R(M),\mathcal{R}(G,K))$.
Moreover, by definition we also have
\begin{equation*}
\Hom_{\mathcal{R}(G,K)}(R(M),\mathcal{R}(G,K))=\varinjlim_{l\rightarrow\infty}\Hom_{D_{[\rho_l,1)}(G,K)}(R_l(M),\mathcal{R}(G,K)).
\end{equation*}
Further, by Lemma \ref{z40} we obtain
\begin{eqnarray}
\Hom_{D_{[\rho_l,1)}(G,K)}(R_l(M),\mathcal{R}(G,K))=
\Hom_{D_{[\rho_l,1)}(G,K)}(D_{[\rho_l,1)}(G,K)\otimes_{D(G,K)}
M,\mathcal{R}(G,K))=\notag\\
=\Hom_{D(G,K)}(M,\mathcal{R}(G,K)).\notag
\end{eqnarray}
So putting these all together we get
\begin{equation}
\Hom_K^{ct}(R(M),K)=\Hom_{D(G,K)}(M,\mathcal{R}(G,K))\label{z41}.
\end{equation}

On the other hand since $M$ is finitely generated we also have
\begin{equation*}
\Hom_{D(G,K)}(M,\mathcal{R}(G,K))=\varinjlim_{l\rightarrow\infty}\Hom_{D(G,K)}(M,D_{[\rho_l,1)}(G,K)).
\end{equation*}
By Lemma \ref{z40} once again we also have
\begin{eqnarray}
\Hom_{D(G,K)}(M,D_{[\rho_l,1)}(G,K))=\Hom_{D_{[\rho_l,1)}(G,K)}(D_{[\rho_l,1)}(G,K)\otimes
M,D_{[\rho_l,1)}(G,K))=\notag\\
=\Hom_{D_{[\rho_l,1)}(G,K)}(R_l(M),D_{[\rho_l,1)}(G,K))=\notag\\
=\varprojlim_n\Hom_{D_{[\rho_l,\rho_n]}(G,K)}(D_{[\rho_l,\rho_n]}(G,K)\otimes
M,D_{[\rho_l,\rho_n]}(G,K)).\notag
\end{eqnarray}
The last equality follows from the fact that $R_l(M)$ is a
coadmissible module over the Fr\'echet-Stein algebra
$D_{[\rho_l,1)}(G,K)$. Now by the flatness of
$D_{[\rho_l,\rho_n]}(G,K)$ over $D_{\rho_n}(G,K)$ and the spectral
sequences associated to
$\Hom_{D_{[\rho_l,\rho_n]}(G,K)}(D_{[\rho_l,\rho_n]}(G,K)\otimes\cdot,D_{[\rho_l,\rho_n]}(G,K))$
and
$D_{[\rho_l,\rho_n]}(G,K)\otimes\Hom_{D_{\rho_n}(G,K)}(\cdot,D_{\rho_n}(G,K))$
we obtain
\begin{eqnarray}
\Hom_{D_{[\rho_l,\rho_n]}(G,K)}(D_{[\rho_l,\rho_n]}(G,K)\otimes
M,D_{[\rho_l,\rho_n]}(G,K))=\notag\\
=D_{[\rho_l,\rho_n]}(G,K)\otimes\Hom_{D_{\rho_n}(G,K)}(M_n,D_{\rho_n}(G,K))\notag
\end{eqnarray}
and hence
\begin{equation*}
\Hom_{D_{[\rho_l,1)}(G,K)}(R_l(M),D_{[\rho_l,1)}(G,K))=R_l(\Hom_{D(G,K)}(M,D(G,K))).
\end{equation*}
The result follows by taking the injective limit.
\end{proof}

\begin{rem}
Whenever $G=\mathbb{Z}_p$ then any torsion-free coadmissible
module over the distribution algebra $D(\mathbb{Z}_p,K)$ is
finitely generated and even free, so the above Theorem applies to \emph{any}
coadmissible module over $D(\mathbb{Z}_p,K)$---we may drop the assumption that
it is finitely generated. Indeed, any homomorphism in $\Hom_{D(\mathbb{Z}_p,K)}(M,D(\mathbb{Z}_p,K))$
factors through the quotient of $M$ by its torsion part. This, however, leaves the following questions
open.
\end{rem}

\begin{que}
Does there exist a uniform pro-$p$ group $G$ and a torsion-free
coadmissible module over its distribution algebra $D(G,K)$ which
is not finitely generated?
\end{que}

\begin{que}
Does $R_l(M)=D_{[\rho_l,1)}(G,K)\otimes_{D(G,K)} M$ always hold or
at least for torsion-free coadmissible modules $M$? It is
certainly true for projective coadmissible modules.
\end{que}

\begin{que}
Is there a torsion-free coadmissible module over $D(G,K)$ such
that
\begin{equation*}
D_{[\rho_l,1)}(G,K)\otimes_{D(G,K)} M
\end{equation*}
is not torsion-free over $D_{[\rho_l,1)}(G,K)$?
\end{que}

\newpage

\appendix

\theoremstyle{plain}

\newcommand{\bb}{\mathbf{b}}

\newtheorem{theorem}{Theorem}[section]
\newtheorem{corollary}[theorem]{Corollary}
\newtheorem{lemma}[theorem]{Lemma}
\newtheorem{remark}[theorem]{Remark}
\newtheorem{proposition}[theorem]{Proposition}
\newtheorem{conjecture}[theorem]{Conjecture}
\newtheorem{definition}[theorem]{Definition}
\newtheorem{fact}[theorem]{Fact}
\newtheorem{question}[theorem]{Question}

\section{Appendix by Peter Schneider: Robba rings for compact
$p$-adic Lie groups}

The Robba ring is a fundamental tool in $p$-adic differential
equations and in $p$-adic Galois representations. It is defined as a
ring of certain infinite Laurent series in one variable over a
$p$-adic field $K$. Conceptually it is related to the cyclotomic
$\mathbb{Z}_p$-extension of $K$ whose Galois group is isomorphic to
the additive group of $p$-adic integers $G = \mathbb{Z}_p$. In fact,
the Robba ring can be understood in terms of the completed group
ring $\mathbb{Z}_p[[G]]$ by a process of localization and
completion.

Recent developments in the theory of $p$-adic Galois representations
require the use of more general compact $p$-adic Lie groups $G$. In
particular $G$ might be nonabelian like $G = GL_n(\mathbb{Z}_p)$. It
then becomes a natural question whether an analog of the Robba ring
exists in this situation, possibly being constructed out of the
completed group ring $\mathbb{Z}_p[[G]]$. But $\mathbb{Z}_p[[G]]$,
in general, will be noncommutative. Hence localization becomes too
difficult a process. The idea of these notes grew out of the
technique of algebraic microlocalization.

In the first section we will adapt microlocalization to the framwork
of $p$-adic Banach algebras. This means that ring filtrations are
replaced by multiplicative nonarchimedean norms. More importantly,
for the application we have in mind, it is crucial to generalize the
theory in such a way that the microlocalization can be preformed
simultaneously with respect to finitely many such norms. Apparently
this has not been done in the algebraic context of ring filtrations.

In the second section we apply this new technique to construct,
under mild assumptions on $G$, various rings out of
$\mathbb{Z}_p[[G]]$ culminating in a ring $R(G,K)$ which we call a
Robba ring of $G$. Actually the ring $R(G,K)$ does depend, which we
suppress in the notation, on the choice of a global coordinate
system for the $p$-adic Lie group $G$. That such a phenomenon occurs
in higher dimensions might not be too surprising.

In the third section we show that various variants of the Robba
ring, which classically play an important role, also exist in our
general setting.

The constructions in these notes were first presented at a workshop
at M\"unster in 2005. Due to the lack of applications they were not
published at the time. Given the progress made by G.\ Z\'abr\'adi on
structural properties of these rings it seemed appropriate to add
these essentially unchanged notes as an appendix to his paper.

\subsection{Generalized microlocalization of quasi-abelian normed
algebras}

Let $K$ be a nonarchimedean complete field with absolute value $|\
|$. To motivate the construction in this section we consider the
Tate algebra $K\langle T\rangle$ over $K$, i.e., the ring of power
series $f(T) = \sum_{n \geq 0} \lambda_nT^n$ over $K$ in one
variable $T$ which converge on the closed unit disk. Its natural
norm is the Gauss norm given by
\begin{equation*}
|f|_1 := \sup_{n \geq 0} |\lambda_n| \ .
\end{equation*}
But for any $0 < r \leq 1$ we also have the norm
\begin{equation*}
|f|_r := \sup_{n \geq 0} |\lambda_n|
\end{equation*}
on $K\langle T\rangle$. If some power of $r$ lies in the value
group $|K^\times|$ then the completion of $K\langle T\rangle$ with
respect to the norm $|\ |_r$ is the algebra of analytic functions
on the closed disk of radius $r$ around the origin. On the other
hand, if we first invert the variable $T$ and then complete with
respect to the norm $|\ |_1$ then we obtain the analytic functions
on the unit circle. Microlocalization can be viewed as a
generalization of this latter construction to certain
noncommutative normed algebras. Finally, we may invert the
variable $T$ and then complete with respect to the norm $\max(|\
|_1,|\ |_r)$ in order to obtain the algebra of analytic functions
on the closed annulus of inner and outer radius $r$ and $1$,
respectively. The purpose of this section is to generalize the
concept of microlocalization in such a way that we obtain a
noncommutative analog of this third construction.

In fact, what we are going to do is a rather straightforward
modification of the arguments and results in \cite{Spr}. But since
we work with normed algebras instead of filtered rings and since
the paper \cite{Spr} is partly obscured my confusing typographical
errors we include, for the benefit of the reader, complete proofs.

Having specific applications in mind we do not strive for ultimate
generality. It is clear that similar ideas will work in the more
general context considered in \cite{vdE}.

We fix a (usually noncommutative) unital $K$-algebra $A$. A
(nonarchimedean) norm $|\ |$ on $A$ is called multiplicative if
\begin{equation*}
|1| = 1\qquad\text{and}\qquad |ab| = |a|\cdot |b|\quad\text{for
any $a,b \in A$}.
\end{equation*}
Let $|\ |$ be such a multiplicative norm. The ring $A$ of course
then is an integral domain. For later reference we also recall the
following triviality.

\begin{remark}\label{rem1}
If $|a_0 - b_0| < |b_0|$ and $|a_1 - b_1| < |b_1|$ then $|a_0a_1 -
b_0b_1| < |b_0b_1|$.
\end{remark}
\begin{proof}
We compute
\begin{align*}
|a_0a_1 - b_0b_1| & = |(a_0 - b_0)a_1 + b_0(a_1 - b_1)| \\
 & = \max(|a_0 - b_0|\cdot|a_1|,|b_0|\cdot|a_1 - b_1)) \\
 & < |b_0|\cdot|b_1| = |b_0b_1| \ .
\end{align*}
\end{proof}

In this paper we are mostly interested in norms of the following
much more restricted kind.

\begin{definition}\label{quasi-abelian}
The multiplicative norm $|\ |$ on $A$ is called quasi-abelian if
there is a constant $0 < \gamma < 1$ such that
\begin{equation*}
|ab - ba| \leq \gamma\cdot|ab|\qquad\hbox{for any}\ a,b \in A\ .
\leqno{(qa)}
\end{equation*}
\end{definition}

Throughout the paper we in fact fix a finite family of
quasi-abelian norms $|\ |_1,\ldots,|\ |_m$ on $A$. Corresponding
to each norm $|\ |_i$ we introduce the function
\begin{equation*}
\Delta_i(x,y) := |s|_i^{-1}\cdot |t|_i^{-1}\cdot |at - sb|_i
\end{equation*}
on $A\setminus\{0\} \times A$ where $x = (s,a)$ and $y = (t,b)$.
As a second input we fix a multiplicatively closed subset $S$ of
$A$ (by convention this includes the requirement that $1 \in S$
but $0 \not\in S$). The \textit{saturation} $S_i$ of $S$ with
respect to $|\ |_i$ is the set
\begin{equation*}
S_i := \{a \in A : |at - s|_i < |s|_i\ \hbox{for some}\ s,t \in
S\}.
\end{equation*}
Note that this definition is symmetric in that, due to the
condition (qa), we have
\begin{equation*}
S_i = \{a \in A : |ta - s|_i < |s|_i\ \hbox{for some}\ s,t \in
S\}.
\end{equation*}

\begin{lemma}\label{multclosed}
$S_i$ is a multiplicatively closed subset containing $S$.
\end{lemma}
\begin{proof}
Obviously $S \subseteq S_i$ and $0 \not\in S_i$. Let $a,b \in S_i$
and $s,t,s',t' \in S$ such that
\begin{equation*}
|at - s|_i < |s|_i\quad \text{and}\quad |bt' - s'|_i < |s'|_i\ .
\end{equation*}
By Remark \ref{rem1} we then have $|atbt' - ss'|_i < |ss'|_i$.
Because of
\begin{align*}
 |abtt' - ss'|_i & = |abtt' - atbt' + atbt' -ss'|_i \\
                 & \leq \max(|a|_i|bt - tb|_i|t'|_i,|atbt' -
                 ss'|_i)
\end{align*}
it therefore suffices, in order to obtain $ab \in S_i$, to check
that
\begin{equation*}
|a|_i|bt - tb|_i|t'|_i < |ss'|_i = |atbt'|_i
\end{equation*}
but which is a consequence of the condition (qa).
\end{proof}

The crucial consequence of the condition (qa) on which everything
later on relies is the following ``approximative'' Ore condition.

\begin{proposition}\label{approximative}
For any $\epsilon > 0$ and any $(s,a) \in S \times A$ we have:
\begin{itemize}
 \item[i.] There is a pair $(t,b) \in S \times A$ such that
\begin{equation*}
|at - sb|_i \leq \epsilon |a|_i|t|_i\qquad\text{(resp.}\ |ta -
bs|_i \leq \epsilon |a|_i|t|_i)
\end{equation*}
and $|s|_i|b|_i \leq |a|_i|t|_i$ for any $1 \leq i \leq m$; if
$\epsilon < 1$ then $|s|_i|b|_i = |a|_i|t|_i$;
 \item[ii.] if in i. we have $a \in S$ and $\epsilon < 1$ then $b \in S_1
\cap\ldots\cap S_m$.
\end{itemize}
\end{proposition}
\begin{proof}
i. Put $a_0 := a$ and $a_n := a_{n-1}s - sa_{n-1}$ for $n \geq 1$.
By (qa) we have a constant $0 < \gamma < 1$ such that
\begin{equation*}
|a_n|_i \leq \gamma |a_{n-1}|_i|s|_i \qquad\text{and hence}\qquad
|a_n|_i \leq \gamma^n |s|_i^n|a|_i
\end{equation*}
for any $n \geq 0$ and any $1 \leq i \leq m$. We therefore find an
$\ell \in \mathbb{N}$ such that $|a_\ell|_i \leq \epsilon |a|_i
|s^\ell|_i$ for any $1 \leq i \leq m$. By induction with respect
to $n$ one checks that
\begin{equation}\label{1}
as^n = \sum_{j = 0}^{n-1} {n \choose j} s^{n-j}a_j + a_n
\end{equation}
and
\begin{equation}\label{2}
a_n = \sum_{j=0}^n (-1)^j {n \choose j} s^j a s^{n-j}\ .
\end{equation}
We put $t := s^\ell$ and $b := \sum_{j=0}^{\ell -1} {\ell \choose
j} s^{\ell -j-1}a_j$ and obtain from \eqref{1} that $at = sb +
a_\ell$ and hence
\begin{equation*}
|at - sb|_i \leq \epsilon |a|_i|t|_i
\end{equation*}
for any $1 \leq i \leq m$. By \eqref{2} we have $|a_n|_i \leq
|s|_i^n |a|_i$ for any $n \geq 0$ and therefore $|b|_i \leq
|s^{\ell -1}|_i |a|_i = |s|_i^{-1} |t|_i |a|_i$. The stated
identity in case $\epsilon < 1$ is obvious. The second half of the
assertion is shown analogously. ii. This is clear.
\end{proof}

\begin{corollary}\label{cor-approx}
For any $\epsilon > 0$ and any $x \in S \times A$ there is a $\xi
\in S \times A$ such that $\Delta_i(x,\xi) \leq \epsilon$ for any
$1 \leq i \leq m$.
\end{corollary}

We now introduce, for any $1 \leq i \leq m$, the function
\begin{equation*}
d_i(x,y) := \inf_{\xi \in S_i \times A}
\max(\Delta_i(x,\xi),\Delta_i(y,\xi))
\end{equation*}
on $(S_i \times A)^2$ as well as the function
\begin{equation*}
d(x,y) := \max(d_1(x,y),\ldots,d_m(x,y))
\end{equation*}
on $(S \times A)^2$. Obviously we have $d(x,y) = d(y,x) \geq 0$.
Furthermore, it follows from Cor.\ \ref{cor-approx} that $d(x,x) =
0$.

\begin{proposition}\label{triangle}
For any $x,y,z \in S \times A$ we have
\begin{equation*}
d(x,z) \leq \max(d(x,y),d(y,z))\ .
\end{equation*}
\end{proposition}
\begin{proof}
It certainly suffices to establish the inequality
\begin{equation*}
d_i(x,z) \leq \max(d_i(x,y),d_i(y,z))
\end{equation*}
for each individual $1 \leq i \leq m$. Given any constant
$\gamma_0
> \max(d_i(x,y),d_i(y,z))$ we have to show that $d_i(x,z) <
\gamma_0$. Let $x = (s,a)$, $y = (t,b)$, and $z = (u,c)$. We find
$\xi = (\sigma,\alpha)$ and $\eta = (\tau,\beta)$ in $S_i \times
A$ such that $\Delta_i(x,\xi)$, $\Delta_i(y,\xi)$,
$\Delta_i(y,\eta)$, and $\Delta_i(z,\eta)$ all are smaller than
$\gamma_0$, i.e., such that
\begin{gather*}
 |a\sigma - s\alpha|_i < \gamma_0 |s|_i|\sigma|_i\ , \ |b\sigma -
 t\alpha|_i < \gamma_0 |t|_i|\sigma|_i\ ,\\
 |b\tau - t\beta|_i < \gamma_0 |t|_i|\tau|_i\ , \ |c\tau - u\beta|_i
 < \gamma_0 |u|_i|\tau|_i\ .
\end{gather*}
We choose a $0 < \epsilon < 1$ such that $\epsilon |b|_i \leq
\gamma_0 |t|_i$ and $\epsilon |c|_i \leq \gamma_0 |u|_i$. By
Prop.\ \ref{approximative}.i applied to the multiplicative set
$S_i$ and $(\tau,\sigma) \in S_i \times A$ there is a pair $(v,d)
\in S_i \times A$ such that
\begin{equation*}
|\sigma v - \tau d|_i < \epsilon |\sigma|_i|v|_i < |\sigma|_i|v|_i
= |\tau|_i|d|_i \ .
\end{equation*}
It follows that
\begin{align*}
 |t\alpha v - t\beta d|_i & = |t\alpha v - b\sigma v + b\sigma v -
 b\tau d + b\tau d - t\beta d|_i \\
 & < \max(\gamma_0 |t|_i|\sigma|_i|v|_i,\epsilon
 |b|_i|\sigma|_i|v|_i,\gamma_0 |t|_i|\tau|_i|d|_i) \\
 & = \gamma_0 |t|_i|\sigma|_i|v|_i\ .
\end{align*}
By the multiplicativity of $|\ |_i$ we then must have
\begin{equation*}
|\alpha v - \beta d|_i < \gamma_0 |\sigma|_i|v|_i\ .
\end{equation*}
We therefore obtain
\begin{align*}
 |s\beta d - a\sigma v|_i & = |s\beta d - s\alpha v + s\alpha v -
 a\sigma v|_i \\
 & < \gamma_0 |s|_i|\sigma|_i|v|_i = \gamma_0 |s|_i|\sigma v|_i
\end{align*}
and
\begin{align*}
 |u\beta d - c\sigma v|_i & = |u\beta d - c\tau d + c\tau d -
 c\sigma v|_i \\
 & < \max(\gamma_0
 |u|_i|\tau|_i|d|_i,\epsilon|c|_i|\sigma|_i|v|_i) \\
 & = \gamma_0 |u|_i|\sigma|_i|v|_i = \gamma_0 |u|_i|\sigma v|_i \
 .
\end{align*}
This means that
\begin{equation*}
d_i(x,z) \leq \max(\Delta_i(x,(\sigma v,\beta
d)),\Delta_i(z,(\sigma v,\beta d))) < \gamma_0 \ .
\end{equation*}
\end{proof}

We obtain that $d$ is a pseudometric on the set $S \times A$. The
object of interest in this paper is the Hausdorff completion of
the pseudometric space $(S \times A,d)$. But first we give a few
explicit formulae for $d$.

\begin{lemma}\label{formulae}
For any $1 \leq i \leq m$ and any $a,b \in A$ and $s,t \in S_i$ we
have:
\begin{enumerate}
 \item[i.] $d_i((s,a),(ts,ta)) = 0$;
 \item[ii.] $d_i((s,a),(s,b)) = |s|_i^{-1}|a - b|_i$;
 \item[iii.] $d_i((s,a),(1,0)) = |s|_i^{-1}|a|_i$;
 \item[iv.] $d_i((s,a),(t,a)) \leq |s|_i^{-1}|t|_i^{-1}|a|_i|s - t|_i$.
\end{enumerate}
\end{lemma}
\begin{proof}
i. For any $\xi \in A\setminus\{0\} \times A$ we have
\begin{equation*}
\Delta_i((ts,ta),\xi) = \Delta_i((s,a),\xi) \ .
\end{equation*}
Hence $d_i((s,a),(ts,ta)) = \mathop{\inf}\limits_{\xi \in S_i
\times A} \Delta_i((s,a),\xi)$ and the assertion follows from
Cor.\ \ref{cor-approx} applied to $S_i \times A$.

ii. For $\xi = (\sigma,\alpha) \in A\setminus\{0\} \times A$ we
have
\begin{align*}
 \max(\Delta_i((s,a),\xi),\Delta_i((s,b),\xi)) & =
 |s|_i^{-1}|\sigma|_i^{-1} \max(|a\sigma - s\alpha|_i,|b\sigma -
 s\alpha|_i) \\
 & \geq |s|_i^{-1}|\sigma|_i^{-1} |a\sigma - b\sigma|_i =
 |s|_i^{-1} |a-b|_i \ .
\end{align*}
Hence $d_i((s,a),(s,b)) \geq |s|_i^{-1} |a-b|_i$. On the other
hand we have
\begin{align*}
 \Delta_i((s,b),\xi) & = |s|_i^{-1}|\sigma|_i^{-1}|b\sigma -
 s\alpha|_i \\
 & = |s|_i^{-1}|\sigma|_i^{-1}|b\sigma - a\sigma + a\sigma -
 s\alpha|_i \\
 & \leq \max(|s|_i^{-1}|b-a|_i,\Delta_i((s,a),\xi)) \ .
\end{align*}
Using Cor.\ \ref{cor-approx} (applied to $S_i \times A$) it
follows that $d_i((s,a),(s,b)) \leq |s|_i^{-1} |a-b|_i$.

iii. As a special case of ii. we have $d_i((s,a),(s,0)) =
|s|_i^{-1}|a|_i$. We also have $d_i((1,0),(s,0)) = 0$ by i. It
then follows from Prop.\ \ref{triangle} that $d_i((s,a),(1,0)) =
d_i((s,a),(s,0)) = |s|_i^{-1}|a|_i$.

iv. For $\xi = (\sigma,\alpha) \in A\setminus\{0\} \times A$ we
have
\begin{align*}
 \Delta_i((t,a),\xi) & = |t|_i^{-1}|\sigma|_i^{-1}
 |a\sigma - t\alpha|_i \\
 & \leq |t|_i^{-1}|\sigma|_i^{-1}\max(|a\sigma -
 s\alpha|_i,|s-t|_i|\alpha|_i) \ .
\end{align*}
By Prop.\ \ref{approximative}.i there is a $\xi \in S_i \times A$,
for any $\epsilon
> 0$, such that
\begin{equation*}
|a\sigma - s\alpha|_i \leq \epsilon|a|_i|\sigma|_i
\qquad\text{and}\qquad |\alpha|_i \leq |s|_i^{-1}|a|_i|\sigma|_i \
.
\end{equation*}
For such a $\xi$ we have
\begin{equation*}
\Delta_i((t,a),\xi) \leq |t|_i^{-1}\max(\epsilon
|a|_i,|s|_i^{-1}|a|_i|s-t|_i) \ .
\end{equation*}
It follows that
\begin{align*}
 d_i((s,a),(t,a)) & \leq
 \max(\Delta_i((s,a),\xi),\Delta_i((t,a),\xi)) \\
 & = \max(|s|_i^{-1}|\sigma|_i^{-1}|a\sigma - s\alpha|_i,
 \Delta_i((t,a),\xi)) \\
 & \leq \max(\epsilon|s|_i^{-1}|a|_i,\epsilon|t|_i^{-1}|a|_i,
 |s|_i^{-1}|t|_i^{-1}|a|_i|s-t|_i)
\end{align*}
for any $\epsilon > 0$ and hence $d_i((s,a),(t,a)) \leq
|s|_i^{-1}|t|_i^{-1}|a|_i|s-t|_i$.
\end{proof}

Let now $\mathcal{C}(S \times A)$ denote the set of all Cauchy
sequences $(x_j)_{j \in \mathbb{N}}$ (w.r.t.\ $d$) in $S \times
A$. It contains $S \times A$ via the constant sequences. The
pseudometric $d$ extends to $\mathcal{C}(S \times A)$ by
\begin{equation*}
d((x_j)_j,(y_j)_j) := \lim_{j \rightarrow \infty} d(x_j,y_j) \ .
\end{equation*}
We let $B := A\langle S;|\ |_1,\ldots,|\ |_m\rangle$ denote the
quotient of $\mathcal{C}(S \times A)$ by the equivalence relation
\begin{equation*}
(x_j)_j \sim (y_j)_j \quad\text{if}\quad d((x_j)_j,(y_j)_j) = 0 \
.
\end{equation*}
The pseudometric $d$ passes to a metric on $B$ which we again
denote by $d$. The metric space $(B,d)$ together with the obvious
map $S \times A \longrightarrow B$ is the Hausdorff completion of
the pseudometric space $(S \times A,d)$. We let $s^{-1}a \in B$
denote the image of $(s,a) \in S \times A$. Obviously
\begin{equation*}
s^{-1}a = t^{-1}b \qquad\text{if and only if}\qquad d((s,a),(t,b))
= 0 \ .
\end{equation*}
By Lemma \ref{formulae}.ii the composed map
\begin{align*}
 (A,\max(|\ |_1,\ldots,|\ |_m)) & \longrightarrow  (B,d) \\
 \hfill a & \longmapsto  1^{-1}a =: a
\end{align*}
is an isometry.

In the following we will show that $B$ naturally is a $K$-Banach
algebra. We begin with the scalar multiplication by $K$. On the
set $S \times A$ we have a $K$-action given by
\begin{equation*}
\lambda(s,a) := (s,\lambda a) \ .
\end{equation*}
It satisfies
\begin{equation*}
d(\lambda x,\lambda y) = |\lambda|d(x,y)
\end{equation*}
and therefore extends to a $K$-action
\begin{equation*}
\lambda(x_j)_j := (\lambda x_j)_j
\end{equation*}
on $\mathcal{C}(S \times A)$ which in turn descends to a
$K$-action on $B$. The natural map $A \longrightarrow B$ is
$K$-equivariant.

Next we construct the addition, and begin with the following
immediate consequence of Lemma \ref{formulae}.ii.

\begin{remark}\label{rem2}
For any $s,t \in S$ and $a_1,a_2,b_1,b_2 \in A$ we have
\begin{equation*}
d((s,a_1 + a_2),(s,b_1 + b_2)) \leq
\max(d((s,a_1),(s,b_1)),d((s,a_2),(s,b_2))) \ .
\end{equation*}
\end{remark}

\begin{lemma}\label{add1}
Let $0 < \epsilon < 1$; for $s,t \in S$ and $a,b,c,d \in A$ such
that
\begin{equation*}
d((s,a),(t,c))\leq \epsilon\quad\text{and}\quad d((s,b),(t,d))
\leq \epsilon
\end{equation*}
we have
\begin{equation*}
d((s,a+b),(t,c+d)) \leq \epsilon \max_{1 \leq i \leq m}
\max(1,|s|_i^{-1}|a|_i,|s|_i^{-1}|b|_i,|t|_i^{-1}|c|_i,|t|_i^{-1}|d|_i)\
.
\end{equation*}
\end{lemma}
\begin{proof}
Applying Prop.\ \ref{approximative} to $(t,s)$ we find $(u,v) \in
S \times S_1\cap\ldots\cap S_m$ such that
\begin{equation*}
|us - vt|_i \leq \epsilon |u|_i|s|_i < |u|_i|s|_i = |v|_i|t|_i
\end{equation*}
for any $1 \leq i \leq m$. We claim that
\begin{equation*}
d_i((s,a),(us,vc)) \leq \epsilon \max(1,|s|_i^{-1}|a|_i) \ .
\end{equation*}
For any $\xi = (\sigma,\alpha) \in S_i \times A$ we have
\begin{align*}
 |vc\sigma - us\alpha|_i & = |v(c\sigma - t\alpha) + (vt -
 us)\alpha|_i \\
 & \leq \max(|t|_i^{-1}|u|_i|s|_i|c\sigma - t\alpha|_i, \epsilon
 |u|_i|s|_i|\alpha|_i)
\end{align*}
and hence
\begin{align*}
 \Delta_i((us,vc),\xi) & \leq \max(|t|_i^{-1}|\sigma|_i^{-1}
 |c\sigma - t\alpha|_i, \epsilon |\sigma|_i^{-1}|\alpha|_i) \\
 & = \max(\Delta_i((t,c),\xi),\epsilon |\sigma|_i^{-1}|\alpha|_i) \
 .
\end{align*}
Since $d((s,a),(t,c)) \leq \epsilon$ we find a $\xi$ with
$\max(\Delta_i((s,a),\xi),\Delta_i((t,c),\xi)) \leq \epsilon$. We
conclude that
\begin{align*}
 d_i((s,a),(us,vc)) & \leq
 \max(\Delta_i((s,a),\xi),\Delta_i((us,vc),\xi)) \\
 & \leq \epsilon \max(1,|\sigma|_i^{-1}|\alpha|_i) \ .
\end{align*}
But $\Delta_i((s,a),\xi) \leq \epsilon$ also implies that
$|\sigma|_i^{-1}|\alpha|_i \leq \max(|s|_i^{-1}|a|_i,\epsilon)$.
Hence our claim (observe that in case $|s|_i^{-1}|a|_i < \epsilon
< 1$ we have $\epsilon \max(1,|\sigma|_i^{-1}|\alpha|_i) =
\epsilon = \epsilon \max(1,|s|_i^{-1}|a|_i)$. By the same
reasoning we also have
\begin{equation*}
d_i((s,b),(us,vd)) \leq \epsilon \max(1,|s|_i^{-1}|b|_i) \ .
\end{equation*}
Applying Lemma \ref{formulae}.i to the element $v \in S_1
\cap\ldots\cap S_m$ we obtain
\begin{multline*}
 d_i((s,a+b),(t,c+d)) = d_i((s,a+b),(vt,v(c+d))) \\
 \leq
 \max(d_i((s,a+b),(us,v(c+d))),d_i((us,v(c+d)),(vt,v(c+d)))) \ .
\end{multline*}
Using again Lemma \ref{formulae}.i and Remark \ref{rem2} we have
\begin{equation*}
\begin{split}
 & d_i((s,a+b),(us,v(c+d))) = d_i((us,u(a+b)),(us,v(c+d))) \\
 & \qquad \leq \max(d_i((us,ua),(us,vc)),d_i((us,ub),(us,vd))) \\
 & \qquad = \max(d_i((s,a),(us,vc)),d_i((s,b),(us,vd))) \\
 & \qquad \leq \epsilon \max(1,|s|_i^{-1}|a|_i,|s|_i^{-1}|b|_i) \ .
\end{split}
\end{equation*}
According to Lemma \ref{formulae}.iv we have
\begin{equation*}
\begin{split}
 & d_i((us,v(c+d)),(vt,v(c+d))) \\
 & \qquad\leq
 |u|_i^{-1}|s|_i^{-1}|v|_i^{-1}|t|_i^{-1}|v|_i|c+d|_i|us - vt|_i) \\
 & \qquad \leq \epsilon |t|_i^{-1}|c+d|_i \ .
\end{split}
\end{equation*}
It follows that
\begin{equation*}
d_i((s,a+b),(t,c+d)) \leq \epsilon
\max(1,|s|_i^{-1}|a|_i,|s|_i^{-1}|b|_i,|t|_i^{-1}|c+d|_i) \ .
\end{equation*}
This finishes the proof.
\end{proof}

\begin{corollary}\label{add2}
For $s,t \in S$ and $a,b,c,d \in A$ such that $s^{-1}a = t^{-1}c$
and $s^{-1}b = t^{-1}d$ we have $s^{-1}(a+b) = t^{-1}(c+d)$.
\end{corollary}
\begin{proof}
Our assumption amounts to $d((s,a),(t,c)) = d((s,b),(t,d)) = 0$.
The previous lemma then implies that $d((s,a+b),(t,c+d)) = 0$.
\end{proof}

At this point we introduce, for any $n \in \mathbb{N}$, the subset
\begin{multline*}
 \qquad\qquad P^{(n)} :=\{(e_1,\ldots,e_n) \in B^n : \text{there are
 $a_1,\ldots,a_n \in A$ and $s \in S$ such} \\
 \text{that $e_j = s^{-1}a_j$ for any $1 \leq j
 \leq n$}\} \qquad\qquad\qquad
\end{multline*}
of $B^n$. Of course, $P^{(1)}$ is just the image of the natural
map $S \times A \longrightarrow B$. By Lemma \ref{formulae}.iii
the element
\begin{equation*}
0 := s^{-1}0
\end{equation*}
is well defined in $P^{(1)}$ (i.e., is independent of the choice
of $s \in S$). Cor.\ \ref{add2} says that the map
\begin{align*}
 P^{(2)} & \longrightarrow  B \\
 (e,f) & \longmapsto  e+f := s^{-1}(a+b)\ \ \text{if}\ e = s^{-1}a,
 f = s^{-1}b
\end{align*}
is well defined. We obviously have:
\begin{enumerate}
 \item[1.] $e+f = f+e$,
 \item[2.] $e+0 = e$,
 \item[3.] $e+(-1)e = 0$,
 \item[4.] $(e+f)+g = e+(f+g)$ for $(e,f,g) \in P^{(3)}$,
 \item[5.] $\lambda(e+f) = \lambda e + \lambda f$ for $\lambda \in
 K$,
 \item[6.] $(\lambda + \mu)e = \lambda e + \mu e$ for $\lambda,\mu \in
 K$.
\end{enumerate}
For any $e \in B$ we put
\begin{equation*}
|e| := d(e,0) \ .
\end{equation*}
It follows from Lemma \ref{formulae}.ii and iii that for $(e,f)
\in P^{(2)}$ we have
\begin{enumerate}
 \item[7.] $d(e,f) = |e-f|$.
\end{enumerate}
With this notation we may rephrase Lemma \ref{add1} as follows.

\begin{corollary}\label{add3}
Let $0 < \epsilon < 1$; for any $(e_1,f_1)$ and $(e_2,f_2)$ in
$P^{(2)}$ such that
\begin{equation*}
d(e_1,e_2)\leq \epsilon\quad\text{and}\quad d(f_1,f_2) \leq
\epsilon
\end{equation*}
we have
\begin{equation*}
d(e_1 + f_1,e_2 + f_2) \leq \epsilon
\max(1,|e_1|,|e_2|,|f_1|,|f_2|)\ .
\end{equation*}
\end{corollary}

This corollary shows that the map $+ : P^{(2)} \longrightarrow B$
is continuous and extends continuously to the closure of $P^{(2)}$
in $B^2$ (cf.\ \cite{B-GT} II\S3.6 Prop.\ 11 and IX\S2.3).

\begin{lemma}\label{add4}
$P^{(n)}$, for any $n \in \mathbb{N}$, is dense in $B^n$ (for the
product topology).
\end{lemma}
\begin{proof}
We have to find, for any given $\epsilon > 0$ and $x_1,\ldots,x_n
\in S \times A$, elements $a_1,\ldots,a_n \in A$ and $s \in S$
such that
\begin{equation*}
d(x_j,(s,a_j)) < \epsilon \qquad\text{for any}\ 1 \leq j \leq n\ .
\end{equation*}
We may assume $n > 1$ and, by induction with respect to $n$, also
that $x_1 = (t,b_1),\ldots,x_{n-1} = (t,b_{n-1})$. Let $x_n =
(u,c)$ and choose $0 < \eta < 1$ such that $\eta\cdot
\mathop{\max}\limits_{1 \leq i \leq m} |c|_i|u|_i^{-1} <
\epsilon$. By Prop.\ \ref{approximative}.i there are elements
$(\sigma,\alpha),(\tau,\beta) \in S \times A$ such that
\begin{gather*}
 |\sigma t - \alpha u|_i \leq \eta |\sigma|_i|t|_i <
 |\sigma|_i|t|_i = |\alpha_i|_i|u|_i \\
 |c\tau - u\beta|_i
 \leq \eta |c|_i|\tau|_i \leq |c|_i|\tau|_i = |u|_i|\beta|_i
\end{gather*}
for any $1 \leq i \leq m$. In particular
\begin{align*}
 |\alpha c\tau - \sigma t\beta|_i & \leq \max(|\alpha|_i|c\tau -
 u\beta|_i,|\alpha u - \sigma t|_i|\beta|_i) \\
 & \leq \eta
 \max(|\alpha|_i|c|_i|\tau|_i,|\beta|_i|\sigma|_i|t|_i) \cr
 & = \eta |u|_i^{-1}|c|_i|\tau|_i|\sigma|_i|t|_i \ .
\end{align*}
We put $y_j := (\sigma t,\sigma b_j)$ for $1 \leq j \leq n-1$ and
$y_n := (\sigma t,\alpha c)$. According to Lemma \ref{formulae}.i
we have
\begin{equation*}
d(x_j,y_j) = 0 \qquad\text{for any}\ 1 \leq j \leq n-1 \ .
\end{equation*}
Furthermore
\begin{align*}
 d_i(x_n,y_n) & \leq
 \max(\Delta_i(x_n,(\tau,\beta)),\Delta_i(y_n,(\tau,\beta))) \\
 & = \max(|u|_i^{-1}|\tau|_i^{-1}|c\tau - u\beta|_i,
 |\sigma|_i^{-1}|t|_i^{-1}|\tau|_i^{-1}|\alpha c\tau - \sigma
 t\beta|_i) \\
 & \leq \eta |c|_i|u|_i^{-1} < \epsilon \ .
\end{align*}
\end{proof}

It follows in particular that our map $+ : P^{(2)} \longrightarrow
B$ extends continuously to a map
\begin{align*}
 B \times B & \longrightarrow  B \\
 (e,f) & \longmapsto  e+f
\end{align*}
which satisfies 1.-7. (for the associativity use Lemma \ref{add4}
for $n = 3$). We see that $|\ |$ is a nonarchimedean norm which
makes $B$ into a $K$-Banach space. The natural map
\begin{equation*}
(A,\max(|\ |_1,\ldots,|\ |_m)) \longrightarrow (B,|\ |)
\end{equation*}
is an isometry of normed $K$-vector spaces.

In order to construct the multiplication on $B$ we proceed in a
similar way.

\begin{lemma}\label{mult1}
Let $0 < \epsilon < 1$; for $s,t,u,v \in S$ and $a,b,c,d \in A$
such that
\begin{equation*}
d((s,at),(u,cv))\leq \epsilon\quad\text{and}\quad d((t,b),(v,d))
\leq \epsilon
\end{equation*}
we have
\begin{equation*}
 d((s,ab),(u,cd))
 \leq \epsilon \mathop{\max}\limits_{1 \leq i \leq m}
 \max(1,\frac{|at|_i}{|s|_i},\frac{|at|_i}{|s|_i} \frac{|b|_i}
 {|t|_i}, \frac{|b|_i}{|t|_i},
 \frac{|cv|_i}{|u|_i}, \frac{|cv|_i}{|u|_i} \frac{|d|_i}{|v|_i},
 \frac{|d|_i}{|v|_i})\ .
\end{equation*}
\end{lemma}
\begin{proof}
Let $0 \leq i \leq m$. We find $\xi = (\sigma,\alpha)$ and $\eta =
(\tau,\beta)$ in $S_i \times A$ such that
\begin{gather*}
 |b\sigma - t\alpha|_i \leq \epsilon |\sigma|_i|t|_i\ ,\ |d\sigma -
 v\alpha|_i \leq \epsilon |\sigma|_i|v|_i , \\
 |at\tau - s\beta|_i \leq \epsilon |\tau|_i|s|_i\ ,\ |cv\tau -
 u\beta|_i \leq \epsilon |\tau|_i|u|_i \ .
\end{gather*}
In particular, we have $|\sigma|_i^{-1}|\alpha|_i \leq
\max(|t|_i^{-1}|b|_i,\epsilon) \leq \max(|t|_i^{-1}|b|_i,1)$ and
also $|\sigma|_i^{-1}|\alpha|_i \leq \max(|v|_i^{-1}|d|_i,1)$.
Using Prop.\ \ref{approximative}.i (for $S_i$) we choose
$(\rho,\kappa) \in S_i \times A$ such that
\begin{equation*}
|\alpha\rho - \tau\kappa|_i \leq \epsilon|\alpha|_i|\rho|_i\ ,\
|\tau|_i|\kappa|_i = |\alpha|_i|\rho|_i \ .
\end{equation*}
We then have
\begin{align*}
 |ab\sigma\rho - s\beta\kappa|_i & = |a(b\sigma - t\alpha )\rho +
 at(\alpha\rho - \tau\kappa ) + (at\tau -s\beta )\kappa|_i \\
 & \leq \epsilon \max(|a|_i|\sigma|_i|t|_i|\rho|_i,
 |a|_i|t|_i|\alpha|_i|\rho|_i, |s|_i|\alpha|_i|\rho|_i)
\end{align*}
and hence
\begin{align*}
 \Delta_i((s,ab),(\sigma\rho,\beta\kappa)) & \leq \epsilon
 \max(|s|_i^{-1}|at|_i,|s|_i^{-1}|at|_i|\sigma|_i^{-1}|\alpha|_i,
 |\sigma|_i^{-1}|\alpha|_i) \\
 & \leq \epsilon \max(|s|_i^{-1}|at|_i,|s|_i^{-1}|at|_i|t|_i^{-1}|b|_i,
 |t|_i^{-1}|b|_i,1) \ .
\end{align*}
Similarly
\begin{align*}
 |cd\sigma\rho - u\beta\kappa|_i & = |c(d\sigma - v\alpha)\rho +
 cv(\alpha\rho - \tau\kappa) + (cv\tau - u\beta)\kappa|_i \\
 & \leq \epsilon \max(|c|_i|\sigma|_i|v|_i|\rho|_i,
 |c|_i|v|_i|\alpha|_i|\rho|_i, |u|_i|\alpha|_i|\rho|_i)
\end{align*}
and
\begin{align*}
 \Delta_i((u,cd),(\sigma\rho,\beta\kappa)) & \leq \epsilon
 \max(|u|_i^{-1}|cv|_i,|u|_i^{-1}|cv|_i|\sigma|_i^{-1}|\alpha|_i,
 |\sigma|_i^{-1}|\alpha|_i) \\
 & \leq \epsilon \max(|u|_i^{-1}|cv|_i,|u|_i^{-1}|cv|_i|v|_i^{-1}|d|_i,
 |v|_i^{-1}|d|_i,1) \ .
\end{align*}
We obtain
\begin{equation*}
 d_i((s,ab),(u,cd))
 \leq \epsilon
 \max(1,\frac{|at|_i}{|s|_i},\frac{|at|_i}{|s|_i} \frac{|b|_i}
 {|t|_i}, \frac{|b|_i}{|t|_i},
 \frac{|cv|_i}{|u|_i}, \frac{|cv|_i}{|u|_i} \frac{|d|_i}{|v|_i},
 \frac{|d|_i}{|v|_i})\ .
\end{equation*}
\end{proof}

\begin{corollary}\label{mult2}
For $s,t,u,v \in S$ and $a,b,c,d \in A$ such that $s^{-1}(at) =
u^{-1}(cv)$ and $t^{-1}b = v^{-1}d$ we have $s^{-1}(ab) =
u^{-1}(cd)$.
\end{corollary}

This corollary says that on the subset
\begin{equation*}
Q := \{(s^{-1}a,t^{-1}b) \in (P^{(1)})^2 : a \in At\}
\end{equation*}
of $B^2$ the map
\begin{align*}
 Q & \longrightarrow  B \\
 (e,f) & \longmapsto  e\cdot f := s^{-1}(ab)\ \ \text{if}\ e = s^{-1}(at),
 f = t^{-1}b
\end{align*}
is well defined.

\begin{corollary}\label{mult3}
Let $0 < \epsilon < 1$; for any $(e_1,f_1)$ and $(e_2,f_2)$ in $Q$
such that
\begin{equation*}
d(e_1,e_2)\leq \epsilon\quad\text{and}\quad d(f_1,f_2) \leq
\epsilon
\end{equation*}
we have
\begin{equation*}
d(e_1 \cdot f_1,e_2 \cdot f_2) \leq \epsilon
\max(1,|e_1|,|e_1||f_1|,|f_1|,|e_2|,|e_2||f_2|,|f_2|)\ .
\end{equation*}
\end{corollary}

This corollary shows that the map $\cdot : Q \longrightarrow B$ is
continuous and extends continuously to the closure of $Q$ in
$B^2$.

\begin{lemma}\label{mult4}
For any given $t \in S$ the set $\{s^{-1}(at) : a \in A, s \in
S\}$ is dense in $B$.
\end{lemma}
\begin{proof}
Let $(u,c) \in S \times A$ and let $\epsilon > 0$. By Prop.\
\ref{approximative}.i we find $(\sigma,\alpha) \in S \times A$
such that $|\sigma c - \alpha t|_i \leq \epsilon |\sigma|_i|c|_i$
for any $1 \leq i \leq m$. Using Lemma \ref{formulae}.i and ii we
obtain
\begin{align*}
 d(u^{-1}c,(\sigma u)^{-1}(\alpha t)) & = d((\sigma u)^{-1}(\sigma
 c),(\sigma u)^{-1}(\alpha t)) \cr
 & = \max_{1 \leq i \leq m}
 |u|_i^{-1}|\sigma|_i^{-1}|\sigma c - \alpha t|_i \cr
 & \leq \epsilon \max_{1 \leq i \leq m}
 |u|_i^{-1}|c|_i \ .
\end{align*}
\end{proof}

This lemma implies that $Q$ is dense in $B^2$. Hence by continuous
extension we have a map
\begin{align*}
 B \times B & \longrightarrow  B \\
 (e,f) & \longmapsto  e\cdot f \ .
\end{align*}

\begin{lemma}\label{mult5}
For any $e,f,g \in B$ we have:
\begin{enumerate}
 \item[i.] $d(e\cdot f,e\cdot g) \leq |e|d(f,g)$ and $e\cdot (f+g) = e\cdot
f + e\cdot g$;
 \item[ii.] $d(e\cdot g,f\cdot g) \leq d(e,f)|g|$ and
$(e+f)\cdot g = e\cdot g + f\cdot g$;
 \item[iii.] $(e\cdot f)\cdot g =
e\cdot (f\cdot g)$.
\end{enumerate}
\end{lemma}
\begin{proof}
By continuity all three assertions need only to be checked on an
appropriate dense subset of $B^3$.

i. As a consequence of Lemma \ref{add4} (for $n=2$) and Lemma
\ref{mult4} the set
\begin{equation*}
\{(s^{-1}(at),t^{-1}b,t^{-1}c) \in (P^{(1)})^3 : a,b,c \in A, s,t
\in S\}
\end{equation*}
is dense in $B^3$. For a triple in this set the first inequality
is immediate from Lemma \ref{formulae}.ii and iii and the second
identity follows from the definitions.

ii. As a consequence of Lemma \ref{mult4} the set
$\{(s_1^{-1}(a't),s_2^{-1}(b't),t^{-1}c)\}$ is dense in $B^3$. In
addition, the proof of Lemma \ref{add4} (for $n=2$) shows that a
pair $(s_1^{-1}(a't),s_2^{-1}(b't))$ can be approximated by a pair
of the form $(s^{-1}(at),s^{-1}(bt))$. Hence the set
\begin{equation*}
\{(s^{-1}(at),s^{-1}(bt),t^{-1}c) \in (P^{(1)})^3 : a,b,c \in A,
s,t \in S\}
\end{equation*}
is dense in $B^3$. For a triple in this set the first inequality
again is immediate from Lemma \ref{formulae}.ii and iii and the
second identity again follows from the definitions.

iii. By Lemma \ref{mult4} the set
\begin{equation*}
\{(s^{-1}(at),t^{-1}(bu),u^{-1}c) \in (P^{(1)})^3 : a,b,c \in A,
s,t,u \in S\}
\end{equation*}
is dense in $B^3$. On this subset the asserted associativity is
immediate from the definition of the multiplication.
\end{proof}

We see that the multiplication $\cdot$ in $B$ is distributive and
associative. It is easy to see, using Lemma \ref{formulae}.i and
the density of $Q$, that $1$ is a unit element and that this
multiplication is compatible with the scalar multiplication by
$K$. Finally, $|1| = 1$ and as a consequence of Lemma
\ref{mult5}1.16.i we have $|e\cdot f| \leq |e||f|$. Hence we
conclude that
\begin{equation*}
B\ \text{is a unital $K$-Banach algebra with submultiplicative
norm $|\ |$}
\end{equation*}
and that the natural map $(A,\max(|\ |_1,\ldots,|\ |_m))
\longrightarrow (B,|\ |)$ is an isometric unital homomorphism of
normed $K$-algebras. By construction we have $(s^{-1}1)\cdot s =
s^{-1}s = 1$ and $s\cdot (s^{-1}a) = a$ for any $a \in A$ and $s
\in S$. In particular, the elements of the multiplicative set $S$
become invertible in the ring $B$.

\begin{proposition}[Universal property]\label{univprop}
Let $(D,|\ |_D)$ be a unital $K$-Banach algebra and let $\phi : A
\longrightarrow D$ be any unital homomorphism of $K$-algebras such
that:
\begin{enumerate}
 \item[(i)] $\phi(s) \in D^\times$ for any $s \in S$;
 \item[(ii)] there is a constant $\gamma > 0$ such that
$|\phi(s)^{-1}\phi(a)|_D \leq \gamma \max_{1 \leq i \leq m}
|s|_i^{-1}|a|_i$ for any $s \in S, a \in A$ (in particular, $\phi$
is continuous);
\end{enumerate}
then there is a unique continuous unital homomorphism of
$K$-Banach algebras
\begin{equation*}
\phi_S : A\langle S;|\ |_1,\ldots,|\ |_m\rangle \longrightarrow D
\end{equation*}
such that $\phi_S|A = \phi$. If $|\ |_D$ is submultiplicative and
the constant in $(ii)$ can be chosen to be $\gamma = 1$ then
$\phi_S$ is norm decreasing.
\end{proposition}
\begin{proof}
The subset $P^{(1)} = \{s^{-1}a : a \in A, s \in S\}$ is dense in
$B$. Because of $s\cdot (s^{-1}a) = a$ we have to have
$\phi_S(s^{-1}a) = \phi(s)^{-1}\phi(a)$. Hence the uniqueness of
$\phi_S$ is clear. To establish existence let $\gamma_0 > 0$ be a
constant such that $|d_1d_2|_D \leq \gamma_0 |d_1|_D|d_2|_D$ for
any $d_1,d_2 \in D$. We claim that there is a constant $\gamma_1
> 0$ such that
\begin{equation*}
|\phi(s)^{-1}\phi(a) - \phi(t)^{-1}\phi(b)|_D \leq
\gamma_1|s^{-1}a - t^{-1}b|
\end{equation*}
for any $(s,a),(t,b) \in S \times A$. Choose $\eta > 0$ such that
$\eta|b| \leq \eta |t||t^{-1}||b| \leq |s^{-1}|^{-1}|s^{-1}a -
t^{-1}b|$. Since $P^{(1)}$ is dense in $B$ we find
$(\sigma,\alpha) \in S \times A$ such that
\begin{equation*}
|st^{-1} - \sigma^{-1}\alpha| \leq \eta \ .
\end{equation*}
We then have
\begin{equation*}
(\sigma s)^{-1}(\sigma a - \alpha b) - (\sigma s)^{-1}(\sigma s -
\alpha t)t^{-1}b = s^{-1}a - t^{-1}b
\end{equation*}
and
\begin{equation*}
|(\sigma s)^{-1}(\sigma s - \alpha t)t^{-1}b| \leq |s^{-1}|\eta
|b| \leq |s^{-1}a - t^{-1}b| \ .
\end{equation*}
It follows that $|(\sigma s)^{-1}(\sigma a - \alpha b)| \leq
|s^{-1}a - t^{-1}b|$ and hence, by (ii), that
\begin{equation*}
|\phi(\sigma s)^{-1}\phi(\sigma a - \alpha b)|_D \leq \gamma
|s^{-1}a - t^{-1}b| \ .
\end{equation*}
On the other hand
\begin{equation*}
\begin{split}
 & |\phi(\sigma s)^{-1}\phi(\sigma s - \alpha
 t)\phi(t)^{-1}\phi(b)|_D \\
 & \qquad \leq \gamma_0 |\phi(\sigma s)^{-1}\phi(\sigma s - \alpha
 t)|_D|\phi(t)^{-1}\phi(b)|_D \\
 & \qquad \leq \gamma_0 \gamma^2 |(\sigma s)^{-1}(\sigma s - \alpha
 t)||t^{-1}b| \\
 & \qquad \leq \gamma_0 \gamma^2 |(\sigma s)^{-1}(\sigma s - \alpha
 t)t^{-1}||t||t^{-1}||b| \\
 & \qquad \leq \gamma_0 \gamma^2 |s^{-1}|\eta |t||t^{-1}||b| \\
 & \qquad \leq \gamma_0 \gamma^2 |s^{-1}a - t^{-1}b| \ .
\end{split}
\end{equation*}
Setting $\gamma_1 := \max(\gamma,\gamma_0\gamma^2)$ we finally
obtain
\begin{equation*}
\begin{split}
 & |\phi(s)^{-1}\phi(a) - \phi(t)^{-1}\phi(b)|_D \\
 & \qquad = |\phi(\sigma s)^{-1}\phi(\sigma a - \alpha b) - \phi(\sigma s)^{-1}\phi(\sigma s - \alpha
 t)\phi(t)^{-1}\phi(b)|_D \\
 & \qquad \leq \max(\gamma |s^{-1}a - t^{-1}b|, \gamma_0\gamma^2 |s^{-1}a -
 t^{-1}b|) \\
 & \qquad = \gamma_1 |s^{-1}a - t^{-1}b| \ .
\end{split}
\end{equation*}
This proves our claim and implies that the map
\begin{align*}
 S \times A & \longrightarrow  D \\
 (s,a) & \longmapsto \phi(s)^{-1}\phi(a)
\end{align*}
is uniformly continuous and therefore extends continuously to a
map $\phi_S : B \longrightarrow D$. It is straightforward to check
that $\phi_S$ is a unital homomorphism of $K$-algebras such that
$\phi_S|A = \phi$.
\end{proof}

Although the norm $|\ |$ on $B$ is only submultiplicative it
nevertheless is quasi-abelian in the following sense. We fix a
constant $0 < \gamma < 1$ so that the condition (qa) is satisfies
simultaneously for any $|\ |_i$.

\begin{lemma}\label{qa-lem1}
$|u^{-1}s^{-1}ec - u^{-1}es^{-1}c| \leq \gamma\cdot
|u^{-1}s^{-1}ec| = \gamma\cdot |u^{-1}es^{-1}c|$ for any $u,s \in
S$, $c \in A$, and $e \in B$.
\end{lemma}
\begin{proof}
By Lemma \ref{mult4} we may assume that $e = t^{-1}(bs)$ with $t
\in S$ and $b \in A$. We then compute
\begin{equation*}
    \begin{split}
       & |u^{-1}s^{-1}t^{-1}bsc - u^{-1}t^{-1}bc| \\
       & \qquad = |u^{-1}s^{-1}t^{-1}bsc - u^{-1}t^{-1}s^{-1}bsc +
       u^{-1}t^{-1}s^{-1}bsc - u^{-1}t^{-1}s^{-1}sbc| \\
       & \qquad \leq \max (d((tsu,bsc),(stu,bsc)),
       d((stu,bsc),(stu,sbc))) \\
       & \qquad = \max_i \max (|tsu|_i^{-1}|stu|_i|bsc|_i |tsu -
       stu|_i, |stu|_i^{-1}|bsc - sbc|_i) \\
       & \qquad \leq \gamma \cdot \max_i |stu|_i^{-1} |bsc|_i =
       \gamma\cdot \max_i |tu|_i^{-1} |bc|_i \\
       & \qquad = \gamma\cdot |u^{-1}t^{-1}bc|
     \end{split}
\end{equation*}
using Lemma \ref{formulae} in the fourth line.
\end{proof}

\begin{lemma}\label{qa-lem2}
$|s^{-1}aec - s^{-1}eac| \leq \gamma\cdot |s^{-1}aec| =
\gamma\cdot |s^{-1}eac|$ for any $s \in S$, $a,c \in A$, and $e
\in B$.
\end{lemma}
\begin{proof}
By density we may assume that $e = t^{-1}b$ with $t \in S$ and $b
\in A$. We then compute
\begin{equation*}
    \begin{split}
      & |s^{-1}at^{-1}bc - s^{-1}t^{-1}bac| \\
      & \qquad = |s^{-1}at^{-1}bc - s^{-1}t^{-1}abc + s^{-1}t^{-1}abc
      - s^{-1}t^{-1}bac| \\
      & \qquad \leq \max (|s^{-1}at^{-1}bc - s^{-1}t^{-1}abc|,
      |s^{-1}t^{-1}abc - s^{-1}t^{-1}bac|) \\
      & \qquad \leq \max(\gamma\cdot |s^{-1}t^{-1}abc|, d((ts,abc),
      (ts,bac))) \\
      & \qquad = \max (\gamma\cdot |s^{-1}t^{-1}abc|, \max_i
      |ts|_i^{-1}|abc - bac|_i) \\
      & \qquad \leq \gamma\cdot\max (|s^{-1}t^{-1}abc|, \max_i
      |ts|_i^{-1}|abc|_i) \\
      & \qquad = \gamma\cdot |s^{-1}t^{-1}abc| \\
      & \qquad = \gamma\cdot |s^{-1}at^{-1}bc|
    \end{split}
\end{equation*}
using the previous lemma in the fourth and the last line.
\end{proof}

\begin{proposition}\label{qa}
There is a constant $0 < \gamma < 1$ such that, for any elements
$e_1,\ldots,e_n \in B$ and any permutation $\sigma$ of
$\{1,\ldots,n\}$, we have
\begin{equation*}
    |e_1\cdot\ldots\cdot e_n - e_{\sigma(1)}\cdot\ldots\cdot
    e_{\sigma(n)}| \leq \gamma \cdot |e_1\cdot\ldots\cdot e_n| =
    \gamma \cdot |e_{\sigma(1)}\cdot\ldots\cdot e_{\sigma(n)}| \ .
\end{equation*}
\end{proposition}
\begin{proof}
Since any permutation is a product of elementary transpositions it
suffices to prove that
\begin{equation*}
    |ge_0e_1f - ge_1e_0f| \leq \gamma\cdot |ge_0e_1f| = \gamma \cdot
    |ge_1e_0f|
\end{equation*}
for any $g,e_0,e_1,f \in B$. Since
\begin{equation*}
\begin{split}
    ge_0e_1f - ge_1e_0f & = g(e_0e_1f) - (e_0e_1f)g + e_0e_1(fg) -
    e_1e_0(fg) \\
    & \quad + (e_1e_0f)g - g(e_1e_0f)
\end{split}
\end{equation*}
we may in fact assume that $g = 1$. By density it furthermore
suffices to consider the case $f = u^{-1}c$ with $u \in S$ and $c
\in A$. Using the identity
\begin{equation*}
\begin{split}
    e_0e_1u^{-1}c - e_1e_0u^{-1}c  & = (e_0e_1)u^{-1}c -
    u^{-1}(e_0e_1)c + u^{-1}e_0e_1c - u^{-1}e_1e_0c \\
    & \quad + u^{-1}(e_1e_0)c - e_1e_0u^{-1}c
\end{split}
\end{equation*}
we therefore are reduced to showing that
\begin{equation*}
    |u^{-1}e_0e_1c - u^{-1}e_1e_0c| \leq \gamma \cdot
    |u^{-1}e_0e_1c| = \gamma \cdot |u^{-1}e_1e_0c|
\end{equation*}
for any $u \in S$, $c \in A$, and $e_0,e_1 \in B$. But for $e_0 =
s^{-1}a$ with $s \in S$ and $a \in A$ this is a consequence of
Lemmas \ref{qa-lem1} and \ref{qa-lem2}. The case of a general
$e_0$ then follows by density.
\end{proof}

\begin{remark}\label{multiplicative}
For a single initial quasi-abelian norm on $A$ (i.e., $m = 1$) the
resulting norm $|\ |$ on $B$ again is multiplicative (and
quasi-abelian).
\end{remark}
\begin{proof}
We need to show that $|ef| = |e||f|$. By continuity it suffices to
consider the case $e = s^{-1}a$ and $f = t^{-1}b$ with $s,t \in S$
and $a,b \in A$. But then, using Prop.\ \ref{qa} and Lemma
\ref{formulae}.iii, we compute
\begin{equation*}
    |s^{-1}at^{-1}b| = |(ts)^{-1}ab| = |ts|_1^{-1}|ab|_1 =
    |s|_1^{-1}|a|_1|t|_1^{-1}|b|_1 = |s^{-1}a||t^{-1}b| \ .
\end{equation*}
\end{proof}

\subsection{Noncommutative annuli for uniform pro-$p$-groups}

For the rest of this paper we assume that $\mathbb{Q}_p \subseteq
K \subseteq \mathbb{C}_p$ is discretely valued. For simplicity we
also assume that $p \neq 2$. We fix a uniform pro-$p$-group $G$ as
well as an ordered basis $h_1,\ldots,h_d$ of $G$. Then the map
\begin{align*}
 \psi : \qquad\qquad \mathbb{Z}_p^d &
\xrightarrow{\;\sim\;}  G \\
 (x_1,\ldots x_d) & \longmapsto
h_1^{x_1}\cdot\ldots\cdot h_d^{x_d}
\end{align*}
is a bijective global chart for $G$ as a locally
$\mathbb{Q}_p$-analytic manifold (cf.\ \cite{DDMS} \S4.2). Using
this chart we may identify the $K$-Fr\'echet spaces of locally
analytic distributions
\begin{equation*}
\psi^{\ast} : D(G,K) \xrightarrow{\;\cong\;} D(\mathbb{Z}_p^d,K)
\end{equation*}
As usual, we embed the group $G$ into the algebra $D(G,K)$ via the
Dirac distributions. We write $b_i := h_i - 1$ and $\bb^{\alpha}
:= b_1^{\alpha_1}b_2^{\alpha_2}\cdots b_{d}^{\alpha_d}$, for
$\alpha = (\alpha_1,\ldots,\alpha_d) \in \mathbb{N}_0^d$. Any
distribution $\mu \in D(G,K)$ has a unique convergent expansion
\begin{equation*}
\mu = \sum_{\alpha\in\mathbb{N}_0^{d}} d_{\alpha}\bb^{\alpha}
\end{equation*}
with $d_\alpha \in K$ such that, for any $0 < r < 1$, the set
$\{|d_\alpha|r^{\alpha}\}_{\alpha\in\mathbb{N}_0^d}$ is bounded.
Here we put $r^\alpha := r^{\alpha_1 + \ldots + \alpha_d}$.
Conversely, any such series is convergent in $D(G,K)$. The
Fr\'echet topology on $D(G,K)$ is defined by the family of norms
\begin{equation*}
|\mu|_r := \mathop{\rm sup}\limits_{\alpha\in\mathbb{N}_0^d}
|d_\alpha|r^{\alpha}
\end{equation*}
for $0 < r < 1$. It is shown in \cite{foundations} Thm.\ 4.5 that,
for any $1/p < r < 1$ such that $r \in p^{\mathbb{Q}}$, the norm
$|\ |_r$ on $D(G,K)$ is multiplicative and quasi-abelian. Whereas
$D(G,K)$ is a ``noncommutative open unit disk'' the completion
$D_r(G,K)$ of $D(G,K)$ with respect to $|\ |_r$ is a
``noncommutative closed disk of radius $r$''. We now use the
technique of the previous section to construct corresponding
``noncommutative closed annuli'' for any radii $1/p < r_0 \leq r <
1$ such that $r_0,r \in p^{\mathbb{Q}}$.

We take $A := D_r(G,K)$, the two norms $|\ |_{r_0}$ and $|\ |_r$,
and the multiplicatively closed subset $S \subseteq D_r(G,K)$
generated by $b_1,\ldots,b_d$. We define
\begin{equation*}
D_{[r_0,r]}(G,K) := A\langle S;|\ |_{r_0},|\ |_r\rangle
\end{equation*}
and we let $|\ |_{r_0,r}$ denote its natural norm. Since $\max(|\
|_{r_0},|\ |_r) = |\ |_r$ the canonical homomorphism of $K$-Banach
algebras $D_r(G,K) \longrightarrow D_{[r_0,r]}(G,K)$ is an
isometry for the natural norms.

\begin{lemma}\label{convex}
For $1/p < r \leq r' \leq r'' < 1$, any integer $m \geq 0$, and
any $\mu \in D(G,K)$ we have
\begin{equation*}
\frac{|\mu|_{r'}}{r'^m} \leq \max(\frac{|\mu|_r}{r^m},
\frac{|\mu|_{r''}}{r''^m}) \ .
\end{equation*}
\end{lemma}
\begin{proof}
Since exponential function and logarithm are monotonous functions
it suffices to show that the function
\begin{equation*}
t \longmapsto {\rm log} (\frac{|\mu|_{\exp(t)}}{\exp(t)^m})
\end{equation*}
on $(- \infty,0)$ is convex. But it is a supremum
\begin{equation*}
t \longmapsto \sup_\alpha (\log |d_\alpha| + (|\alpha| - m)t)
\end{equation*}
of affine functions and hence is visibly convex.
\end{proof}

Let $1/p \leq r_0 < r \leq r' \leq r'' < r_1 \leq 1$ all contained
in $p^{\mathbb{Q}}$ and consider the composed unital homomorphism
of $K$-Banach algebras
\begin{equation*}
D_{r''}(G,K) \longrightarrow D_r(G,K) \longrightarrow
D_{[r,r']}(G,K)
\end{equation*}
which is norm decreasing and makes the elements in $S$ invertible.
Let $\mu \in D_{r'}(G,K)$ and suppose that the monomial $s$ in the
$b_1,\ldots,b_d$ has $m$ factors. Using Lemma \ref{convex} we
compute
\begin{align*}
 |s^{-1}\mu|_{r,r'} & = \max(|s|_r^{-1}|\mu|_r,
 |s|_{r'}^{-1}|\mu|_{r'}) = \max(\frac{|\mu|_r}{r^m},
 \frac{|\mu|_{r'}}{r'^m}) \\
 & \leq \max(\frac{|\mu|_r}{r^m},
 \frac{|\mu|_{r''}}{r''^m}) = \max(|s|_r^{-1}|\mu|_r,
 |s|_{r''}^{-1}|\mu|_{r''}) \ .
\end{align*}
This shows that the assumptions of the universal property Prop.\
\ref{univprop} are satisfied. The above composed homomorphism
extends uniquely to a norm decreasing unital homomorphism of
$K$-Banach algebras
\begin{equation*}
D_{[r,r'']}(G,K) \longrightarrow D_{[r,r']}(G,K) \ .
\end{equation*}
We then may pass  to the projective limit with respect to $r''$
and obtain the $K$-Fr\'echet algebra
\begin{equation*}
D_{[r,r_1)}(G,K) := \varprojlim_{r \leq r'' < r_1}
D_{[r,r'']}(G,K)
\end{equation*}
representing a ``noncommutative half open annulus''.

A similar argument shows that the natural homomorphism
$D_{r''}(G,K) \longrightarrow D_{[r',r'']}(G,K)$ extends uniquely
to a norm decreasing unital homomorphism of $K$-Banach algebras
\begin{equation*}
D_{[r,r'']}(G,K) \longrightarrow D_{[r',r'']}(G,K) \ .
\end{equation*}
Again we obtain a $K$-Fr\'echet algebra
\begin{equation*}
D_{(r_0,r'']}(G,K) := \varprojlim_{r_0 < r \leq r''}
D_{[r,r'']}(G,K) \ .
\end{equation*}
It also follows that the Fr\'echet algebras $D_{[r,r_1)}(G,K)$
form an inductive system with respect to $r$. Especially in the
case $r_1 = 1$ the inductive limit
\begin{equation*}
R(G,K) := \varinjlim D_{[r,1)}(G,K)
\end{equation*}
is a locally convex unital $K$-algebra which we call the
\textit{Robba ring} of $G$. Although we suppress this in the
notation this ring does depend on the initial choice of a basis
$\{h_1,\ldots,h_d\}$ of $G$ but not on its ordering. It is shown
in \cite{foundations} in the discussion following Thm.\ 4.10 that,
for $1/p \leq r < 1$, the norm $\|\ \|_r$ on $D(G,K)$ is
completely independent of the choice of the ordered basis. It
follows that each $D_{[r,r']}(G,K)$ together with its norm $\|\
\|_{r,r'}$ as well as the topological algebras $D_{[r,r_1)}(G,K)$,
$D_{(r_0,r'']}(G,K)$, and $R(G,K)$ do not depend on the ordering
of the chosen basis of $G$. In fact a little more is true. In the
ring $\mathbb{Z}_p[[Z]]$ of formal power series in one variable
over $\mathbb{Z}_p$ we have, for any $x \in \mathbb{Z}_p$, the
identity
\begin{equation}\label{3}
    (1 + Z)^x - 1 = \sum_{i \geq 1} {x \choose i} Z^i = Z(x +
    Zf_x(Z)) \qquad \textrm{with $f_x \in \mathbb{Z}_p[[Z]]$}.
\end{equation}
If $x \in \mathbb{Z}_p^\times$ then $x + Zf_x(Z)$ is a unit in
$\mathbb{Z}_p[[Z]]$. This shows that for $x \in
\mathbb{Z}_p^\times$ we have $h_i^x - 1 \in b_i \cdot
D(G,K)^\times$ for any $1 \leq i \leq d$. Applying the universal
property Prop.\ \ref{univprop} we conclude that replacing
$\{h_1,\ldots,h_d\}$ by $\{h_1^{x_1},\ldots,h_d^{x_d}\}$ for some
$x_1,\ldots,x_d \in \mathbb{Z}_p^\times$ does not change the
Banach algebras $D_{[r,r']}(G,K)$ together with their norm $\|\
\|_{r,r'}$ and hence does not change the topological algebras
$D_{[r,r_1)}(G,K)$, $D_{(r_0,r'']}(G,K)$, and $R(G,K)$ either.

In order to be able to work with these rings we will show that
their elements can be viewed as Laurent series. For that we
introduce the affinoid domain
\begin{equation*}
    X^d_{[r,r']} := \{(z_1,\ldots,z_d) \in \mathbb{C}_p^d : r
    \leq |z_1| = \ldots = |z_d| \leq r'\} \ .
\end{equation*}
The ring $\mathcal{O}_K(X^d_{[r,r']})$ of $K$-analytic functions
on $X^d_{[r,r']}$ is the ring of all Laurent series
\begin{equation*}
    f(Z_1,\ldots,Z_d) = \sum_{\alpha \in \mathbb{Z}^d} d_\alpha
    \mathbf {Z}^\alpha
\end{equation*}
with $d_\alpha \in K$ and such that $\lim_{|\alpha| \rightarrow
\infty} |d_\alpha| \rho^\alpha = 0$ for any $r \leq \rho \leq r'$.
Here
\begin{equation*}
    \mathbf{Z}^\alpha := Z_1^{\alpha_1}\cdot\ldots\cdot
Z_d^{\alpha_d} \qquad\text{and}\qquad |\alpha| := |\alpha_1| +
\ldots + |\alpha_d|
\end{equation*}
(with the abuse of notation that $|\alpha_i|$ exceptionally means
the archimedean absolute value). Since $\rho^\alpha \leq
\max(r^\alpha ,r'^\alpha)$ for any $r \leq \rho \leq r'$ and any
$\alpha \in \mathbb {Z}^d$ the latter condition on $f$ is
equivalent to
\begin{equation*}
    \lim_{|\alpha| \rightarrow
\infty} |d_\alpha| r^\alpha = \lim_{|\alpha| \rightarrow \infty}
|d_\alpha| r'^\alpha = 0 \ .
\end{equation*}
The spectral norm on the affinoid algebra
$\mathcal{O}_K(X^d_{[r,r']})$ is given by
\begin{align*}
    |f|_{X^d_{[r,r']}} & = \sup_{r \leq \rho \leq r'} \max_{\alpha
    \in \mathbb{Z}^d} |d_\alpha|\rho^\alpha \\
    & = \max( \max_{\alpha \in \mathbb{Z}^d} |d_\alpha|r^\alpha,
    \max_{\alpha \in \mathbb{Z}^d} |d_\alpha| r'^\alpha) \ .
\end{align*}
Setting $\bb^\alpha := b_1^{\alpha_1} \cdot\ldots\cdot
b_d^{\alpha_d}$ for any $\alpha = (\alpha_1,\ldots,\alpha_d) \in
\mathbb{Z}^d$ we claim that $f(b_1,\ldots,b_d) := \sum_{\alpha \in
\mathbb{Z}^d} d_\alpha \bb^\alpha$ converges in $D_{[r,r']}(G,K)$.
As a consequence of Prop.\ \ref{qa} and Lemma \ref{formulae}.iii
we have
\begin{equation*}
    |\bb^\alpha|_{r,r'} = \max(r^\alpha,r'^\alpha)
\end{equation*}
for any $\alpha \in \mathbb{Z}^d$. Hence
\begin{equation*}
    \lim_{|\alpha| \rightarrow
\infty} |d_\alpha\bb^\alpha|_{r,r'} = \lim_{|\alpha| \rightarrow
\infty} \max(|d_\alpha| r^\alpha, |d_\alpha| r'^\alpha) = \max(
\lim_{|\alpha| \rightarrow \infty} |d_\alpha| r^\alpha,
\lim_{|\alpha| \rightarrow \infty} |d_\alpha| r'^\alpha) = 0 \ .
\end{equation*}
Therefore
\begin{align*}
    \mathcal{O}_K(X^d_{[r,r']}) & \longrightarrow
    D_{[r,r']}(G,K) \\
    f & \longmapsto f(b_1,\ldots,b_d)
\end{align*}
is a well defined $K$-linear map. In order to investigate this map
we introduce the filtration
\begin{equation*}
    F^iD_{[r,r']}(G,K) := \{ e \in D_{[r,r']}(G,K) : |e|_{r,r'}
    \leq |p|^i \} \qquad\text{for $i \in \mathbb{R}$}
\end{equation*}
on $D_{[r,r']}(G,K)$. Since $K$ is discretely valued and $r,r' \in
p^{\mathbb{Q}}$ this filtration is quasi-integral in the sense of
\cite{foundations} \S1. The corresponding graded ring $gr^\cdot
D_{[r,r']}(G,K)$, by Prop.\ \ref{qa}, is commutative. We let
$\sigma(e) \in gr^\cdot D_{[r,r']}(G,K)$ denote the principal
symbol of any element $e \in D_{[r,r']}(G,K)$.

\begin{proposition}\label{expansion}
\begin{itemize}
    \item[i.] $gr^\cdot D_{[r,r']}(G,K)$ is a free $gr^\cdot
K$-module with basis $\{ \sigma(\bb^\alpha) : \alpha \in
\mathbb{Z}^d \}$.
    \item[ii.] The map
\begin{align*}
    \mathcal{O}_K(X^d_{[r,r']}) & \xrightarrow{\;\cong\;}
    D_{[r,r']}(G,K) \\
    f & \longmapsto f(b_1,\ldots,b_d)
\end{align*}
is a $K$-linear isometric bijection.
\end{itemize}
\end{proposition}
\begin{proof}
 Since $\{ s^{-1}\mu
: s \in S , \mu \in D_r(G,K) \}$ is dense in $D_{[r,r']}(G,K)$
every element in the graded ring $gr^\cdot D_{[r,r']}(G,K)$ is of
the form $\sigma(s^{-1}\mu)$. Suppose that $\mu =
\sum_{\alpha\in\mathbb{N}_0^{d}} d_{\alpha}\bb^{\alpha}$ and $s =
\bb^\beta$ for some $\beta \in \mathbb{N}_0^d$. Then $s^{-1}\mu =
\sum_{\alpha\in\mathbb{N}_0^{d}}
d_{\alpha}\bb^{-\beta}\bb^{\alpha}$ and, using Lemma
\ref{formulae}.iii we compute
\begin{align*}
    |s^{-1}\mu|_{r,r'} & = \max( |s|_r^{-1} |\mu|_r,
    |s|_{r'}^{-1}|\mu|_{r'}) \\
    & = \max( \max_{\alpha \in \mathbb{N}_0^d} |d_\alpha| r^{\alpha
    - \beta}, \max_{\alpha \in \mathbb{N}_0^d} |d_\alpha| r'^{\alpha
    - \beta}) \\
    & = \max_{\alpha \in \mathbb{N}_0^d} |d_\alpha| \max(r^{\alpha
    - \beta}, r'^{\alpha - \beta}) \\
    & = \max_{\alpha \in \mathbb{N}_0^d} |d_\alpha|
    |\bb^{-\beta}\bb^\alpha|_{r,r'} \ .
\end{align*}
It follows that $gr^\cdot D_{[r,r']}(G,K)$ as a $gr^\cdot
K$-module is generated by the principal symbols
$\sigma(\bb^{-\beta}\bb^\alpha)$ with $\alpha, \beta \in
\mathbb{N}_0^d$. But it also follows that, for a fixed $\beta \in
\mathbb{N}_0^d$, the principal symbols
$\sigma(\bb^{-\beta}\bb^\alpha)$ with $\alpha$ running over
$\mathbb{N}_0^d$ are linearly independent over $gr^\cdot K$. By
Prop.\ \ref{qa} we may permute the factors in
$\sigma(\bb^{-\beta}\bb^\alpha)$ arbitrarily. Hence $gr^\cdot
D_{[r,r']}(G,K)$ as a $gr^\cdot K$-module also is generated by the
principal symbols $\sigma(\bb^\alpha)$ with $\alpha$ running over
$\mathbb{Z}^d$. Since any given finite number of the latter can be
written in the form $\sigma(\bb^{-\beta}\bb^\alpha)$ with $\alpha
\in \mathbb{N}_0^d$ and a joint $\beta \in \mathbb{N}_0^d$ we in
fact obtain that $gr^\cdot D_{[r,r']}(G,K)$ is a free $gr^\cdot
K$-module with basis $\{ \sigma(\bb^\alpha) : \alpha \in
\mathbb{Z}^d \}$.

On the other hand, we of course have
\begin{align*}
    f(b_1,\ldots,b_d) & \leq \max_{\alpha \in \mathbb{Z}^d}
    |d_\alpha||\bb^\alpha|_{r,r'} \\
    & = \max_{\alpha \in \mathbb{Z}^d} |d_\alpha|
    \max(r^\alpha,r'^\alpha) = \max( \max_{\alpha \in \mathbb{Z}^d}
    |d_\alpha|r^\alpha, \max_{\alpha \in \mathbb{Z}^d}
    |d_\alpha|r'^\alpha) \\
    & = |f|_{X^d_{[r,r']}} \ .
\end{align*}
This means that if we introduce on $\mathcal{O}_K(X^d_{[r,r']})$
the filtration defined by the spectral norm then the asserted map
respects the filtrations, and by the above reasoning it induces an
isomorphism between the associated graded rings. Hence, by
completeness of these filtrations, it is an isometric bijection.
\end{proof}

If $G$ is commutative the map in Prop.\ \ref{expansion}, of
course, is an isometric isomorphism of Banach algebras. But in
general it is very far from being multiplicative. The graded ring
$gr^\cdot D_{[r,r']}(G,K)$ is discussed further in section
\ref{sec:Frechet-Stein} of the main paper.

One useful consequence of Prop.\ \ref{expansion} is that we have
\begin{equation*}
    |\ |_{r,r'} = \max (|\ |_{r,r},|\ |_{r',r'})
\end{equation*}
on $D_{[r,r']}(G,K)$.

The ring $\mathcal{O}_K(X^d_{[r,1)})$ of $K$-analytic functions on
the rigid variety
\begin{equation*}
    X^d_{[r,1)} := \{(z_1,\ldots,z_d) \in \mathbb{C}_p^d : r
    \leq |z_1| = \ldots = |z_d| < 1\}
\end{equation*}
is the Fr\'echet algebra of all Laurent series
\begin{equation*}
    f(Z_1,\ldots,Z_d) = \sum_{\alpha \in \mathbb{Z}^d} d_\alpha
    \mathbf{Z}^\alpha
\end{equation*}
with $d_\alpha \in K$ and such that $\lim_{|\alpha| \rightarrow
\infty} |d_\alpha| \rho^\alpha = 0$ for any $r \leq \rho < 1$. We
also introduce the locally convex $K$-algebra
\begin{equation*}
     \mathcal{R}_K^d := \varinjlim_r
     \mathcal{O}_K(X^d_{[r,1)}) \ .
\end{equation*}
By limit arguments the map in Prop.\ \ref{expansion} extends to
$K$-linear topological isomorphisms
\begin{equation*}
    \mathcal{O}_K(X^d_{[r,1)}) \xrightarrow{\;\cong\;}
    D_{[r,1)}(G,K)
\end{equation*}
and
\begin{equation*}
    \mathcal{R}_K^d \xrightarrow{\;\cong\;} R(G,K) \ .
\end{equation*}

The coefficients $c_{\beta\gamma,\alpha} \in \mathbb{Q}_p$ in the
expansions
\begin{equation*}
    \bb^\beta \bb^\gamma = \sum_{\alpha \in \mathbb{Z}^d}
    c_{\beta\gamma,\alpha} \bb^\alpha \qquad\textrm{for any $\beta, \gamma \in \mathbb{Z}^d$}
\end{equation*}
can be viewed in any of the rings under consideration.

\begin{lemma}\label{coefficients}
$|c_{\beta\gamma,\alpha}| \leq \min(1, p^{\alpha - \beta -
\gamma})$ for any $\alpha \neq \beta + \gamma$, and
$|c_{\beta\gamma,\beta + \gamma} - 1| < 1$.
\end{lemma}
\begin{proof}
By Prop.\ \ref{qa} the coefficients of the expansion
\begin{equation*}
    \bb^\beta \bb^\gamma - \bb^{\beta + \gamma} = \sum_{\alpha \in \mathbb{Z}^d}
    c'_{\beta\gamma,\alpha} \bb^\alpha
\end{equation*}
satisfy
\begin{equation*}
    |c'_{\beta\gamma,\alpha}| r^\alpha < |\bb^{\beta
    + \gamma}|_{r,r} = r^{\beta + \gamma}
\end{equation*}
for any $1/p < r < 1$ in $p^{\mathbb{Q}}$. Hence
$|c'_{\beta\gamma,\alpha}| < r^{\beta + \gamma - \alpha}$. By
letting tend $r$ to $1$ and $p^{-1}$, respectively, we obtain
$|c'_{\beta\gamma,\alpha}| \leq \min(1, p^{\alpha - \beta -
\gamma})$. For $\alpha = \beta +\gamma$ this means
$|c'_{\beta\gamma,\beta + \gamma}| < 1$.
\end{proof}

\subsection{Bounded rings}

In $\mathcal{R}_K^d$ we have the subrings
\begin{equation*}
    \mathcal{R}_K^{d,b} := \{ f = \sum_{\alpha \in \mathbb{Z}^d}
    d_\alpha \mathbf {Z}^\alpha \in \mathcal{R}_K^d : \|f\| :=
    \sup_\alpha |d_\alpha| < \infty \}
\end{equation*}
and
\begin{equation*}
    \mathcal{R}_K^{d,int} := \{ f \in \mathcal{R}_K^d : \|f\| \leq 1
    \} \ .
\end{equation*}
It is well known that the norm $\|\ \|$ on $\mathcal{R}_K^{d,b}$
is multiplicative. We let $\mathcal{E}_K^d$ denote the completion
of $\mathcal{R}_K^{d,b}$ with respect to $\|\ \|$ obtaining a
$K$-Banach algebra.

\begin{lemma}\label{series}
$\mathcal{E}_K^d$ is the ring of all formal series $\sum_{\alpha
\in \mathbb{Z}^d} d_\alpha \mathbf{Z}^\alpha$ such that
$\sup_\alpha |d_\alpha| < \infty$ and $\lim_{\sum \alpha_i \leq m,
|\alpha| \rightarrow \infty} |d_\alpha| = 0$ for any $m \in
\mathbb{N}$.
\end{lemma}
\begin{proof}
Let $\widetilde{\mathcal{E}}$ denote the vector space of all these
formal series in the assertion. It is easy to see that
$\widetilde{\mathcal{E}}$ is complete with respect to $\|\ \|$. We
also need the subspace $\widetilde{\mathcal{E}}_0 \subseteq
\widetilde{\mathcal{E}}$ of all formal series $\sum_{\alpha \in
\mathbb{Z}^d} d_\alpha \mathbf{Z}^\alpha$ such that, for any $m
\in \mathbb{N}$, there are only finitely many $\alpha$ with
$\sum_i \alpha_i \leq m$ and $d_\alpha \neq 0$.

In a first step we consider an arbitrary element $\sum_{\alpha \in
\mathbb{Z}^d} d_\alpha \mathbf{Z}^\alpha$ in $\mathcal{R}_K^d$.
Then $\lim_{|\alpha| \rightarrow \infty} |d_\alpha| r^\alpha = 0$
for some $0 < r < 1$. It follows that
\begin{equation*}
    \lim_{\sum \alpha_i \leq 0, |\alpha| \rightarrow \infty}
    |d_\alpha| \leq \lim_{\sum \alpha_i \leq 0, |\alpha| \rightarrow \infty}
    |d_\alpha| r^\alpha = 0
\end{equation*}
and
\begin{equation*}
    r^m \lim_{\sum \alpha_i = m, |\alpha| \rightarrow \infty}
    |d_\alpha| = \lim_{\sum \alpha_i = m, |\alpha| \rightarrow \infty}
    |d_\alpha| r^\alpha = 0 \qquad\text{for any $m \in \mathbb{N}$} \
    .
\end{equation*}
Hence
\begin{equation*}
    \lim_{\sum \alpha_i \leq m, |\alpha| \rightarrow \infty}
    |d_\alpha| = 0 \qquad\text{for any $m \in \mathbb{N}$} \
    .
\end{equation*}
In particular, this shows that $\mathcal{R}_K^{d,b} \subseteq
\widetilde{\mathcal{E}}$.

In a second step we suppose that $\sum_{\alpha \in \mathbb{Z}^d}
d_\alpha \mathbf{Z}^\alpha$ lies in $\widetilde{\mathcal{E}}_0$.
We claim that $\lim_{|\alpha| \rightarrow \infty} |d_\alpha|
r^\alpha = 0$ for any $0 < r < 1$. Let $\epsilon > 0$. We have
show that $|d_\alpha|r^\alpha < \epsilon$ for all but finitely
many $\alpha$. Choose $m \in \mathbb{N}$ in such a way that
$(\sup_\alpha |d_\alpha|)r^m < \epsilon$. Then certainly
$|d_\alpha|r^\alpha < \epsilon$ for any $\alpha$ such that $\sum
\alpha_i \geq m$. But by assumption there are only finitely many
nonzero $d_\alpha$ with $\sum \alpha_i \leq m$. This establishes
that $\widetilde{\mathcal{E}}_0 \subseteq \mathcal{R}_K^{d,b}$.

In a third step we argue that $\widetilde{\mathcal{E}}_0$ is dense
in $\widetilde{\mathcal{E}}$. Let $f = \sum_{\alpha \in
\mathbb{Z}^d} d_\alpha \mathbf{Z}^\alpha$ be an arbitrary element
in $\widetilde{\mathcal{E}}$. For any $\epsilon > 0$ the sets
\begin{equation*}
    N_0(\epsilon) := \{ \alpha : \sum \alpha_i \leq 0, |d_\alpha|
    \geq \epsilon\} \quad\text{and}\quad N_m(\epsilon) :=
    \{ \alpha : \sum \alpha_i = m, |d_\alpha| \geq \epsilon\}\
    \text{for $m \in \mathbb{N}$}
\end{equation*}
are finite. Hence $f_\epsilon := \sum_{\alpha \in N(\epsilon)}
d_\alpha \mathbf{Z}^\alpha$ with $N(\epsilon) := \bigcup_{m \geq
0} N_m(\epsilon)$ lies in $\widetilde{\mathcal{E}}_0$ and $\|f -
f_\epsilon\| < \epsilon$.
\end{proof}

Let $R^b(G,K)$ and $R^{int}(G,K)$ denote the image in $R(G,K)$ of
$\mathcal{R}_K^{d,b}$ and $\mathcal{R}_K^{d,int}$, respectively,
under the above isomorphism. By transport of structure we view
both subspaces as normed spaces with respect to $\|\ \|$.

\begin{lemma}\label{limit}
For any $e \in R(G,K)$ we have
\begin{equation*}
    \lim_{r < 1, r \rightarrow 1} |e|_{r,r} =
    \left\{
    \begin{aligned} \|e\| & \quad \text{if $e \in R^b(G,K)$,} \\
    \infty & \quad \text{otherwise.}
    \end{aligned}\right.
\end{equation*}
\end{lemma}
\begin{proof}
Let $0 < r_0 < 1$ be such that $e \in D_{[r_0,1)}(G,K)$. We
certainly may assume that $e \neq 0$. Then we may consider the
function
\begin{equation*}
    \phi(\rho) := \log (|e|_{\exp(\rho),\exp(\rho)})
\end{equation*}
on $[\rho_0,0)$ where $\rho_0 := \log(r_0)$. Let $e = \sum_{\alpha
\in \mathbb{Z}^d} d_\alpha \bb^\alpha$ and define $N := \{\alpha
\in \mathbb{Z}^d : d_\alpha \neq 0\}$. Then
\begin{equation*}
    \phi(\rho) = \max_{\alpha \in N} (\log(|d_\alpha|) + (\sum_i
    \alpha_i) \rho ) \ .
\end{equation*}
In particular, as a supremum of affine functions the function
$\phi$ is convex on $[\rho_0,0)$. If $e$ is not in $R^b(G,K)$ then
$\{\log(|d_\alpha |)\}_\alpha$ is unbounded which easily implies
that $\lim_{\rho < 0, \rho \rightarrow 0} \phi(\rho) = \infty$. On
the other hand, if $e \in R^b(G,K)$ then $\phi$ extends by
$\phi(0) := \log(\|e\|)$ to a convex function on $[\rho_0,0]$, and
we have
\begin{equation*}
    \log(|d_\alpha|) + (\sum_i \alpha_i)\rho \leq \phi(\rho) \leq
    \phi(0) + \frac{\phi(\rho_0) - \phi(0)}{\rho_0} \rho
\end{equation*}
for any $\alpha \in N$ and any $\rho \in [\rho_0,0]$. Let
\begin{equation*}
    M := \{\beta \in N : |d_\beta| = \max_\alpha |d_\alpha| =
    \psi(1) \} \ .
\end{equation*}
We have $\sum_i \beta_i \geq \frac{\phi(\rho_0) -
\phi(0)}{\rho_0}$ for any $\beta \in M$. Hence we may choose a
$\gamma \in M$ in such a way that $\sum_i \gamma_i$ is minimal.
Then
\begin{equation*}
    \log(|d_\beta|) + (\sum_i \beta_i)\rho \leq
    \log(|d_\gamma|) + (\sum_i \gamma_i)\rho
\end{equation*}
for any $\beta \in M$ and any $\rho \in [\rho_0,0]$. On the other
hand, if we put
\begin{equation*}
    c := \max \{ \log(|d_\alpha|) : \alpha \in N \setminus M \}
\end{equation*}
then
\begin{equation*}
    \log(|d_\alpha|) + (\sum_i \alpha_i)\rho \leq c +
    \frac{\phi(\rho_0) - c}{\rho_0} \rho
\end{equation*}
for any $\alpha \in N \setminus M$ and any $\rho \in [\rho_0,0]$.
Altogether we obtain
\begin{equation*}
    \phi(\rho) \leq \max ( c + \frac{\phi(\rho_0) - c}{\rho_0} \rho ,
    \log(|d_\gamma|) + (\sum_i \gamma_i)\rho )
\end{equation*}
for any $\rho \in [\rho_0,0]$. We certainly find a $\rho_0 <
\rho_1 < 0$ such that
\begin{equation*}
    c + \frac{\phi(\rho_0) - c}{\rho_0} \rho \leq
    \log(|d_\gamma|) + (\sum_i \gamma_i)\rho
\end{equation*}
for $\rho \in [\rho_1,0]$. We conclude that on $[\rho_1,0]$ the
function $\phi$ coincides with the affine function
$\log(|d_\gamma|) + (\sum_i \gamma_i)\rho$.
\end{proof}

\begin{proposition}\label{bounded}
$R^b(G,K)$ and $R^{int}(G,K)$ are subrings of $R(G,K)$; moreover,
the norm $\|\ \|$ is multiplicative in this ring multiplication.
\end{proposition}
\begin{proof}
For any $f_1 = \sum_{\beta \in \mathbb{Z}^d} c_\beta \bb^\beta$
and $f_2 = \sum_{\gamma \in \mathbb{Z}^d} d_\gamma \bb^\gamma$ in
$R(G,K)$ we have
\begin{align*}
    f_1f_2 & = (\sum_{\beta} c_\beta \bb^\beta) (\sum_{\gamma} d_\gamma
    \bb^\gamma) = \sum_{\beta,\gamma} c_\beta d_\gamma
    \bb^\beta\bb^\gamma \\
    & = \sum_{\beta,\gamma} c_\beta d_\gamma \sum_\alpha
    c_{\beta\gamma,\alpha} \bb^\alpha = \sum_\alpha
    (\sum_{\beta,\gamma} c_\beta d_\gamma
    c_{\beta\gamma,\alpha}) \bb^\alpha \ .
\end{align*}
If $f_1,f_2 \in R^b(G,K)$ we therefore, using Lemma
\ref{coefficients}, obtain
\begin{equation*}
    \|f_1f_2\| = \sup_\alpha |\sum_{\beta,\gamma} c_\beta d_\gamma
    c_{\beta\gamma,\alpha}| \leq \sup_\beta |c_\beta| \cdot
    \sup_\gamma |d_\gamma| = \|f_1\| \cdot \|f_2\| < \infty \ .
\end{equation*}
Hence $R^b(G,K)$ and $R^{int}(G,K)$ are subrings and $\|\ \|$ is
submultiplicative. But $\|\ \|$, according to Lemma \ref{limit}
and Lemma \ref{multiplicative}, is a limit of multiplicative norms
and therefore is, in fact, multiplicative.
\end{proof}

As a consequence of this latter proposition we may complete the
algebra $R^b(G,K)$ with respect to the norm $\|\ \|$ obtaining a
$K$-Banach algebra $E(G,K)$ with multiplicative norm $\|\ \|$. Of
course, as Banach spaces, we have our isometric isomorphism
\begin{equation*}
    \mathcal{E}_K^d \xrightarrow{\;\cong\;} E(G,K) \ .
\end{equation*}
It follows from Lemma \ref{limit} that the rings $R^b(G,K)$,
$R^{int}(G,K)$, and $E(G,K)$ together with their norm $\|\ \|$ are
independent of the ordering of our chosen basis $h_1,\ldots,h_d$
of $G$ (because this is the case for the norms $|\ |_{r,r}$ as we
have argued earlier).

The argument in the proof of Lemma \ref{coefficients} has another
interesting consequence. To formulate it we introduce the
following convention. Any of the rings $D_{[r_0,r]}(G,K)$,
$D_{[r_0,1)}(G,K)$, $R(G,K)$, and $E(G,K)$ has its natural
multiplication which in general is noncommutative and to which we
some times refer as the intrinsic multiplication (always written
as $(e,f) \longmapsto ef$). But using the bijection from Prop.\
\ref{qa} and its extensions these rings carry, by transport of
structure, also a commutative multiplication which we write as
$(e,f) \longmapsto e \circ f$. The notation $e^{-1}$ always will
refer to the inverse with respect to the intrinsic multiplication.

\begin{lemma}\label{comparison}
For any $e,f \in D_{[r_0,r]}(G,K)$ we have
\begin{equation*}
    |ef - e \circ f|_{r_0,r} < |ef|_{r_0,r} = |e \circ f|_{r_0,r} \
    .
\end{equation*}
\end{lemma}
\begin{proof}
Since $|\ |_{r_0,r} = \max(|\ |_{|r_0,r_0},|\ |_{r,r})$ it
suffices to treat the case $r_0 = r$. Then the norm is
multiplicative so that we have to show that
\begin{equation*}
    |ef - e \circ f|_{r,r} < |e|_{r,r} |f|_{r,r} \ .
\end{equation*}
Let $e = \sum_{\beta \in \mathbb{Z}^d} c_\beta \bb^\beta$ and $f =
\sum_{\gamma \in \mathbb{Z}^d} d_\gamma \bb^\gamma$. We have
\begin{equation*}
    ef = \sum_\alpha (\sum_{\beta,\gamma} c_\beta d_\gamma
    c_{\beta\gamma,\alpha}) \bb^\alpha \quad\text{and}\quad
    e \circ f = \sum_\alpha (\sum_{\beta + \gamma = \alpha}
    c_\beta d_\gamma) \bb^\alpha
\end{equation*}
and hence
\begin{equation*}
    |ef - e \circ f|_{r,r} = \max_\alpha | \sum_{\beta + \gamma \neq
    \alpha} c_\beta d_\gamma c_{\beta\gamma,\alpha} + \sum_{\beta + \gamma
    = \alpha} c_\beta d_\gamma (c_{\beta\gamma,\alpha} - 1)|r^\alpha
    \ .
\end{equation*}
From the proof of Lemma \ref{coefficients} we know that
\begin{equation*}
    |c_{\beta\gamma,\alpha}|r^\alpha < r^{\beta + \gamma}
    \quad\text{for $\beta + \gamma \neq \alpha$ and}\quad
    |c_{\beta\gamma,\beta + \gamma} - 1| < 1 \ .
\end{equation*}
We deduce that
\begin{align*}
    |ef - e \circ f|_{r,r} & < \max_{\beta,\gamma} |c_\beta d_\gamma|
    r^{\beta + \gamma} \\
    & \leq \max_\beta |c_\beta|r^\beta \cdot \max_\gamma
    |d_\gamma|r^\gamma \\
    & = |e|_{r,r} \cdot |f|_{r,r} \ .
\end{align*}
\end{proof}

\begin{proposition}\label{units}
Let $D$ be any of the rings $D_{[r_0,r]}(G,K)$,
$D_{[r_0,1)}(G,K)$, or $R(G,K)$; we then have:
\begin{itemize}
  \item[i.] $e \in D$ is a unit with respect to the intrinsic
  multiplication if and only if it is a unit with respect to the
  commutative multiplication;
  \item[ii.] any left or right unit in $D$ is a unit.
\end{itemize}
\end{proposition}
\begin{proof}
By limit arguments it suffices to consider the case $D =
D_{[r_0,r]}(G,K)$. Suppose first that $e$ is a commutative unit,
i.e., $e \circ f = 1$ for some $f \in D$. By Lemma
\ref{comparison} we then have $|ef -1|_{r_0,r} < 1$ which implies,
since we are in a Banach algebra, that $ef$ is an intrinsic unit.
Starting from $f \circ e = 1$ we similarly obtain that $fe$ is an
intrinsic unit. Hence $e$ and $f$ are intrinsic units. Now let,
vice versa, $e$ be a left intrinsic unit (the case of a right one
being analogous), say, $ef = 1$ for some $f \in D$. By a totally
analogous reasoning as above we obtain that $e \circ f = f \circ
e$ and hence $e$ and $f$ are commutative units. But then we
actually may apply the first part of the proof to conclude that
$e$ is an intrinsic unit.
\end{proof}

\begin{proposition}\label{b-units}
\begin{itemize}
  \item [i.] $R^b(G,K) \cap R(G,K)^\times = R^b(G,K)^\times$;
  \item [ii.] an element in $R^b(G,K)$ is a unit with respect to the intrinsic
  multiplication if and only if it is a unit with respect to the
  commutative multiplication.
\end{itemize}
\end{proposition}
\begin{proof}
i. Let $e \in R(G,K)^\times$ such that $\|e\| < \infty$. Suppose
that $e \in D_{[r_0,1)}(G,K)^\times$. By Lemma \ref{limit} and its
proof we know that $\lim_{r < 1, r \rightarrow 1} |e|_{r,r} =
\|e\|$ and that the function $\phi(\rho) := \log
(|e|_{\exp(\rho),\exp(\rho)})$ is an affine function on
$[\rho_1,0]$ for $\rho_1$ sufficiently close to $1$. Since the
norms $|\ |_{\exp(\rho),\exp(\rho)}$, by Lemma
\ref{multiplicative}, are multiplicative on $D_{[r_0,1)}(G,K)$ it
follows that
\begin{equation*}
    \lim_{\rho < 0, \rho \rightarrow 0}
    \log(|e^{-1}|_{\exp(\rho),\exp(\rho)}) =
    - \lim_{\rho < 0, \rho \rightarrow 0} \phi(\rho) = - \log(\|e\|)
    < \infty \ .
\end{equation*}
Hence, again by Lemma \ref{limit}, we have $e^{-1} \in R^b(G,K)$.
ii. This follows from Prop.\ \ref{units}.i by applying the present
assertion i. to $G$ as well as the commutative group
$\mathbb{Z}_p^d$.
\end{proof}

\begin{lemma}\label{specialunits}
We have $h_i^x - 1 \in R^{int}(G,K)^\times$ for any $x \in
\mathbb{Z}_p \setminus \{0\}$ and any $1 \leq i \leq d$.
\end{lemma}
\begin{proof}
Write $x = p^my$ with $y \in \mathbb{Z}_p^\times$. We know from
\eqref{3} that $(1 + Z)^y - 1 \in Z \cdot
\mathbb{Z}_p[[Z]]^\times$ and hence $(1 + Z)^x - 1 \in [(1 +
Z)^{p^m} - 1] \cdot \mathbb{Z}_p[[Z]]^\times$. Moreover the
leading coefficient of $(1 + Z)^{p^m} - 1$ is equal to $1$. Hence
\begin{equation}\label{f:unit}
    h_i^x - 1 \in (h_i^{p^m} - 1) \cdot R^{int}(G,K)^\times
    \qquad\textrm{and}\qquad \|h_i^x - 1\| = 1 \ .
\end{equation}
By Prop.\ \ref{b-units}.i and the multiplicativity of $\|\ \|$ it
now suffices to show that $h_i^{p^m} - 1$ is invertible in
$R(G,K)$. But the polynomial $(1 + Z)^{p^m} - 1$ has no zeros in
an appropriate annulus $r_m \leq |z| < 1$. This implies that
$h_i^{p^m} - 1 \in D_{[r_0,1)}(G,K)^\times$.
\end{proof}

The formula \eqref{f:unit} implies that the rings $R^{int}(G,K)
\subseteq R^b(G,K) \subseteq E(G,K)$ together with the norm $\|\
\|$ do not change if we replace the generating set $\{h_1, \ldots,
h_d\}$ by $\{h_1^{x_1}, \ldots, h_d^{x_d}\}$ for some $x_1,
\ldots, x_d \in \mathbb{Z}_p^\times$.

\end{document}